\documentclass[11pt]{amsart}

\usepackage{bm}
\usepackage{amssymb, amscd}
\usepackage{epsfig, mathtools, tensor}
\usepackage[vmargin=1in, hmargin=1.25in]{geometry}
\usepackage[font=small,format=plain,labelfont=bf,up,textfont=it,up]{caption}
\usepackage{microtype}
\usepackage[shortlabels]{enumitem}
\usepackage[backref=page, bookmarks, bookmarksdepth=2, colorlinks=true, linkcolor=blue, citecolor=blue, urlcolor=blue]{hyperref}
\usepackage{etoolbox}
\usepackage{tikz-cd}
\usepackage{tabularray}
\apptocmd{\thebibliography}{\raggedright}{}{}
\usepackage[only,llbracket,rrbracket,llparenthesis, rrparenthesis]{stmaryrd}
\usepackage{xcolor}

\usepackage[only,llparenthesis, rrparenthesis, llbracket,rrbracket]{stmaryrd}

\definecolor{colorblind_blue}{RGB}{0,114,178}
\definecolor{colorblind_orange}{RGB}{213,94,0}
\newcommand{\orange}[1]{\ensuremath{{\color{colorblind_orange} #1}}}
\newcommand{\blue}[1]{\ensuremath{{\color{colorblind_blue} #1}}}

\newcommand{\Cg}[1]{\ensuremath{\llparenthesis #1 \rrparenthesis}}
\newcommand{\CgCg}[2]{\ensuremath{\underline{#1 \otimes #2}}}
\newcommand{\Pres}[1]{\ensuremath{\llbracket #1 \rrbracket}}
\newcommand{\PresPrime}[1]{\ensuremath{\llbracket #1 \rrbracket'}}

\newcommand{\GrpR}[1]{\ensuremath{\{#1\}}}

\mathtoolsset{showonlyrefs}

\makeatletter
\patchcmd{\@maketitle}{\global\topskip42\p@\relax}
  {\global\topskip42\p@\relax \vspace*{-38pt}}
  {}{}
\makeatother

\setenumerate[0]{label=(\alph*)}

\renewcommand*{\backref}[1]{}
\renewcommand*{\backrefalt}[4]{%
    \ifcase #1 (Not cited.)%
    \or        (Cited on page~#2.)%
    \else      (Cited on pages~#2.)%
    \fi}

\newcommand{\arxiv}[1]{\href{http://arxiv.org/abs/#1}{{\tt arXiv:#1}}}

\newcommand*{\Cdot}[1][1.25]{%
  \mathpalette{\CdotAux{#1}}\cdot%
}
\newdimen\CdotAxis
\newcommand*{\CdotAux}[3]{%
  {%
    \settoheight\CdotAxis{$#2\vcenter{}$}%
    \sbox0{%
      \raisebox\CdotAxis{%
        \scalebox{#1}{%
          \raisebox{-\CdotAxis}{%
            $\mathsurround=0pt #2#3$%
          }%
        }%
      }%
    }%
    \dp0=0pt %
    \sbox2{$#2\bullet$}%
    \ifdim\ht2<\ht0 %
      \ht0=\ht2 %
    \fi
    \sbox2{$\mathsurround=0pt #2#3$}%
    \hbox to \wd2{\hss\usebox{0}\hss}%
  }%
}

\numberwithin{equation}{section}

\theoremstyle{plain}
\newtheorem{theorem}{Theorem}[section]
\newtheorem{maintheorem}{Theorem}

\newtheorem{proposition}[theorem]{Proposition}
\newtheorem{lemma}[theorem]{Lemma}

\newtheorem{corollary}[theorem]{Corollary}

\newtheorem*{unnumberedclaim}{Claim}

\newtheorem{stepx}{Step}
\newenvironment{step}[1]
 {\renewcommand\thestepx{#1}\stepx}
 {\endstepx}

\newtheorem{claimx}{Claim}
\newenvironment{claim}[1]
 {\renewcommand\theclaimx{#1}\claimx}
 {\endclaimx}

\theoremstyle{definition}
\newtheorem{asm}[theorem]{Assumption}
\newenvironment{assumption}[1][]{\begin{asm}[#1]\pushQED{\qed}}{\popQED \end{asm}}
\newtheorem{defn}[theorem]{Definition}
\newenvironment{definition}[1][]{\begin{defn}[#1]\pushQED{\qed}}{\popQED \end{defn}}
\newtheorem{notn}[theorem]{Notation}

\theoremstyle{remark}
\newtheorem{rmk}[theorem]{Remark}
\newenvironment{remark}[1][]{\begin{rmk}[#1] \pushQED{\qed}}{\popQED \end{rmk}}
\newtheorem{eg}[theorem]{Example}

\newtheorem{cvn}[theorem]{Convention}

\theoremstyle{plain}

\newcommand\Figure[1]{\centerline{\psfig{file=#1,scale=1}}}

\DeclareMathOperator{\coker}{coker}

\DeclareMathOperator{\Image}{Im}

\DeclareMathOperator{\Sp}{Sp}

\newcommand\Z{\ensuremath{\mathbb{Z}}}
\newcommand\Q{\ensuremath{\mathbb{Q}}}

\DeclareMathOperator{\HH}{H}

\DeclareMathOperator{\Interior}{Int}

\DeclareMathOperator{\Sym}{Sym}

\newcommand\Span[1]{\ensuremath{\langle #1 \rangle}}
\newcommand\Set[2]{\ensuremath{\left\{\text{#1 $|$ #2}\right\}}}
\newcommand\SpanSet[2]{\ensuremath{\langle \text{#1 $|$ #2} \rangle}}

\newcommand\cC{\ensuremath{\mathcal{C}}}

\newcommand\cK{\ensuremath{\mathcal{K}}}

\newcommand\cM{\ensuremath{\mathcal{M}}}

\newcommand\cQ{\ensuremath{\mathcal{Q}}}

\newcommand\cV{\ensuremath{\mathcal{V}}}
\newcommand\cW{\ensuremath{\mathcal{W}}}

\newcommand\fB{\ensuremath{\mathfrak{B}}}

\newcommand\fK{\ensuremath{\mathfrak{K}}}

\newcommand\fa{\ensuremath{\mathfrak{a}}}

\newcommand\fc{\ensuremath{\mathfrak{c}}}

\newcommand\fh{\ensuremath{\mathfrak{h}}}

\newcommand\fp{\ensuremath{\mathfrak{p}}}
\newcommand\fq{\ensuremath{\mathfrak{q}}}
\newcommand\hfq{\ensuremath{\widehat{\mathfrak{q}}}}
\newcommand\fr{\ensuremath{\mathfrak{r}}}

\newcommand\bV{\ensuremath{\mathbf{V}}}

\newcommand\bb{\ensuremath{\mathbf{b}}}

\newcommand\bd{\ensuremath{\mathbf{d}}}

\newcommand\bk{\ensuremath{\mathbf{k}}}

\newcommand\bpartial{\ensuremath{\bm{\partial}}}

\newcommand\tgamma{\ensuremath{\widetilde{\gamma}}}
\newcommand\tdelta{\ensuremath{\widetilde{\delta}}}

\newcommand\tSigma{\ensuremath{\widetilde{\Sigma}}}

\newcommand\tast{\ensuremath{\widetilde{\ast}}}

\newcommand\oQ{\ensuremath{\overline{Q}}}

\newcommand\oc{\ensuremath{\overline{c}}}

\newcommand\ox{\ensuremath{\overline{x}}}
\newcommand\oy{\ensuremath{\overline{y}}}
\newcommand\oz{\ensuremath{\overline{z}}}
\newcommand\osigma{\ensuremath{\overline{\sigma}}}
\newcommand\okappa{\ensuremath{\overline{\kappa}}}
\newcommand\ophi{\ensuremath{\overline{\phi}}}

\newcommand\ofr{\ensuremath{\overline{\fr}}}
\newcommand\obpartial{\ensuremath{\overline{\bpartial}}}

\newcommand\ogamma{\ensuremath{\overline{\gamma}}}

\newcommand\oalpha{\ensuremath{\overline{\alpha}}}
\newcommand\obeta{\ensuremath{\overline{\beta}}}
\newcommand\oeta{\ensuremath{\overline{\eta}}}

\newcommand\oomega{\ensuremath{\overline{\omega}}}
\newcommand\olambda{\ensuremath{\overline{\lambda}}}

\newcommand\precon[2]{\ensuremath{\prescript{#1}{}{#2}}}
\newcommand\Mod{\ensuremath{\operatorname{Mod}}}
\newcommand\Torelli{\ensuremath{\mathcal{I}}}

\title{Abelian covers of surfaces and the homology of the Torelli group}

\author{Daniel Minahan}
\address{Dept of Mathematics; University of Chicago; Chicago, IL 60637}
\email{dminahan@uchicago.edu}

\author{Andrew Putman}
\address{Dept of Mathematics; University of Notre Dame; 255 Hurley Hall; Notre Dame, IN 46556}
\email{andyp@nd.edu}

\thanks{AP was supported by NSF grant DMS-2305183.  DM was supported by NSF grant DMS-2402060.}

\begin{document}

\newpage

\begin{abstract}
We study the first homology group of the mapping class group and Torelli group with coefficients in the first rational
homology group of the universal abelian cover of the surface.  We prove two contrasting results: for
surfaces with one boundary component these twisted homology groups are finite-dimensional, but
for surfaces with one puncture they are infinite-dimensional.  These results play 
an important role in a recent paper of the authors calculating the second rational homology group
of the Torelli group.
\end{abstract}

\maketitle
\thispagestyle{empty}

\section{Introduction}

Let $\Sigma_{g,p}^b$ be an oriented genus $g$ surface with $p$ punctures 
and $b$ boundary components.\footnote{We omit $p$ or $b$ if they vanish.}  The
mapping class group $\Mod_{g,p}^b$ is the group
of isotopy classes of orientation-preserving diffeomorphisms of $\Sigma_{g,p}^b$ that fix each puncture
and boundary
component pointwise.  Assume\footnote{See \cite{PutmanCutPaste} for a discussion of the Torelli
group on surfaces with multiple boundary components.  Our main results are about
$\Mod_g^1$ and $\Mod_{g,1}$, so we do not need to go into this here.} 
that $p+b \leq 1$.  By Poincar\'{e} duality, the
intersection form on $\HH_1(\Sigma_{g,p}^b) \cong \Z^{2g}$ is a symplectic form.  The
action of $\Mod_{g,p}^b$ on $\HH_1(\Sigma_{g,p}^b)$ preserves this form, yielding
a surjection $\Mod_{g,p}^b \rightarrow \Sp_{2g}(\Z)$ whose kernel $\Torelli_{g,p}^b$
is the Torelli group.  This fits into an exact sequence
\[1 \longrightarrow \Torelli_{g,p}^b \longrightarrow \Mod_{g,p}^b \longrightarrow \Sp_{2g}(\Z) \longrightarrow 1.\]
In this paper, we study the first homology of $\Mod_{g,p}^b$ and $\Torelli_{g,p}^b$ with coefficients in the homology
of the universal abelian cover of $\Sigma_{g,p}^b$.

\subsection{Homology of mapping class group}

The mapping class group $\Mod_{g,p}^b$ is of type $F_{\infty}$.  In other words, it has a classifying space
whose $k$-skeleton is compact for all $k \geq 0$ (see, e.g., \cite{HarerDuality}).  This
implies that $\Mod_{g,p}^b$ is finitely presented and 
all of its homology groups are finitely generated.  In fact, for any finitely generated
$\Mod_{g,p}^b$-module $V$ the homology group $\HH_k(\Mod_{g,p}^b;V)$ is finitely generated.  These have
been calculated in many cases, at least in the ``stable range'' when $g \gg k$.  See, e.g., 
\cite{KawazumiSymplectic, LooijengaSymplectic, MadsenWeiss, ZhongStableLevel}.

\subsection{Low degree homology of Torelli}

We return to the case where $p+b \leq 1$.  Since $\Torelli_{g,p}^b$ is an
infinite-index subgroup of $\Mod_{g,p}^b$, it does not inherit any finiteness
properties.  In fact, it is known that many of its homology groups are
infinitely generated; see \cite{AkitaTorelli, BestvinaBuxMargalit, Gaifullin}.

However, it does have some unexpected finiteness properties.  Johnson \cite{JohnsonFinite}
proved that $\Torelli_{g,p}^b$ is finitely generated for $g \geq 3$.  He also calculated
its first homology group \cite{JohnsonKer, JohnsonAbel}.  Over $\Q$, this has
the following simple description: letting $H = \HH_1(\Sigma_{g,p}^b;\Q)$, we have
\[\HH_1(\Torelli_g^1;\Q) \cong \HH_1(\Torelli_{g,1};\Q) \cong \wedge^3 H \quad \text{and} \quad \HH_1(\Torelli_g;\Q) \cong (\wedge^3 H)/H.\]
The conjugation action of $\Mod_{g,p}^b$ on $\Torelli_{g,p}^b$
induces an action of $\Sp_{2g}(\Z)$ on each $\HH_d(\Torelli_{g,p}^b)$.  
The above isomorphisms are $\Sp_{2g}(\Z)$-equivariant.  They imply that
$\HH_1(\Sigma_{g,p}^b;\Q)$ is not just finite-dimensional, but is also
an algebraic\footnote{A representation $\bV$ of $\Sp_{2g}(\Z)$ over a field
$\bk$ of characteristic $0$ is algebraic if the action of $\Sp_{2g}(\Z)$ on $\bV$ extends to
a polynomial representation of
the $\bk$-points $\Sp_{2g}(\bk)$ of the algebraic group $\Sp_{2g}$.  Since
$\Sp_{2g}(\Z)$ is Zariski dense in $\Sp_{2g}(\bk)$, such an extension is unique if
it exists.} representation of $\Sp_{2g}(\Z)$.  

Verifying a long-standing folk conjecture, the authors \cite{MinahanPutmanH2} recently
calculated $\HH_2(\Torelli_{g,p}^b;\Q)$ for $g \geq 6$.  Like
the first homology, the second homology is also a finite-dimensional algebraic representation of
$\Sp_{2g}(\Z)$.  The proof in \cite{MinahanPutmanH2} required a result about the first homology
of $\Torelli_g^1$ with certain twisted coefficients that we prove
in the present paper (see Theorem \ref{maintheorem:findim} below). 

\begin{remark}
It is not known if the integral homology group $\HH_2(\Torelli_{g,p}^b)$ is finitely generated.
\end{remark}

\subsection{Action on fundamental group}

Fix basepoints $\ast$ on $\Sigma_g^1$ and $\Sigma_g$, with the basepoint for $\Sigma_g^1$ on $\partial \Sigma_g^1$.
Define
\[\pi_g^1 = \pi_1(\Sigma_g^1,\ast) \quad \text{and} \quad \pi_g = \pi_1(\Sigma_g,\ast).\]
By definition, elements of $\Mod_g^1$ fix $\partial \Sigma_g^1$ pointwise.  In particular, they
fix the basepoint $\ast \in \partial \Sigma_g^1$, so we get a well-defined action of $\Mod_g^1$
on $\pi_g^1$.  Also, we have $\Sigma_{g,1} \cong \Sigma_g \setminus \{\ast\}$, so
we can regard $\Mod_{g,1}$ as the group of isotopy classes of orientation-preserving diffeomorphisms
of $\Sigma_g$ that fix $\ast$.  We therefore get a well-defined action of $\Mod_{g,1}$ on $\pi_g$.

Define\footnote{The $\cC$ stands for ``commutator subgroup''.}
\[\cC_g^1 = \HH_1([\pi_g^1,\pi_g^1];\Q) \quad \text{and} \quad \cC_g = \HH_1([\pi_g,\pi_g];\Q).\]
Alternatively, $\cC_g^1$ and $\cC_g$ are the first rational homology
groups of the universal abelian covers of $\Sigma_g^1$ and $\Sigma_{g}$.
The action of $\Mod_g^1$ on $\pi_g^1$ preserves $[\pi_g^1,\pi_g^1]$, so $\Mod_g^1$ acts
on $\cC_g^1$.  Similarly, $\Mod_{g,1}$ acts on $\cC_g$.  
The vector spaces $\cC_g^1$ and $\cC_g$ are infinite-dimensional representations of
$\Mod_g^1$ and $\Mod_{g,1}$, and have been intensely studied
via the so-called ``Magnus representations''.  See \cite{SakasaiMagnusSurvey} for a survey.

\begin{remark}
Since they do not preserve basepoints, neither $\Mod_g$ nor $\Torelli_g$ act on $\cC_g$.
\end{remark}

\subsection{Main theorems}

Since $\cC_g^1$ and $\cC_g$ are infinite-dimensional, there is no reason to expect
that homology with these representations  as coefficients has any finiteness properties.  However,
we will prove:

\begin{maintheorem}
\label{maintheorem:findim}
For $g \geq 4$, both $\HH_1(\Mod_g^1;\cC_g^1)$ and $\HH_1(\Torelli_g^1;\cC_g^1)$ are finite-dimensional.
Moreover, $\HH_1(\Torelli_g^1;\cC_g^1)$ is an algebraic representation of $\Sp_{2g}(\Z)$.
\end{maintheorem}

\begin{remark}
\label{remark:notcalc}
Theorem \ref{maintheorem:findim} is what is needed for the authors' work
on the second homology group of the Torelli group in \cite{MinahanPutmanH2}.
\end{remark}

Typically decorations like boundary components or punctures have only a minor effect on
the mapping class group, so
Theorem \ref{maintheorem:findim} might lead the reader to expect that $\HH_1(\Mod_{g,1};\cC_g)$
and $\HH_1(\Torelli_{g,1};\cC_g)$ are also finite-dimensional.  However, to illustrate the subtlety
of Theorem \ref{maintheorem:findim} we will prove:

\begin{maintheorem}
\label{maintheorem:infdim}
For $g \geq 4$, both $\HH_1(\Mod_{g,1};\cC_g)$ and $\HH_1(\Torelli_{g,1};\cC_g)$ are infinite-dimensional.
\end{maintheorem}

\subsection{Linear congruence subgroups}

For $\ell \geq 2$, let $\Mod_{g,p}^b[\ell]$ be the level-$\ell$ congruence subgroup of $\Mod_{g,p}^b$, i.e.,
the kernel of the action of $\Mod_{g,p}^b$ on $\HH_1(\Sigma_{g,p}^b;\Z/\ell)$.  This is a sort of
``mod-$\ell$ Torelli group''.  In \cite{PutmanAbelian}, Putman proved versions of
Theorems \ref{maintheorem:findim} and \ref{maintheorem:infdim} for $\Mod_g^1[\ell]$ and
$\Mod_{g,1}[\ell]$, which we now describe.

The group $[\pi_g^1,\pi_g^1]$ is the kernel of the map $\pi_g^1 \rightarrow \HH_1(\pi_g^1)$, and similarly
for $[\pi_g,\pi_g]$.  This suggests defining
\[\pi_g^1[\ell] = \ker(\pi_g^1 \rightarrow \HH_1(\Sigma_g^1;\Z/\ell)) \quad \text{and} \quad
\pi_g[\ell] = \ker(\pi_g \rightarrow \HH_1(\Sigma_{g,1};\Z/\ell)).\]
We then set
$\cC_g^1[\ell] = \HH_1(\pi_g^1[\ell];\Q)$ and $\cC_g[\ell] = \HH_1(\pi_g[\ell];\Q)$.
Just like $\cC_g^1$ and $\cC_g$ are the first rational homology groups of the universal abelian
covers of $\Sigma_g^1$ and $\Sigma_g$, the vector spaces $\cC_g^1[\ell]$ and $\cC_g[\ell]$ are the first
rational homology groups of the universal mod-$\ell$ covers of these surfaces.

Since $\cC_g^1[\ell]$ and $\cC_g[\ell]$ are finite-dimensional, homology with coefficients in
them will also be finite-dimensional.  What Putman proved in \cite{PutmanAbelian} is
that for $g \geq 4$, we have
\begin{equation}
\label{eqn:leveliso}
\HH_1(\Mod_g^1[\ell];\cC_g^1[\ell]) \cong \HH_1(\Mod_g^1;\cC_g^1[\ell]) = \Q.
\end{equation}
Moreover, letting $H_{\Z/\ell} = \HH_1(\Sigma_g;\Z/\ell)$ and $\tau(\ell)$ be the number of
positive divisors of $\ell$, Putman also proved that for $g \geq 4$ we have
\[\HH_1(\Mod_{g,1}[\ell];\cC_g[\ell]) \cong \Q[H_{\Z/\ell}] \quad \text{and} \quad \HH_1(\Mod_{g,1};\cC_g[\ell]) \cong \Q^{\tau(\ell)}.\]
In particular, the dimensions of $\HH_1(\Mod_{g,1}[\ell];\cC_g[\ell])$ and $\HH_1(\Mod_{g,1};\cC_g[\ell])$
are much larger than those of $\HH_1(\Mod_g^1[\ell];\cC_g^1[\ell])$ and $\HH_1(\Mod_g^1;\cC_g^1[\ell])$.
These should be seen as analogues of Theorems \ref{maintheorem:findim} and \ref{maintheorem:infdim}, and
indeed our proofs of these results use some of the ideas from \cite{PutmanAbelian}.

\begin{remark}
The first isomorphism in \eqref{eqn:leveliso} was extended to higher homology groups
in \cite{PutmanStableLevel}, and these higher twisted homology groups were calculated
in \cite{ZhongStableLevel}.
\end{remark}

\subsection{Reidemeister pairing}

Much our work is devoted to understanding an equivariant intersection form on our
representations.
Let $\tSigma_g$ be the universal abelian cover of $\Sigma_g$.  Letting $H_{\Z} = \HH_1(\Sigma_g;\Z)$, this
is a regular $H_{\Z}$-cover of $\Sigma_g$ with $\cC_g = \HH_1(\tSigma_g;\Q)$.  Using the action of $H_{\Z}$
on $\HH_1(\tSigma_g;\Q)$, the algebraic intersection
pairing on $\cC_g = \HH_1(\tSigma_g;\Q)$ can be enriched to a pairing
\[\fr\colon \cC_g \otimes \cC_g \longrightarrow \Q[H_{\Z}]\]
called the Reidemeister pairing.  See \S \ref{section:reidemeister} for the details.

The group $\Mod_{g,1}$ acts on $\cC_g \otimes \cC_g$ and $\Q[H_{\Z}]$, and the Reidemeister pairing $\fr$ is $\Mod_{g,1}$-equivariant.
In particular, since $\Torelli_{g,1}$ acts trivially on $\Q[H_{\Z}]$ the map $\fr$ provides a large
trivial quotient of the $\Torelli_{g,1}$-representation $\cC_g \otimes \cC_g$.  It turns out that
$\fr$ is a connecting homomorphism in a long exact sequence arising when comparing $\HH_{\bullet}(\Torelli_g^1;\cC_g^1)$
and $\HH_{\bullet}(\Torelli_{g,1};\cC_g)$.  

Roughly speaking, the mechanism behind
the difference between $\HH_1(\Torelli_g^1;\cC_g^1)$ and $\HH_1(\Torelli_{g,1};\cC_g)$ is that
$\HH_1(\Torelli_{g,1};\cC_g)$ involves $\Image(\fr)$, while $\HH_1(\Torelli_g^1;\cC_g^1)$ involves
$\ker(\fr)$.  The main technical result that goes into our proofs is as follows.  Since $\Torelli_{g,1}$
acts trivially on $H_{\Z}$, the Reidemeister pairing factors through the $\Torelli_{g,1}$-coinvariants
of $\cC_g \otimes \cC_g$:
\[\ofr\colon (\cC_g \otimes \cC_g)_{\Torelli_{g,1}} \rightarrow \Q[H_{\Z}].\]
The action of $\Mod_{g,1}$ on $(\cC_g \otimes \cC_g)_{\Torelli_{g,1}}$ factors
through $\Sp_{2g}(\Z)$, and most of this paper is devoted to proving that $\ker(\ofr)$
is a finite-dimensional algebraic
representation of $\Sp_{2g}(\Z)$.
In fact, we identify this representation precisely: letting
$H = \HH_1(\Sigma_g;\Q)$, the vector space $\ker(\ofr)$ is the kernel of a contraction
\begin{equation}
\label{eqn:symcontractionintro}
\fc\colon \left((\wedge^2 H)/\Q\right)^{\otimes 2} \rightarrow \Sym^2(H).
\end{equation}
See \S \ref{section:kgalg}.  This same kernel appears in our work on the second homology of Torelli.

\subsection{Outline}
This paper has three parts.  In Part \ref{part:1}, we describe the Reidemeister
pairing and show how to reduce Theorems \ref{maintheorem:findim} and \ref{maintheorem:infdim}
to the description of its kernel we described above.  We then construct generators
for its kernel in Part \ref{part:2} and relations in Part \ref{part:3}.  

\subsection{Notation and conventions}
\label{section:conventions}

We will use all the notation discussed above.  In particular, $H$ will always
mean $\HH_1(\Sigma_{g,1};\Q) \cong \HH_1(\Sigma_g^1;\Q) \cong \HH_1(\Sigma_g;\Q)$
and $H_{\Z}$ will always mean $\HH_1(\Sigma_{g,1}) \cong \HH_1(\Sigma_g^1) \cong \HH_1(\Sigma_g)$.
This is a little ambiguous since the genus $g$ does not appear in $H$ or $H_{\Z}$, but the
correct $g$ will always be clear from the context.  As is traditional when writing about
the mapping class group, we will not distinguish between curves and their homotopy classes.
For instance, we will talk about simple closed curves in $\pi_g$.

\subsection{Standing assumption}
To avoid having to constantly impose hypotheses to rule out low-genus exceptions to our results,
we make the following standing assumption:

\begin{assumption}
\label{assumption:genus}
Throughout the paper, we assume that $g \geq 4$.
\end{assumption}

\subsection{Warning}
Though this is in some sense a sequel to \cite{PutmanAbelian}, 
our notation is often different from \cite{PutmanAbelian}.  For instance,
in \cite{PutmanAbelian} the notation $\Sigma_g^1$ means a surface with one puncture and $\Sigma_{g,1}$
a surface with one boundary component, while our convention is the opposite.\footnote{This better aligns our notation
with the algebro-geometric notation $\cM_{g,p}$ for the moduli space of smooth genus-$g$ curves with $p$ marked points.}
We also use various brackets like $\Pres{x,y}$ differently than \cite{PutmanAbelian}.
The paper \cite{PutmanAbelian} was
written while the second author was a postdoc, and he wishes to apologize for his youthful expository sins.

\part{The Reidemeister pairing and the homology of the Torelli group}
\label{part:1}

This part of the paper introduces the Reidemeister and reduces Theorems \ref{maintheorem:findim} and \ref{maintheorem:infdim}
to a description of its kernel.  It has four sections.  In \S \ref{section:birman} -- \ref{section:twocalc}, 
we discuss how $\Mod_g^1$ and $\Mod_{g,1}$ as well as $\cC_g^1$ and $\cC_g$ are related.  In \S \ref{section:reidemeister},
we give two different characterizations of the Reidemeister pairing.  
Finally, in \S \ref{section:mainproofs} we derive
our main theorems from the aforementioned description of the kernel of the Reidemeister pairing.

\section{Boundary components, punctures, and the Birman exact sequence}
\label{section:birman}

This section explores the inter-relationships between the different mapping class
groups and representations appearing in our theorems.

\subsection{Boundary component vs puncture: mapping class groups}

The only difference between $\Mod_g^1$ and $\Mod_{g,1}$ is the Dehn
twist\footnote{For a simple closed curve $\gamma$, our convention is that $T_{\gamma}$
denotes the left-handed Dehn twist about $\gamma$.} about the boundary component in $\Mod_g^1$:

\begin{lemma}[{\cite[Proposition 3.19]{FarbMargalitPrimer}}]
\label{lemma:capboundary}
We have a central extension
\[\begin{tikzcd}
1 \arrow{r} & \Z \arrow{r} & \Mod_g^1 \arrow{r} & \Mod_{g,1} \arrow{r} & 1,
\end{tikzcd}\]
where the map $\Mod_g^1 \rightarrow \Mod_{g,1}$ is induced by gluing a punctured
disk to $\partial = \partial \Sigma_g^1$ and the central $\Z$ is generated
by $T_{\partial}$.
\end{lemma}

Since the action of $\Mod_g^1$ on $H_{\Z} = \HH_1(\Sigma_g^1) \cong \HH_1(\Sigma_{g,1})$ factors
through $\Mod_{g,1}$, this lemma immediately implies a corresponding result for the Torelli group:

\begin{corollary}
\label{corollary:capboundarytorelli}
We have a central extension
\[\begin{tikzcd}
1 \arrow{r} & \Z \arrow{r} & \Torelli_g^1 \arrow{r} & \Torelli_{g,1} \arrow{r} & 1,
\end{tikzcd}\]
where the map $\Torelli_g^1 \rightarrow \Torelli_{g,1}$ is induced by gluing a punctured
disk to $\partial = \partial \Sigma_g^1$ and the central $\Z$ is generated
by $T_{\partial}$.
\end{corollary}

\subsection{Factoring the representation}

Using Lemma \ref{lemma:capboundary}, we prove:

\begin{lemma}
\label{lemma:factoraction}
The action of $\Mod_g^1$ on $\cC_g^1$ factors through $\Mod_{g,1}$.
\end{lemma}
\begin{proof}
Let $\partial = \partial \Sigma_g^1$.  By Lemma \ref{lemma:capboundary}, it is enough to prove
that $T_{\partial}$ acts trivially on $\cC_g^1$.  Letting $\delta \in \pi_g^1$ be the loop going
around $\partial$ with $\Sigma_g^1$ to its left, for $x \in \pi_g^1$ we have $T_{\partial}(x) = \delta^{-1} x \delta$:\\
\Figure{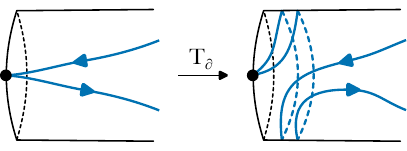}
Since $\delta \in [\pi_g^1,\pi_g^1]$, this implies that $T_{\partial}$ acts trivially on the abelianization
of $[\pi_g^1,\pi_g^1]$, and thus also acts trivially on $\cC_g^1$.
\end{proof}

Using this, we will henceforth view $\cC_g^1$ as a representation of both $\Mod_g^1$ and $\Mod_{g,1}$.
Via the map $\Mod_g^1 \rightarrow \Mod_{g,1}$, the group $\Mod_g^1$ acts on $\cC_g$, so $\cC_g$ is also
a representation of both $\Mod_g^1$ and $\Mod_{g,1}$.

\subsection{Boundary component vs puncture: representations}

The following lemma explains the relationship between the $\Mod_{g}^1$-representations
$\cC_g^1$ and $\cC_{g}$.  

\begin{lemma}
\label{lemma:capboundaryrep}
We have a short exact sequence
\[\begin{tikzcd}
0 \arrow{r} & \Q[H_{\Z}] \arrow{r} & \cC_g^1 \arrow{r} & \cC_{g} \arrow{r} & 0
\end{tikzcd}\]
of $\Mod_{g}^1$-representations.
\end{lemma}
\begin{proof}
Let $\tSigma_g^1$ and $\tSigma_g$ be the universal abelian covers of $\Sigma_g^1$ and $\Sigma_g$, respectively, so
$\cC_g^1 = \HH_1(\tSigma_g^1;\Q)$ and 
$\cC_g = \HH_1(\tSigma_g;\Q)$.
To construct $\tSigma_g$ from $\tSigma_g^1$, we glue disks to all components of $\partial \tSigma_g^1$.
Let $K$ be the image of $\HH_1(\partial \tSigma_g^1;\Q)$ in $\cC_g^1 = \HH_1(\tSigma_g^1;\Q)$,
so we have a short exact sequence
\[\begin{tikzcd}
0 \arrow{r} & K \arrow{r} & \cC_g^1 \arrow{r} & \cC_{g} \arrow{r} & 0.
\end{tikzcd}\]
We must prove that $K \cong \Q[H_{\Z}]$.  Let $\partial_0$ be a
component of $\partial \tSigma_g^1$, oriented such that $\tSigma_g^1$ is to its left.  The deck group $H_{\Z}$ of the cover $\tSigma_g^1$ acts simply transitively on
its boundary components.  Since $\tSigma_g^1$ is not compact, the
homology classes of its boundary components are linearly independent. 
The map $\Q[H_{\Z}] \rightarrow K$ taking $h \in H_{\Z}$ to the homology class of $h \Cdot \partial_0$
is thus an isomorphism, as desired.
\end{proof}

\begin{remark}
\label{remark:capboundaryrep}
For later use, we make the above injection $\Q[H_{\Z}] \hookrightarrow \cC_g^1$ explicit as follows.
For $z \in [\pi_g^1,\pi_g^1]$, let $\Cg{z}$ denote the corresponding element of $\cC_g^1 = \HH_1([\pi_g^1,\pi_g^1];\Q)$.  Also, for
$x \in \pi_g^1$ let $\ox$ denote the corresponding element of $H_{\Z}$.  Let
$\delta \in [\pi_g^1,\pi_g^1]$ be the following curve:\\
\Figure{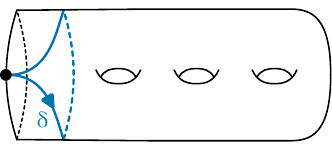}
Then for $x \in \pi_g^1$, the element of $\cC_g^1$ corresponding to $\ox \in H_{\Z}$ is $\Cg{x \delta x^{-1}}$.
\end{remark}

\subsection{Birman exact sequence}
\label{section:birmanexactseq}

The effect on the mapping class group of adding a puncture is described by the Birman exact sequence:

\begin{theorem}[\cite{BirmanThesis}]
\label{theorem:birman}
Let $p,b \geq 0$ and let $x_0 \in \Sigma_{g,p}^b$ be the location of the $(p+1)^{\text{st}}$
puncture in $\Sigma_{g,p+1}^b$.  We then have a short exact sequence
\[\begin{tikzcd}
1 \arrow{r} & \pi_1(\Sigma_{g,p}^b,x_0) \arrow{r} & \Mod_{g,p+1}^b \arrow{r} & \Mod_{g,p}^b \arrow{r} & 1,
\end{tikzcd}\]
where the map $\Mod_{g,p+1}^b \rightarrow \Mod_{g,p}^b$ fills in the $(p+1)^{\text{st}}$ puncture.
\end{theorem}

The normal subgroup $\pi_{g,p}^b = \pi_1(\Sigma_{g,p}^b,x_0)$ of $\Mod_{g,p+1}^b$ is called the {\em point-pushing subgroup}.
Using the action of $\Mod_{g,p+1}^b$ on $\pi_{g,p}^b$, the point-pushing subgroup can be identified with the group
of inner automorphisms of $\pi_{g,p}^b$ in the sense that if $\gamma \in \pi_{g,p}^b$ and $f_{\gamma} \in \Mod_{g,p+1}^b$ 
is the associated mapping class, then\footnote{The reader might expect to see $f_{\gamma}(x) = \gamma x \gamma^{-1}$ since
this is what would be needed for $f_{\gamma_1 \gamma_2} = f_{\gamma_1} \circ f_{\gamma_2}$ to hold.  This points to a tiny annoying issue:
in $\Mod_{g,p+1}^b$, elements are composed right to left like functions, but in $\pi_{g,p}^b$ paths are composed left to right.
Strictly speaking, therefore, the map taking $\gamma \in \pi_{g,p}^b$ to $f_{\gamma} \in \Mod_{g,p+1}^b$ is not an isomorphism
but an anti-isomorphism.  We could avoid this by changing our conventions about either
functions or fundamental groups, but this would lead to endless confusion.  This does 
not affect any of our calculations, and we will not mention it again.}
\begin{equation}
\label{eqn:innerauto}
f_{\gamma}(x) = \gamma^{-1} x \gamma \quad \text{for all $x \in \pi_{g,p}^b$}.
\end{equation}
This implies in particular that the point-pushing subgroup acts trivially on $\HH_1(\Sigma_{g,p}^b)$.

The Torelli group $\Torelli_{g,p+1}^b$ is only defined thus far for $(p+1)+b \leq 1$, so 
as far as Torelli is concerned the above is only relevant for $p = b = 0$.  In that case,
since $\HH_1(\Sigma_{g,1}) \cong \HH_1(\Sigma_g)$ the point-pushing subgroup
of $\Mod_{g,1}$ acts trivially on $\HH_1(\Sigma_{g,1})$ and
thus lies in $\Torelli_{g,1}$.  Even more is true: the action of $\Mod_{g,1}$ on
$\HH_1(\Sigma_{g,1}) \cong \HH_1(\Sigma_g)$ factors through $\Mod_g$, so:

\begin{corollary}
\label{corollary:birmantorelli}
The Birman exact sequence restricts to a short exact sequence
\[\begin{tikzcd}
1 \arrow{r} & \pi_g \arrow{r} & \Torelli_{g,1} \arrow{r} & \Torelli_g \arrow{r} & 1.
\end{tikzcd}\]
\end{corollary}

\subsection{Coinvariants}
\label{section:cgpicoinv}

We will henceforth view $\cC_g^1$ and $\cC_g$ as representations of $\pi_g$ by embedding
$\pi_g$ into $\Mod_{g,1}$ as the point-pushing subgroup and using the actions of $\Mod_{g,1}$ on $\cC_g^1$ and $\cC_g$.
By \eqref{eqn:innerauto}, this action of $\pi_g$ on $\cC_{g} = \HH_1([\pi_g,\pi_g];\Q)$
is the action induced by the conjugation action of $\pi_g$ on $[\pi_g,\pi_g]$.
However, the action of $\pi_g$ on $\cC_g^1$ is not as easy to describe.

As a first calculation, we determine the coinvariants $(\cC_g)_{\pi_g}$.
We can view the algebraic intersection form on $H$ as an $\Sp_{2g}(\Q)$-invariant
element $\omega \in \wedge^2 H$.  If $\{a_1,b_1,\ldots,a_g,b_g\}$ is a symplectic basis
for $H$, then
\[\omega = a_1 \wedge b_1 + \cdots + a_g \wedge b_g.\]
Henceforth we will view $\Q$ as the trivial subrepresentation of $\wedge^2 H$ spanned
by $\omega$, which allows us to talk about $(\wedge^2 H)/\Q$.  We then have:

\begin{lemma}
\label{lemma:cgpicoinv}
Letting the notation be as above, we have $(\cC_g)_{\pi_g} \cong (\wedge^2 H)/\Q$.
\end{lemma}
\begin{proof}
Using \eqref{eqn:innerauto}, we have
\[(\cC_g)_{\pi_g} = \HH_1([\pi_g,\pi_g];\Q)_{\pi_g} \cong \frac{[\pi_g,\pi_g]}{[\pi_g,[\pi_g,\pi_g]]} \otimes \Q.\]
That this is $(\wedge^2 H)/\Q$ is now classical.\footnote{In fact, even before tensoring
with $\Q$ this is isomorphic to $(\wedge^2 H_{\Z})/\Z$.}  See, for instance, \cite[Theorem D]{PutmanCommutator}.
\end{proof}

\begin{remark}
Since $\Mod_{g,1}/\pi_g \cong \Mod_g$, the coinvariants $(\cC_g)_{\pi_g}$ are
a representation of $\Mod_g$.  The isomorphism in Lemma \ref{lemma:cgpicoinv}
is an isomorphism of $\Mod_g$-representations, where $\Mod_g$ acts on $(\wedge^2 H)/\Q$
via $\Sp_{2g}(\Z)$.
\end{remark}

\section{Some deeper relations between punctures and boundary components}
\label{section:twocalc}

We continue our study of the interplay between punctures and boundary components.

\subsection{Splitting Birman exact sequence}

Consider the Birman exact sequence for $\Mod_{g,1}^1$:
\[\begin{tikzcd}
1 \arrow{r} & \pi_g^1 \arrow{r} & \Mod_{g,1}^1 \arrow{r} & \Mod_g^1 \arrow{r} & 1.
\end{tikzcd}\]
This splits via the map $\Mod_g^1 \rightarrow \Mod_{g,1}^1$ induced by
embedding $\Sigma_g^1$ as a subsurface of $\Sigma_{g,1}^1$ and extending mapping
classes by the identity; see here:\\
\Figure{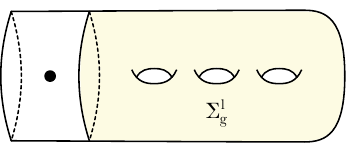}
This leads to a semidirect product decomposition
$\Mod_{g,1}^1 = \pi_g^1 \rtimes \Mod_g^1$, and in particular an action of $\Mod_g^1$ on $\pi_g^1$.  Regarding
the puncture of $\Sigma_{g,1}^1$ as a marked point, we can deformation retract $\Sigma_{g,1}^1$ onto
$\Sigma_g^1$ such that the marked point ends up as a basepoint on $\partial \Sigma_g^1$:\\
\Figure{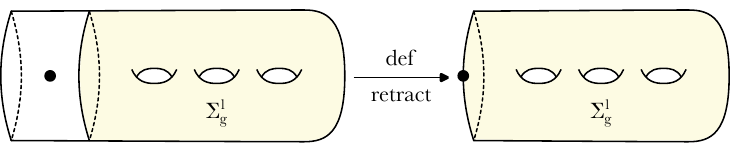}
From this, we see that our action of $\Mod_g^1$ on $\pi_g^1$ is the natural one arising from placing
the basepoint of $\pi_g^1 = \pi_1(\Sigma_g^1)$ on $\partial \Sigma_g^1$.

\subsection{Extending Torelli}

For this section only, we need a version of the Torelli group on $\Sigma_{g,1}^1$.  View the puncture
on $\Sigma_{g,1}^1$ as a marked point $x_0 \in \Sigma_g^1$.  Let $y_0 \in \partial \Sigma_g^1$ be another point.
Let $S \cong \Sigma_g^1$ be a subsurface of $\Sigma_{g}^1$ such that $x_0,\partial \Sigma_g^1 \subset \Sigma_{g,1}^1 \setminus S$
and let $\delta \in \HH_1(\Sigma_{g}^1,\{x_0,y_0\})$ be the homology class of
an arc connecting $x_0$ to $y_0$ in $\Sigma_{g,1}^1 \setminus S$:\\
\Figure{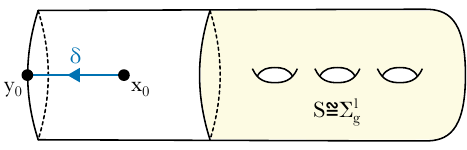}
We have
\[\HH_1(\Sigma_{g}^1,\{x_0,y_0\}) = \Z\Span{\delta} \oplus \HH_1(S) \cong \Z \oplus \Z^{2g}.\]
Let $\Torelli_{g,1}^1$ be the kernel of the action of $\Mod_{g,1}^1$ on $\HH_1(\Sigma_g^1,\{x_0,y_0\})$.
The point-pushing subgroup $\pi_g^1$ of $\Mod_{g,1}^1$ does not lie in $\Torelli_{g,1}^1$
since it does not fix $\delta$.  In fact:
 
\begin{theorem}[{\cite[Theorem 1.2]{PutmanCutPaste}}]
\label{theorem:birmantorelli2}
The intersection of the point-pushing subgroup $\pi_g^1 < \Mod_{g,1}^1$ with $\Torelli_{g,1}^1$ is
$[\pi_g^1,\pi_g^1]$, and the Birman exact sequence restricts to a split
exact sequence
\[\begin{tikzcd}
1 \arrow{r} & {[\pi_g^1,\pi_g^1]} \arrow{r} & \Torelli_{g,1}^1 \arrow{r} & \Torelli_g^1 \arrow{r} & 1.
\end{tikzcd}\]
\end{theorem}
    
\begin{remark} 
\label{remark:puncturesboundary}
The results in \cite{PutmanCutPaste} are for surfaces with multiple boundary components,
not boundary components and punctures.  However, the proofs work verbatim to prove
the results we discuss above.  In fact, even more is true: letting $\{\partial_1,\partial_2\}$
be the boundary components of $\Sigma_g^2$, the group $\Torelli_{g,1}^1$ we discussed above
is {\em isomorphic} to the group $\Torelli(\Sigma_g^2,\{\{\partial_1,\partial_2\}\})$ from
\cite{PutmanCutPaste}.  The isomorphism glues a punctured disk to $\partial_1$; since
no power of $T_{\partial_1}$ lies in $\Torelli(\Sigma_g^2,\{\{\partial_1,\partial_2\}\})$,
this does not change the Torelli group (c.f.\ Lemma \ref{lemma:capboundary}).
\end{remark}

The conjugation action of $\Mod_{g,1}^1$ on
$\Torelli_{g,1}^1$ induces an action of $\Mod_{g,1}^1/\Torelli_{g,1}^1$ on each $\HH_d(\Torelli_{g,1}^1)$.
Using the Birman exact sequences for $\Mod_{g,1}^1$ and $\Torelli_{g,1}^1$, we get an exact
sequence
\[\begin{tikzcd}[row sep=small]
1 \arrow{r} & \pi_g^1/{[\pi_g^1,\pi_g^1]} \arrow[equals]{d} \arrow{r} & \Mod_{g,1}^1/\Torelli_{g,1}^1 \arrow{r} & \Mod_g^1/\Torelli_g^1 \arrow[equals]{d} \arrow{r} & 1. \\
            & H_{\Z}                                                  &                                         & \Sp_{2g}(\Z)                                      &
\end{tikzcd}\]
The map $\Mod_g^1 \rightarrow \Mod_{g,1}^1$ that splits the Birman exact sequence induces a splitting of this exact sequence,
giving an inclusion $\Sp_{2g}(\Z) \hookrightarrow \Mod_{g,1}^1/\Torelli_{g,1}^1$ and thus an action
of $\Sp_{2g}(\Z)$ on each $\HH_d(\Torelli_{g,1}^1)$.  Putman \cite{PutmanJohnson} calculated
$\HH_1(\Torelli_{g,1}^1;\Q)$.  His calculation implies:\footnote{Like Johnson's work on the abelianization of Torelli, this
theorem requires $g \geq 3$; see Assumption \ref{assumption:genus}.}

\begin{theorem}[{\cite[Theorem B]{PutmanJohnson}}]
\label{theorem:abeltorelli}
The homology group $\HH_1(\Torelli_{g,1}^1;\Q)$ is a finite-dimensional
algebraic representation of $\Sp_{2g}(\Z)$.
\end{theorem}

\begin{remark}
\label{remark:calch1}
The same caveat about punctures vs boundary components from Remark \ref{remark:puncturesboundary} also
applies to \cite{PutmanJohnson}.
\end{remark}

\subsection{Capping boundary of groups}

Since $\Mod_{g,1} / \Torelli_{g,1} \cong \Sp_{2g}(\Z)$, the coinvariants $(\cC_g^1)_{\Torelli_{g,1}}$
are a representation of $\Sp_{2g}(\Z)$.  We have:

\begin{lemma}
\label{lemma:cg1torellicoinv}
We have $(\cC_g^1)_{\Torelli_{g,1}} \cong \wedge^2 H$.
\end{lemma}
\begin{proof}
Consider the split exact sequence from Theorem \ref{theorem:birmantorelli2}:
\[\begin{tikzcd}
1 \arrow{r} & {[\pi_g^1,\pi_g^1]} \arrow{r} & \Torelli_{g,1}^1 \arrow{r} & \Torelli_g^1 \arrow{r} & 1.
\end{tikzcd}\]
The above short exact sequence yields a five term exact sequence in group homology.  Since our short exact sequence splits, the rightmost three terms of the associated five term exact sequence are in fact a short exact sequence
\[\begin{tikzcd}[row sep=small]
0 \arrow{r} & {\left(\HH_1([\pi_g^1,\pi_g^1];\Q)\right)_{\Torelli_g^1}} \arrow[equals]{d} \arrow{r} & \HH_1(\Torelli_{g,1}^1;\Q) \arrow{r} & \HH_1(\Torelli_g^1;\Q) \arrow{r} & 0 \\
            & (\cC_g^1)_{\Torelli_{g,1}}                                                         &                                      &                                  &
\end{tikzcd}\]
The result now follows from the calculation of $\HH_1(\Torelli_{g,1}^1;\Q)$ from \cite{PutmanJohnson}.
\end{proof}

This has the following corollary:

\begin{corollary}
\label{corollary:capboundaryseq}
We have a short exact sequence of $\Sp_{2g}(\Z)$-representations
\[\begin{tikzcd}
0 \arrow{r} & V \arrow{r} & \HH_1(\Torelli_g^1;\cC_g^1) \arrow{r} & \HH_1(\Torelli_{g,1};\cC_g^1) \arrow{r} & 0
\end{tikzcd}\]
with $V$ a finite-dimensional algebraic representation of $\Sp_{2g}(\Z)$.
\end{corollary}

For the proof of Corollary \ref{corollary:capboundaryseq} and many future proofs as well, we need:
\begin{align*}
\tag{$\spadesuit$}\label{eqn:algclosed} &\text{\space\space\space the collection of finite-dimensional algebraic representations of $\Sp_{2g}(\Z)$ is closed}\\ &\text{\space\space\space under subquotients, extensions, and tensor products.}
\end{align*}

\begin{proof}[Proof of Corollary \ref{corollary:capboundaryseq}]
Consider the central extension from Corollary \ref{corollary:capboundarytorelli}:
\[\begin{tikzcd}
1 \arrow{r} & \Z \arrow{r} & \Torelli_g^1 \arrow{r} & \Torelli_{g,1} \arrow{r} & 1
\end{tikzcd}\]
Let $\partial = \partial \Sigma_g^1$.  The central $\Z$ is generated by $T_{\partial}$, which
acts trivially on $\cC_g^1$.  It follows that the associated 5-term exact sequence in group
homology contains 
\[\begin{tikzcd}[row sep=small]
\HH_1(\Z;\cC_g^1)_{\Torelli_{g,1}} \arrow{r} \arrow[equals]{d} & \HH_1(\Torelli_g^1;\cC_g^1) \arrow{r} & \HH_1(\Torelli_{g,1};\cC_g^1) \arrow{r} & 0 \\
\left(\cC_g^1\right)_{\Torelli_{g,1}}
\end{tikzcd}\]
Lemma \ref{lemma:cg1torellicoinv} implies that $(\cC_g^1)_{\Torelli_{g,1}}$
is a finite-dimensional algebraic representation of
$\Sp_{2g}(\Z)$.  An application of \eqref{eqn:algclosed} now proves the corollary.
\end{proof}

\subsection{Capping boundary of representations}

Corollary \ref{corollary:capboundaryseq} shows how to go from $\Torelli_g^1$ to $\Torelli_{g,1}$.  We now
show how to go from $\cC_g^1$ to $\cC_g$:

\begin{lemma}
\label{lemma:capboundaryrep2}
We have an exact sequence
\[\begin{tikzcd}
0 \arrow{r} & W \arrow{r} & \HH_1(\Torelli_{g}^1;\cC_g^1) \arrow{r} & \HH_1(\Torelli_{g}^1;\cC_g)
\end{tikzcd}\]
with $W$ a finite-dimensional algebraic representation of $\Sp_{2g}(\Z)$.
\end{lemma}
\begin{proof}
From the long exact sequence in group homology associated to the short exact sequence
\[\begin{tikzcd}
0 \arrow{r} & \Q[H_{\Z}] \arrow{r} & \cC_g^1 \arrow{r} & \cC_{g} \arrow{r} & 0
\end{tikzcd}\]
of $\Torelli_g^1$-representations from Lemma \ref{lemma:capboundaryrep}, it is enough to prove that the image $W$ of the map
$\HH_1(\Torelli_{g}^1;\Q[H_{\Z}]) \rightarrow \HH_1(\Torelli_{g}^1;\cC_g^1)$
is a finite-dimensional algebraic representation of $\Sp_{2g}(\Z)$.

Now consider the split exact sequence from Theorem \ref{theorem:birmantorelli2}:
\begin{equation}
\label{eqn:splitbirman2}
\begin{tikzcd}
1 \arrow{r} & {[\pi_g^1,\pi_g^1]} \arrow{r} & \Torelli_{g,1}^1 \arrow{r} & \Torelli_g^1 \arrow{r} & 1.
\end{tikzcd}
\end{equation}
Associated to \eqref{eqn:splitbirman2} is a Hochschild--Serre spectral sequence with the following properties:
\begin{itemize}
\item Since $[\pi_g^1,\pi_g^1]$ is a free group, we have $\HH_q([\pi_g^1,\pi_g^1];\Q) = 0$ for $q \geq 2$ and
thus our spectral sequence only has two nonzero rows.
\item Since \eqref{eqn:splitbirman2} splits, all the differentials coming out of the bottom row of our
spectral sequence vanish.
\end{itemize}
Combining these two facts, our spectral sequence collapses on page $2$.  It therefore breaks up into
short exact sequences, one of which is
\[\begin{tikzcd}[row sep=small]
0 \arrow{r} & \HH_1(\Torelli_g^1;\HH_1([\pi_g^1,\pi_g^1];\Q)) \arrow[equals]{d} \arrow{r} & \HH_2(\Torelli_{g,1}^1;\Q) \arrow{r} & \HH_2(\Torelli_g^1;\Q) \arrow{r} & 0. \\
            & \HH_1(\Torelli_g^1;\cC_g^1)                                                 &                                      &                                  &
\end{tikzcd}\]
From this, we get an injection $\HH_1(\Torelli_{g}^1;\cC_g^1) \hookrightarrow \HH_2(\Torelli_{g,1}^1;\Q)$.  Let
$\Phi$ be the composition
\[\begin{tikzcd}
\HH_1(\Torelli_g^1;\Q) \otimes \Q[H_{\Z}] = \HH_1(\Torelli_g^1;\Q[H_{\Z}]) \arrow{r} & \HH_1(\Torelli_{g}^1;\cC_g^1) \arrow[hookrightarrow]{r} & \HH_2(\Torelli_{g,1}^1;\Q).
\end{tikzcd}\]
To prove that $W$ is a finite-dimensional algebraic representation of $\Sp_{2g}(\Z)$, it is 
enough to prove the same statement for $\Image(\Phi)$.

We introduce some notation:
\begin{itemize}
\item For subgroups $G_1,G_2 < \Torelli_{g,1}^1$ with $[G_1,G_2] = 1$, let
$\HH_1(G_1;\Q) \otimes \HH_1(G_2;\Q)$ denote the corresponding subgroup of $\HH_2(G_1 \times G_2;\Q)$ and let
$\overline{\HH_1(G_1;\Q) \otimes \HH_1(G_2;\Q)}$ be its image under the map
$\HH_2(G_1 \times G_2;\Q) \rightarrow \HH_2(\Torelli_{g,1}^1;\Q)$.
\item For a subgroup $G < \Torelli_{g,1}^1$ and an element $x \in \Torelli_{g,1}^1$ with $[G,x] = 1$, let
$[x] \in \HH_1(\Span{x};\Q)$ be the corresponding element, let $\HH_1(G;\Q) \otimes [x]$ be the corresponding
subgroup of $\HH_2(G \times \Span{x};\Q)$, and let $\overline{\HH_1(G;\Q) \otimes [x]}$ be its image
in $\overline{\HH_1(G;\Q) \otimes \HH_1(\Span{x};\Q)}$.
\item For $x,y \in \Torelli_g^1$ with $[x,y] = 1$, let $[x] \otimes [y] \in \HH_1(\Span{x} \times \Span{y};\Q)$
be the corresponding element and let $\overline{[x] \otimes [y]}$ be its image in
$\overline{\HH_1(\Span{x};\Q) \otimes \HH_1(\Span{y};\Q)}$.
\end{itemize}
Now let $\partial = \partial \Sigma_{g,1}^1$, so $T_{\partial} \in \Torelli_{g,1}^1$.  Since $T_{\partial}$
is a central element, we can define
\[U = \overline{[T_{\partial}] \otimes \HH_1(\Torelli_{g,1}^1;\Q)} \subset \HH_2(\Torelli_{g,1}^1;\Q).\]
This is a quotient of $\HH_1(\Torelli_{g,1}^1;\Q)$, which by Theorem \ref{theorem:abeltorelli} is a finite-dimensional
algebraic representation of $\Sp_{2g}(\Z)$.  By \eqref{eqn:algclosed}, we have that $U$ is also a finite-dimensional
algebraic representation of $\Sp_{2g}(\Z)$.  
Again using \eqref{eqn:algclosed}, to prove that $\Image(\Phi)$
a finite-dimensional algebraic representation of $\Sp_{2g}(\Z)$, it is enough to show:

\begin{unnumberedclaim}
$\Image(\Phi) < U$.
\end{unnumberedclaim}

Consider $f \in \Torelli_g^1$ and $h \in H_{\Z}$.  Let $[f] \in \HH_1(\Torelli_g^1;\Q)$ be the corresponding element.
We must prove that $\Phi([f] \otimes h) \in U$.
The conjugation action of $\Mod_{g,1}^1$ induces an action of $\Mod_{g,1}^1$ on $\HH_2(\Torelli_{g,1}^1;\Q)$.  Since
$T_{\partial}$ is central, this action preserves $U$.
The group $\Mod_{g,1}^1$ also acts on 
\[\Q[H_{\Z}] < \cC_g^1 = \HH_1([\pi_g^1,\pi_g^1];\Q)\]
via its action on $[\pi_g^1,\pi_g^1]$ and on 
$\HH_1(\Torelli_{g}^1;\Q)$ via its projection $\Mod_{g,1}^1 \rightarrow \Mod_{g}^1$.  With
respect to these actions, the map
\[\Phi\colon \HH_1(\Torelli_g^1;\Q) \otimes \Q[H_{\Z}] \rightarrow \HH_2(\Torelli_{g,1}^1;\Q)\]
is $\Mod_{g,1}^1$-equivariant.  To check that $\Phi([f] \otimes h) \in U$,
we can therefore first apply any element of $\Mod_{g,1}^1$ we wish to $[f] \otimes h$.

The point-pushing subgroup $\pi_g^1 < \Mod_{g,1}^1$ acts on $\Q[H_{\Z}]$ via its projection
$\pi_g^1 \rightarrow H_{\Z}$.  Since this projection is surjective, we see that $\pi_g^1$ acts
transitively on $H_{\Z}$.  In light of Remark \ref{remark:capboundaryrep}, applying an appropriate element of $\Mod_{g,1}^1$ we can thus assume that
$h \in \cC_g^1$ is the homology class of the element $\delta \in [\pi_g^1,\pi_g^1]$ shown here:\\
\Figure{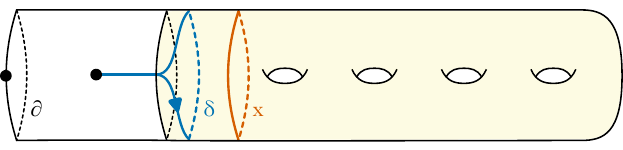}
Here the subsurface $\Sigma_g^1$ is shaded and $x$ is a loop parallel to $\partial \Sigma_g^1$.  
The loops $\delta$ and $x$ are homotopic to
the following configuration, which illustrates the element of the point-pushing subgroup of $\Torelli_{g,1}^1$ corresponding to
$\delta$:\\
\Figure{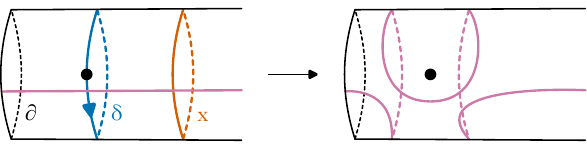}
As is clear from this figure, this element is $T_{x} T_{\partial}^{-1}$.
Since $T_x$ is in the center of $\Mod_g^1$, 
embedding $\Torelli_g^1$ into $\Torelli_{g,1}^1$ using the subsurface inclusion we indicated
above identifies $f \in \Torelli_g^1$ with an element of $\Torelli_{g,1}^1$ that commutes with both
$T_x$ and $T_{\partial}$.  Tracing through the definitions of all of our maps, we have
\[\Phi([f] \otimes h) = \overline{[f] \otimes [T_{x} T_{\partial}^{-1}]} = \overline{[f] \otimes [T_{x}]} - \overline{[f] \otimes [T_{\partial}]}.\]
Since $\overline{[f] \otimes [T_{\partial}]} \in U$, it is enough to prove that $\overline{[f] \otimes [T_x]} = 0$.

The group $\Torelli_g^1$ is generated by genus-$1$ bounding pair maps \cite{JohnsonGenus1BP}, i.e., maps $T_a T_b^{-1}$ such that $a$ and $b$ are disjoint
curves such that $a \cup b$ separates $\Sigma_g^1$ into two subsurfaces, one of which is homeomorphic to $\Sigma_1^2$.
It is thus enough to prove that $\overline{[f] \otimes [T_x]} = 0$ for $f = T_a T_b^{-1}$ a genus-$1$ bounding pair map. 
Let $T \cong \Sigma_{g-2}^3$ be the subsurface bounded by $x \cup a \cup b$:\\
\Figure{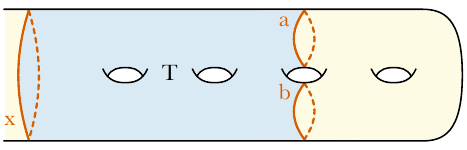}
Let $\Torelli_{g,1}^1(T)$ be the subgroup of $\Torelli_{g,1}^1$ consisting of mapping classes
supported on $T$.  We have 
\[\overline{[f] \otimes [T_x]} \in \overline{[f] \otimes \HH_1(\Torelli_{g,1}^1(T);\Q)}.\]  
Since $g \geq 4$ (see Assumption \ref{assumption:genus}), the surface $T$ has genus at least $2$.  It then follows from \cite[Lemma 6.2]{PutmanJohnson} that $T_x$ maps to $0$ in
$\HH_1(\Torelli_{g,1}^1(T);\Q)$.  This implies that $\overline{[f] \otimes [T_x]} = 0$, as desired.
\end{proof}

\section{Reidemeister pairing}
\label{section:reidemeister}

We now turn to the Reidemeister pairing on $\cC_g = \HH_1([\cC_g,\cC_g];\Q)$.

\subsection{Definition of pairing}
\label{section:reidemeisterdefinition}

Let $\tSigma_g \rightarrow \Sigma_g$ be the universal abelian cover of $\Sigma_g$.
On $\cC_g = \HH_1(\tSigma_g;\Q)$, we have the algebraic intersection form $\iota\colon \cC_g \times \cC_g \rightarrow \Q$.  
We also have the action of the deck group $H_{\Z}$ on $\tSigma_g$ and hence on $\cC_g$.  The
{\em Reidemeister pairing} \cite{Reidemeister1, Reidemeister2} is the linear map $\fr\colon \cC_g^{\otimes 2} \rightarrow \Q[H_{\Z}]$ defined
via the formula\footnote{Since $x$ and $y$ are supported on compact subsurfaces of $\tSigma_g$ and $H_{\Z}$ acts properly discontinuously
on $\tSigma_g$, all but finitely many terms in this sum vanish.}
\[\fr(x \otimes y) = \sum_{h \in H_{\Z}} \iota(h \Cdot x, y) h \quad \text{for $x,y \in \cC_g$}.\]

\subsection{Connecting homomorphism}

The reason the Reidemeister pairing is relevant to our work is the following.  Consider
the short exact sequence of representations from Lemma \ref{lemma:capboundaryrep}:
\[\begin{tikzcd}
0 \arrow{r} & \Q[H_{\Z}] \arrow{r} & \cC_g^1 \arrow{r} & \cC_{g} \arrow{r} & 0.
\end{tikzcd}\]
These are representations of $\Mod_{g,1}$, but we will view them as representations
of the subgroup $[\pi_g,\pi_g]$ of the point-pushing subgroup $\pi_g < \Mod_{g,1}$.  The
group $[\pi_g,\pi_g]$ acts trivially on $\Q[H_{\Z}]$ and $\cC_g$, but it does not
act trivially on $\cC_g^1$.  Consider the long exact sequence in group homology:
\[\begin{tikzcd}[row sep=small, column sep=small]
\cdots \arrow{r} & \HH_1([\pi_g,\pi_g];\cC_g^1) \arrow{r} & \HH_1([\pi_g,\pi_g];\cC_g) \arrow[equals]{d} \arrow{r}{\fr'} & \HH_0([\pi_g,\pi_g];\Q[H_{\Z}]) \arrow[equals]{d} \arrow{r} & \cdots \\
                 & & \cC_g^{\otimes 2}                                                                                   & \Q[H_{\Z}] & 
\end{tikzcd}\]
As the following shows, the connecting homomorphism $\fr'\colon \cC_g^{\otimes 2} \rightarrow \Q[H_{\Z}]$ equals the Reidemeister pairing:

\begin{lemma}
\label{lemma:reidemeisterboundary}
Let the notation be as above.  Then $\fr' = \fr$.
\end{lemma}
\begin{proof}
This can be proved exactly like \cite[Lemma 5.2]{PutmanAbelian}.  See \cite[Lemma 7.1]{PutmanStableLevel} for an alternate
exposition of the argument.
\end{proof}

\section{Proofs of main theorems}
\label{section:mainproofs}

We now come to the proofs of our main theorems.  The key to them
is the following.  Since $\Torelli_{g,1}$ acts trivially
on $H_{\Z}$, the
Reidemeister pairing $\fr\colon \cC_g^{\otimes 2} \rightarrow \Q[H_{\Z}]$ factors
through a map $\ofr\colon (\cC_g^{\otimes 2})_{\Torelli_{g,1}} \rightarrow \Q[H_{\Z}]$
of $\Sp_{2g}(\Z)$-representations that we will call the {\em coinvariant Reidemeister pairing}.  We then have:

\begin{theorem}
\label{theorem:bigtheorem}
Let $\ofr\colon (\cC_g^{\otimes 2})_{\Torelli_{g,1}} \rightarrow \Q[H_{\Z}]$ be the coinvariant Reidemeister pairing.
Then both $\ker(\ofr)$ and $\coker(\ofr) = \coker(\fr)$ are finite-dimensional.  Moreover, $\ker(\ofr)$ is an algebraic
representation of $\Sp_{2g}(\Z)$.
\end{theorem}

\begin{remark}
With a bit more work, one can show that $\Image(\ofr) = \Image(\fr)$ is as follows.  Let
$\epsilon\colon \Q[H_{\Z}] \rightarrow \Q$ be the augmentation and let $\fh\colon \Q[H_{\Z}] \rightarrow H$
be the map coming from the inclusion $H_{\Z} \hookrightarrow H$.
Then $\Image(\ofr)$ is the kernel of the surjective map $\epsilon \times \fh\colon \Q[H_{\Z}] \rightarrow \Q \times H$.
Consequently, $\coker(\ofr) \cong \Q \oplus H$ is also an algebraic representation of $\Sp_{2g}(\Z)$.
\end{remark}

The rest of this paper will be devoted to proving Theorem \ref{theorem:bigtheorem}: in Part \ref{part:2}
we calculate $\Image(\ofr)$ and find generators for $\ker(\ofr)$, and in Part \ref{part:3} we
find some relations in $\ker(\ofr)$ and use these relations to 
prove that $\ker(\ofr)$ is a finite-dimensional algebraic representation of $\Sp_{2g}(\Z)$.
For now, we assume the truth of Theorem \ref{theorem:bigtheorem} and show how to prove our main results.

\subsection{First main theorem}
The first of these main theorems is:\footnote{This is copied from the introduction before we imposed our genus
assumption, but we remind the reader that we are assuming throughout the paper that $g \geq 4$ (see Assumption \ref{assumption:genus}).}

\newtheorem*{maintheorem:findim}{Theorem \ref{maintheorem:findim}}
\begin{maintheorem:findim}
For $g \geq 4$, both $\HH_1(\Mod_g^1;\cC_g^1)$ and $\HH_1(\Torelli_g^1;\cC_g^1)$ are finite-dimensional.
Moreover, $\HH_1(\Torelli_g^1;\cC_g^1)$ is an algebraic representation of $\Sp_{2g}(\Z)$.
\end{maintheorem:findim}
\begin{proof}[Proof of Theorem \ref{maintheorem:findim} for the Torelli group, assuming Theorem \ref{theorem:bigtheorem}]
Lemma \ref{lemma:capboundaryrep} gives an exact sequence of representations
\[\begin{tikzcd}
0 \arrow{r} & \Q[H_{\Z}] \arrow{r} & \cC_g^1 \arrow{r} & \cC_{g} \arrow{r} & 0.
\end{tikzcd}\]
The associated long exact sequence in $\Torelli_g^1$-homology contains 
\[\begin{tikzcd}[row sep=small]
\HH_1(\Torelli_g^1;\cC_g^1) \arrow{r}{\iota} & \HH_1(\Torelli_g^1;\cC_{g}) \arrow{r}{\bb} & \HH_0(\Torelli_g^1;\Q[H_{\Z}]). \arrow[equals]{d} \\
                                             &                                            & \Q[H_{\Z}]
\end{tikzcd}\]
Lemma \ref{lemma:capboundaryrep2} says that $\ker(\iota)$ is a finite-dimensional algebraic representation of $\Sp_{2g}(\Z)$.  We have $\Image(\iota) = \ker(\bb)$, so using
\eqref{eqn:algclosed} we deduce that it is enough to prove that the kernel of the connecting homomorphism $\bb$ is also a finite-dimensional
algebraic representation of $\Sp_{2g}(\Z)$.

There is a similar connecting homomorphism $\bd\colon \HH_1(\Torelli_{g,1};\cC_g) \rightarrow \Q[H_{\Z}]$.  
Regarding $\pi_g$ as the point-pushing subgroup of $\Torelli_{g,1}$, 
we also have connecting homomorphisms
$\bpartial\colon \HH_1(\pi_g;\cC_g) \rightarrow \Q[H_{\Z}]$ and $\fr\colon \HH_1([\pi_g,\pi_g];\cC_g) \rightarrow \Q[H_{\Z}]$.  Identifying
$\HH_1([\pi_g,\pi_g];\cC_g)$ with $\cC_g^{\otimes 2}$, Lemma \ref{lemma:reidemeisterboundary} says that $\fr$ is the Reidemeister pairing.
These factor through maps $\obpartial\colon \HH_1(\pi_g;\cC_g)_{\Torelli_{g,1}} \rightarrow \Q[H_{\Z}]$ and
$\ofr\colon \HH_1([\pi_g,\pi_g];\cC_g)_{\Torelli_{g,1}} \rightarrow \Q[H_{\Z}]$, with $\ofr$ the coinvariant Reidemeister pairing.
These all fit into a commutative diagram
\[\begin{tikzcd}[column sep=small]
\HH_1([\pi_g,\pi_g];\cC_g)_{\Torelli_{g,1}} \arrow{d}{\ofr} \arrow{r} & \HH_1(\pi_g;\cC_g)_{\Torelli_{g,1}} \arrow{d}{\obpartial} \arrow{r} & \HH_1(\Torelli_{g,1};\cC_g) \arrow{d}{\bd} & \HH_1(\Torelli_g^1;\cC_g) \arrow{d}{\bb} \arrow{l} \\
\Q[H_{\Z}]                \arrow[equals]{r}        & \Q[H_{\Z}]                  \arrow[equals]{r} & \Q[H_{\Z}]         \arrow[equals]{r}              & \Q[H_{\Z}]
\end{tikzcd}\]
From this, we get maps
\[\begin{tikzcd}
\ker(\ofr) \arrow{r} & \ker(\obpartial) \arrow{r} & \ker(\bd) & \ker(\bb) \arrow{l}
\end{tikzcd}\]
Our goal is to prove that $\ker(\bb)$ is a finite-dimensional algebraic representation of $\Sp_{2g}(\Z)$, and by Theorem \ref{theorem:bigtheorem} we know that
this holds for $\ker(\ofr)$.  We work from left to right: 

\begin{claim}{1}
\label{claim:fintorelli1}
We have that $\ker(\obpartial)$ is a finite-dimensional algebraic representation of $\Sp_{2g}(\Z)$.
\end{claim}

Since $\ker(\ofr)$ is a finite-dimensional algebraic representation of $\Sp_{2g}(\Z)$, by \eqref{eqn:algclosed}
it is enough to prove
that this also holds for the cokernel of the map $\ker(\ofr) \rightarrow \ker(\obpartial)$.
Consider the short exact sequence
\[\begin{tikzcd}
1 \arrow{r} & {[\pi_g,\pi_g]} \arrow{r} & \pi_g \arrow{r} & H_{\Z} \arrow{r} & 1.
\end{tikzcd}\]
Since $[\pi_g,\pi_g]$ acts trivially on $\cC_g$, it follows from the 5-term exact sequence in group
homology with coefficients in $\cC_g$ that we have an exact sequence\footnote{The usual 5-term exact
sequence has $\HH_1([\pi_g,\pi_g];\cC_g)_{H_{\Z}}$; however, taking these coinvariants is only needed
if you want to continue it to the left.}
\[\begin{tikzcd}[column sep=small]
\HH_1([\pi_g,\pi_g];\cC_g) \arrow{r} & \HH_1(\pi_g;\cC_g) \arrow{r} & \HH_1(H_{\Z};\cC_g) \arrow{r} & 0.
\end{tikzcd}\]
It follows from \cite[Theorem D]{PutmanCommutator} that
$\HH_1(H_{\Z};\cC_g) \cong \wedge^{3} H$, which is a finite-dimensional algebraic representation
of $\Sp_{2g}(\Z)$.  Since
taking coinvariants is right-exact, the above remains exact if we take $\Torelli_{g,1}$-coinvariants.
Do this and add $\ofr$ and $\obpartial$:
\[\begin{tikzcd}[column sep=small]
\HH_1([\pi_g,\pi_g];\cC_g)_{\Torelli_{g,1}} \arrow{d}{\ofr} \arrow{r} & \HH_1(\pi_g;\cC_g)_{\Torelli_{g,1}} \arrow{d}{\obpartial} \arrow{r} & \HH_1(H_{\Z};\cC_g)_{\Torelli_{g,1}} \arrow[equals]{d} \arrow{r} & 0 \\
\Q[H_{\Z}]                                  \arrow[equals]{r}         & \Q[H_{\Z}]                                                          & \wedge^3 H
\end{tikzcd}\]
This is a commutative diagram of $\Mod_{g,1}$-representations.  Let $U$ be the image of $\ker(\obpartial)$ in $\wedge^3 H$.  By \eqref{eqn:algclosed}, $U$ is a finite-dimensional
algebraic representation of $\Sp_{2g}(\Z)$.  Examining
the above diagram, we see that its top row restricts to an exact sequence
\[\begin{tikzcd}[column sep=small]
\ker(\ofr) \arrow{r} & \ker(\obpartial) \arrow{r} & U  \arrow{r} & 0.
\end{tikzcd}\]
In other words, the cokernel of the map $\ker(\ofr) \rightarrow \ker(\obpartial)$
is isomorphic to $U$, which is a finite-dimensional algebraic representation of $\Sp_{2g}(\Z)$, as desired.

\begin{claim}{2}
We have that $\ker(\bd)$ is a finite-dimensional algebraic representation of $\Sp_{2g}(\Z)$.
\end{claim}

Since $\ker(\obpartial)$ is a finite-dimensional algebraic representation of $\Sp_{2g}(\Z)$, by
\eqref{eqn:algclosed} it is enough to prove that this also holds for the cokernel of the map
$\ker(\obpartial) \rightarrow \ker(\bd)$.
Consider the Birman exact sequence from Corollary \ref{corollary:birmantorelli}:
\[\begin{tikzcd}
1 \arrow{r} & \pi_g \arrow{r} & \Torelli_{g,1} \arrow{r} & \Torelli_g \arrow{r} & 1.
\end{tikzcd}\]
Just like in Claim \ref{claim:fintorelli1}, we can derive from this a commutative diagram\footnote{The reader might
expect to see the coinvariants $\HH_1(\Torelli_{g,1};\cC_g)_{\Torelli_{g,1}}$ here, but since $\Torelli_{g,1}$ acts trivially
on $\HH_1(\Torelli_{g,1};\cC_g)$ we have $\HH_1(\Torelli_{g,1};\cC_g)_{\Torelli_{g,1}} = \HH_1(\Torelli_{g,1};\cC_g)$, so the coinvariants are not needed.}
\[\begin{tikzcd}[column sep=small]
\HH_1(\pi_g;\cC_g)_{\Torelli_{g,1}} \arrow{r} \arrow{d}{\obpartial} & \HH_1(\Torelli_{g,1};\cC_g) \arrow{d}{\bd} \arrow{r} & \HH_1(\Torelli_g;(\cC_g)_{\pi_g}) \arrow{r} & 0. \\
\Q[H_{\Z}]                                  \arrow[equals]{r}         & \Q[H_{\Z}]                                         &                                             &
\end{tikzcd}\]
whose first row is exact.
Lemma \ref{lemma:cgpicoinv} says that $(\cC_g)_{\pi_g} \cong (\wedge^2 H)/\Q$, so
\[\HH_1(\Torelli_g;(\cC_g)_{\pi_g}) \cong \HH_1(\Torelli_g;(\wedge^2 H)/\Q) \cong \HH_1(\Torelli_g;\Q) \otimes ((\wedge^2 H)/\Q).\]
Johnson \cite{JohnsonAbel} proved that $\HH_1(\Torelli_g;\Q)$ is a finite-dimensional algebraic representation of $\Sp_{2g}(\Z)$,
so by \eqref{eqn:algclosed} we deduce that $\HH_1(\Torelli_g;(\cC_g)_{\pi_g})$ is as well.
An argument identical to the one used in Claim \ref{claim:fintorelli1} now shows that the 
cokernel of the map
$\ker(\obpartial) \rightarrow \ker(\bd)$ is isomorphic to a subrepresentation
of $\HH_1(\Torelli_g(\cC_g)_{\pi_g})$, and thus by \eqref{eqn:algclosed} is a finite-dimensional
algebraic representation of $\Sp_{2g}(\Z)$, as desired.

\begin{claim}{3}
We have that $\ker(\bb)$ is a finite-dimensional algebraic representation of $\Sp_{2g}(\Z)$.
\end{claim}

Since $\ker(\bd)$ is a finite-dimensional algebraic representation of $\Sp_{2g}(\Z)$, by
\eqref{eqn:algclosed} it is enough to prove that this also holds for the kernel of the map
$\ker(\bb) \rightarrow \ker(\bd)$.  To help us manipulate this kernel, we call this map
$\phi\colon \ker(\bb) \rightarrow \ker(\bd)$.

Recall that our connecting homomorphisms $\bb$ and $\bd$ form part of the long exact sequences in homology associated
to the short exact sequence
\[\begin{tikzcd}
0 \arrow{r} & \Q[H_{\Z}] \arrow{r} & \cC_g^1 \arrow{r} & \cC_{g} \arrow{r} & 0
\end{tikzcd}\]
of representations from Lemma \ref{lemma:capboundaryrep}.  They thus fit into a commutative diagram
\[\begin{tikzcd}
\HH_1(\Torelli_g^1;\Q[H_{\Z}])   \arrow{r} \arrow{d}{\Psi} & \HH_1(\Torelli_g^1;\cC_g^1)   \arrow{r} \arrow{d}{\Phi} & \HH_1(\Torelli_g^1;\cC_g)   \arrow{r}{\bb} \arrow{d} & \Q[H_{\Z}] \arrow[equals]{d} \\
\HH_1(\Torelli_{g,1};\Q[H_{\Z}]) \arrow{r}                 & \HH_1(\Torelli_{g,1};\cC_g^1) \arrow{r}                 & \HH_1(\Torelli_{g,1};\cC_g) \arrow{r}{\bd}           & \Q[H_{\Z}]
\end{tikzcd}\]
with exact rows.  This gives a commutative diagram
\[\begin{tikzcd}
\HH_1(\Torelli_g^1;\Q[H_{\Z}])   \arrow{r} \arrow{d}{\Psi} & \HH_1(\Torelli_g^1;\cC_g^1)   \arrow{r} \arrow{d}{\Phi} & \ker(\bb) \arrow{d}{\phi} \arrow{r} & 0\\
\HH_1(\Torelli_{g,1};\Q[H_{\Z}]) \arrow{r}                 & \HH_1(\Torelli_{g,1};\cC_g^1) \arrow{r}                 & \ker(\bd) \arrow{r}                 & 0
\end{tikzcd}\]
with exact rows.

We claim that the map $\Psi$ in this diagram is an isomorphism.  Indeed, both $\Torelli_g^1$ and $\Torelli_{g,1}$ act trivially on $\Q[H_{\Z}]$, so
\[\HH_1(\Torelli_g^1;\Q[H_{\Z}]) = \HH_1(\Torelli_g^1;\Q) \otimes \Q[H_{\Z}] \quad \text{and} \quad \HH_1(\Torelli_{g,1};\Q[H_{\Z}]) = \HH_1(\Torelli_{g,1};\Q) \otimes \Q[H_{\Z}].\]
That $\Psi$ is an isomorphism is thus equivalent to the fact that the map $\HH_1(\Torelli_g^1;\Q) \rightarrow \HH_1(\Torelli_{g,1};\Q)$ is an isomorphism,
which follows from Johnson's computation of the first homology of the Torelli group \cite{JohnsonAbel}.

Since $\Psi$ is an isomorphism,\footnote{Actually, all that is needed is that it is a surjection.} a diagram chase shows that the map $\ker(\Phi) \rightarrow \ker(\phi)$ is a surjection.  Our goal is to prove
that $\ker(\phi)$ is a finite-dimensional algebraic representation of $\Sp_{2g}(\Z)$, so by \eqref{eqn:algclosed} it is enough
to prove this for $\ker(\Phi)$.  This is exactly the content of Corollary \ref{corollary:capboundaryseq}, so we are done. 
\end{proof}

The above only proved part of Theorem \ref{maintheorem:findim}.  It remains to prove that $\HH_1(\Mod_g^1;\cC_g^1)$ is finite-dimensional.

\begin{proof}[Proof of Theorem \ref{maintheorem:findim} for the mapping class group, assuming Theorem \ref{theorem:bigtheorem}]
Consider the exact sequence
\[\begin{tikzcd}
1 \arrow{r} & \Torelli_g^1 \arrow{r} & \Mod_g^1 \arrow{r} & \Sp_{2g}(\Z) \arrow{r} & 1.
\end{tikzcd}\]
The associated 5-term exact sequence with coefficients in $\cC_g^1$ contains
\[\begin{tikzcd}
\HH_1(\Torelli_g^1;\cC_g^1)_{\Sp_{2g}(\Z)} \arrow{r} & \HH_1(\Mod_g^1;\cC_g^1) \arrow{r} & \HH_1(\Sp_{2g}(\Z);(\cC_g^1)_{\Torelli_g^1}) \arrow{r} & 0.
\end{tikzcd}\]
We proved above that $\HH_1(\Torelli_g^1;\cC_g^1)$ is finite-dimensional, so
$\HH_1(\Torelli_g^1;\cC_g^1)_{\Sp_{2g}(\Z)}$ is also finite-dimensional.  Also, Lemma \ref{lemma:cg1torellicoinv} says
that $(\cC_g^1)_{\Torelli_g^1}$ is a finite-dimensional algebraic representation of $\Sp_{2g}(\Z)$.
Since $\Sp_{2g}(\Z)$ is finitely generated, it follows
that $\HH_1(\Sp_{2g}(\Z);(\cC_g^1)_{\Torelli_g^1})$ is finite-dimensional.  Plugging all of this into the above exact sequence, we conclude
that $\HH_1(\Mod_g^1;\cC_g^1)$ is finite-dimensional, as desired.
\end{proof}

\subsection{Second main theorem}

Our second main theorem is:\footnote{Just like above, this is copied from the introduction before we imposed our genus
assumption, but we remind the reader that we are assuming throughout the paper that $g \geq 4$ (see Assumption \ref{assumption:genus}).}

\newtheorem*{maintheorem:infdim}{Theorem \ref{maintheorem:infdim}}
\begin{maintheorem:infdim}
For $g \geq 4$, both $\HH_1(\Mod_{g,1};\cC_g)$ and $\HH_1(\Torelli_{g,1};\cC_g)$ are infinite-dimensional.
\end{maintheorem:infdim}

\begin{proof}[Proof of Theorem \ref{maintheorem:infdim} for the Torelli group, assuming Theorem \ref{theorem:bigtheorem}]
Consider the short exact sequence of representations from Lemma \ref{lemma:capboundaryrep}:
\[\begin{tikzcd}
0 \arrow{r} & \Q[H_{\Z}] \arrow{r} & \cC_g^1 \arrow{r} & \cC_{g} \arrow{r} & 0.
\end{tikzcd}\]
There is an associated long exact sequence in $\Torelli_{g,1}$-homology that contains the
connecting homomorphism $\bd\colon \HH_1(\Torelli_{g,1};\cC_g) \rightarrow \Q[H_{\Z}]$.
It is enough to prove that the image of $\bd$ is infinite-dimensional.
Let $\pi_g$ be the point-pushing subgroup of $\Torelli_{g,1}$.  By looking at the associated long exact sequence
in $[\pi_g,\pi_g]$-homology, we get a connecting homomorphism $\fr\colon \HH_1([\pi_g,\pi_g];\cC_g) \rightarrow \Q[H_{\Z}]$
fitting into a commutative diagram
\begin{equation}
\label{eqn:reidemeisternaturality}
\begin{tikzcd}[row sep=small]
\HH_1([\pi_g,\pi_g];\cC_g) \arrow{d} \arrow{r}{\fr}     & \Q[H_{\Z}] \arrow[equals]{d} \\
\HH_1(\Torelli_{g,1};\cC_g) \arrow{r}{\bd} & \Q[H_{\Z}]
\end{tikzcd}
\end{equation}
Identifying $\HH_1([\pi_g,\pi_g];\cC_g)$ with $\cC_g^{\otimes 2}$, Lemma \ref{lemma:reidemeisterboundary} says that
$\fr$ is the Reidemeister pairing.  Theorem \ref{theorem:bigtheorem} then implies that $\Image(\fr)$ is infinite-dimensional,
so $\Image(\bd)$ is also infinite-dimensional.
\end{proof}

The above only proved part of Theorem \ref{maintheorem:infdim}.  It remains to prove that $\HH_1(\Mod_{g,1};\cC_g)$ is infinite-dimensional:

\begin{proof}[Proof of Theorem \ref{maintheorem:infdim} for the mapping class group, assuming Theorem \ref{theorem:bigtheorem}]
We have a short exact sequence
\[\begin{tikzcd}
1 \arrow{r} & \Torelli_{g,1} \arrow{r} & \Mod_{g,1} \arrow{r} & \Sp_{2g}(\Z) \arrow{r} & 1.
\end{tikzcd}\]
The associated 5-term exact sequence in homology with coefficients in $\cC_g$ contains the segment
\[\begin{tikzcd}[column sep=small]
\HH_2(\Sp_{2g}(\Z);(\cC_g)_{\Torelli_{g,1}}) \arrow{r} & \HH_1(\Torelli_{g,1};\cC_g)_{\Sp_{2g}(\Z)} \arrow{r} & \HH_1(\Mod_{g,1};\cC_g).
\end{tikzcd}\]
Since $\Torelli_{g,1}$ contains the point-pushing subgroup $\pi_g$, Lemma \ref{lemma:cgpicoinv} implies
that $(\cC_g)_{\Torelli_{g,1}} \cong (\wedge^2 H)/\Q$.  It follows that
\[\HH_2(\Sp_{2g}(\Z);(\cC_g)_{\Torelli_{g,1}}) \cong \HH_2(\Sp_{2g}(\Z);(\wedge^2 H)/\Q)\]
is finite-dimensional.  We deduce that it is enough to prove that the $\Sp_{2g}(\Z)$-coinvariants of $\HH_1(\Torelli_{g,1};\cC_g)$ 
are infinite-dimensional.  Let $\bd\colon \HH_1(\Torelli_{g,1};\cC_g) \rightarrow \Q[H_{\Z}]$ be the connecting homomorphism
discussed in the previous proof.  Since taking coinvariants is right exact, we have a surjection
$\HH_1(\Torelli_{g,1};\cC_g)_{\Sp_{2g}(\Z)} \rightarrow \Image(\bd)_{\Sp_{2g}(\Z)}$.  It is thus
enough to prove that $\Image(\bd)_{\Sp_{2g}(\Z)}$ is infinite-dimensional.

Consider the short exact sequence of representations
\[\begin{tikzcd}
0 \arrow{r} & \Image(\bd) \arrow{r} & {\Q[H_{\Z}]} \arrow{r} & \coker(\bd) \arrow{r} & 0.
\end{tikzcd}\]
The associated long exact sequence in $\Sp_{2g}(\Z)$-homology contains
\begin{small}
\[\begin{tikzcd}[column sep=small, row sep=small]
\HH_1(\Sp_{2g}(\Z);\coker(\bd)) \arrow{r} & {\HH_0(\Sp_{2g}(\Z);\Image(\bd))} \arrow[equals]{d} \arrow{r} & {\HH_0(\Sp_{2g}(\Z);\Q[H_{\Z}])} \arrow[equals]{d} \arrow{r} & {\HH_0(\Sp_{2g}(\Z);\coker(\bd))} \arrow[equals]{d}\\
                                                  & {\Image(\bd)_{\Sp_{2g}(\Z)}}                          & {\Q[H_{\Z}]_{\Sp_{2g}(\Z)}}        & {\coker(\bd)_{\Sp_{2g}(\Z)}}                          
\end{tikzcd}\]
\end{small}%
Using the commutative diagram \eqref{eqn:reidemeisternaturality},
Theorem \ref{theorem:bigtheorem} implies that $\coker(\bd)$ is finite-dimensional.  Since $\Sp_{2g}(\Z)$ is finitely generated, we see that
$\HH_1(\Sp_{2g}(\Z);\coker(\bd))$ and $\coker(\bd)_{\Sp_{2g}(\Z)}$ are finite-dimensional.

It follows that $\Image(\bd)_{\Sp_{2g}(\Z)}$ is infinite-dimensional if and only if $\Q[H_{\Z}]_{\Sp_{2g}(\Z)}$ is infinite-dimensional, so we
only need to prove the latter fact.  But this is easy: the dimension of $\Q[H_{\Z}]_{\Sp_{2g}(\Z)}$ is the cardinality of the set of $\Sp_{2g}(\Z)$-orbits
in $H_{\Z}$, and there are infinitely many orbits.  Indeed, if $v \in H_{\Z}$ is primitive,\footnote{That is, not divisible by any integers other than $\pm 1$.}
then $\Set{$\ell \Cdot v$}{$\ell \geq 0$}$ is a complete set of orbit representatives.
\end{proof}

\part{Generators for the kernel of the coinvariant Reidemeister pairing}
\label{part:2}

Let $\ofr$ be the coinvariant Reidemeister
pairing.  Our remaining task is to prove Theorem \ref{theorem:bigtheorem}, which says
that $\ker(\ofr)$ is a finite-dimensional
algebraic representation of $\Sp_{2g}(\Z)$ and that $\coker(\ofr)$ is finite-dimensional.  
In this part of the paper, we calculate $\Image(\ofr)$ and find
generators for $\ker(\ofr)$.  We will then find some relations in $\ker(\ofr)$ in Part \ref{part:3} and
use these generators and relations to complete the proof of Theorem \ref{theorem:bigtheorem}.
We will outline this part more in the introductory \S \ref{section:part2intro} below.

\section{Introduction to Part \ref{part:2}}
\label{section:part2intro}

This section fixes some notation and does some preliminary calculations, and then outlines
what we will do in the rest of Part \ref{part:2}.

\subsection{Conjugation and commutator conventions}
\label{section:conjugationcommutator}

Let $G$ be a group.  We have been viewing the conjugation action of $G$ on
itself as a left action.  For $x,y \in G$, we therefore write $\precon{y}{x}$ for $y x y^{-1}$.
With this notation, we have $\precon{z}{(\precon{y}{x})} = \precon{zy}{x}$.
For $x,y \in G$ we also write $[x,y] = x y x^{-1} y^{-1}$.

\subsection{Notation for group ring}

For $h \in H_{\Z}$, let $\GrpR{h}$ denote the corresponding element of $\Q[H_{\Z}]$.
Though $\Z$ acts on both $H_{\Z}$ and $\Q[H_{\Z}]$, for $n \in \Z$
the elements $n \GrpR{h}$ and $\GrpR{n h}$ are not the same.  For $x \in \pi_g$, 
let $\ox \in H_{\Z}$ be its image, so $\GrpR{\ox} \in \Q[H_{\Z}]$.

\subsection{Notation for \texorpdfstring{$\cC_g$}{Cg}}

For $x,y \in \pi_g$, let $\Cg{x,y}$ be the element of $\cC_g = \HH_1([\pi_g,\pi_g];\Q)$ corresponding
to $[x,y]$.  Similarly, for $z \in [\pi_g,\pi_g]$ let $\Cg{z}$ be the corresponding element
of $\cC_g$.  The group $\pi_g$ has a left action by conjugation on $[\pi_g,\pi_g]$.  This descends to a left action of
$H_{\Z}$ on $\cC_g$.  Since $H_{\Z}$ is abelian, it is harmless\footnote{The point here is that since $H_{\Z}$ is abelian,
even though this is a left action for $c \in \cC_g$ and $h_1,h_2 \in H_{\Z}$ we have $(c^{h_1})^{h_2} = c^{h_1+h_2}$.}
to write this with superscripts: 
for $c \in \cC_g$ and $h \in H_{\Z}$, we denote the image of $c$ under the action
of $h$ by $c^h$.  For $x,y,z \in \pi_g$, it follows that
$\Cg{x,y}^{\oz}$ is the element of $\cC_g$ corresponding to $\precon{z}{[x,y]}$.

\subsection{Commutator identities}

Commutator identities give identities between the elements $\Cg{x,y}^h$.  The ones we need are:

\begin{lemma}[Commutator identities]
\label{lemma:commutatoridentities}
Let $x,y,z \in \pi_g$ and $h \in H_{\Z}$.  The following hold:
\begin{itemize}
\item $\Cg{y,x}^h = -\Cg{x,y}^h$
\item $\Cg{xy,z}^h = \Cg{x,z}^h + \Cg{y,z}^{h+\ox}$
\item $\Cg{x^{-1},y}^h = -\Cg{x,y}^{h - \ox}$.
\end{itemize}
\end{lemma}
\begin{proof}
The first follows from the commutator identity $[y,x]=[x,y]^{-1}$.  The second follows
from the commutator identity $[xy,z]=\precon{x}[y,z] [x,z]$.  The third follows 
from the commutator identity $([x^{-1},y])(\precon{x^{-1}}{[x,y]}) = 1$.
\end{proof}

\subsection{Separation properties of curves}
\label{section:separationpropertiesofcurves}

Say that a collection of elements of $\pi_g$ are {\em almost disjoint} if they can be realized
so as to only intersect at the basepoint.  Also, $\delta \in \pi_g$ is said to {\em separate} 
a subset $C_1 \subset \pi_g$ from a subset $C_2 \subset \pi_g$ if:
\begin{itemize}
\item $\delta$ is a simple closed separating curve\footnote{This implies that $\delta \in [\pi_g,\pi_g]$.  Also, like
we described in \S \ref{section:conventions} when we say that $\delta$ is a simple closed separating curve we mean
that it can be realized by such a curve.} that separates $\Sigma_g$
into subsurfaces $S_1$ and $S_2$, ordered such that $S_1$ is to the left of $\delta$ and
$S_2$ to the right of $\delta$; and
\item for $i=1,2$, each curve in $C_i$ can be realized so as to lie in $S_i$.
\end{itemize}
See here, where the curves in $C_1$ and $C_2$ are in two different colors:\\
\Figure{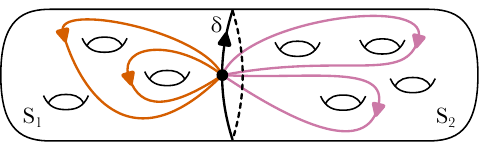}
Note that this is {\em not} symmetric; in fact, if $\delta$ separates $C_1$ from $C_2$, then
$\delta^{-1}$ separates $C_2$ from $C_1$.  We also allow $\delta$ to be an element of $C_1$ or
$C_2$ (or both!).  For instance, if $\delta \in \pi_g$ is a simple closed separating curve and $\gamma \in \pi_g$
is almost disjoint from $\delta$, then $\delta$ separates $\{\delta\}$ from $\{\gamma\}$.  In fact, in this
case $\delta$ even separates $\{\delta\}$ from $\{\delta,\gamma\}$.

For subsets $C_1 \subset \pi_g$ and $C_2 \subset \pi_g$, we say that $C_1$ and $C_2$ are {\em separated} if
there exists a $\delta \in \pi_g$ that separates $C_1$ from $C_2$.  This implies that each curve in $C_1$ is almost
disjoint from each curve in $C_2$.
We will abuse notation in the obvious way and talk about a single $\gamma \in \pi_g$ and a subset $C \subset \pi_g$ being
separated, etc.  

\subsection{Key Reidemeister pairing calculation}

Recall that the Reidemeister pairing is the map $\fr\colon \cC_g^{\otimes 2} \rightarrow \Q[H_{\Z}]$
defined as follows.  Let $\tSigma_g \rightarrow \Sigma_g$ be the universal abelian cover of $\Sigma_g$, so
$H_{\Z}$ is the deck group of $\tSigma_g$ and $\cC_g = \HH_1(\tSigma_g;\Q)$.  Let $\iota$ be the algebraic
intersection pairing on $\cC_g = \HH_1(\tSigma_g;\Q)$.  Then
\[\fr(x \otimes y) = \sum_{h \in H_{\Z}} \iota(h \Cdot x, y) h \quad \text{for $x,y \in \cC_g$}.\]
Our results about the Reidemeister pairing are based on the following calculation:

\begin{lemma}
\label{lemma:coinvariantreidemeister}
Let $\fr\colon \cC_g^{\otimes 2} \rightarrow \Q[H_{\Z}]$ be the Reidemeister pairing.  Then:
\begin{itemize}
\item[(a)] If $\gamma_1,\gamma_2 \in [\pi_g,\pi_g]$ are almost 
disjoint and $h_1,h_2 \in H_{\Z}$, then $\fr(\Cg{\gamma_1}^{h_1} \otimes \Cg{\gamma_2}^{h_2}) = 0$.
\item[(b)] If $\delta \in [\pi_g,\pi_g]$ separates $\eta \in \pi_g$ from $\lambda \in \pi_g$, then for arbitrary $h \in H_{\Z}$ we have
$\fr(\Cg{\delta} \otimes \Cg{\eta,\lambda}^h) = \GrpR{h} - \GrpR{h+\oeta} - \GrpR{h+\olambda} + \GrpR{h+\oeta+\olambda}$.
\end{itemize}
\end{lemma}
\begin{proof}
Let $\rho\colon \tSigma_g \rightarrow \Sigma_g$ be the universal abelian cover of $\Sigma_g$ and let $\iota$ be the algebraic
intersection pairing on $\tSigma_g$.  The cover $\tSigma_g$ has a basepoint, and 
for $x \in [\pi_g,\pi_g]$, the corresponding element $\Cg{x} \in \cC_g$ is the homology class of the closed curve
obtained by lifting the based curve $x$ to
$\tSigma_g$ starting at the basepoint of $\tSigma_g$.  We prove the two parts separately:

\begin{claim}{1}
\label{claim:lift1}
If $\gamma_1,\gamma_2 \in [\pi_g,\pi_g]$ are almost 
disjoint and $h_1,h_2 \in H_{\Z}$, then $\fr(\Cg{\gamma_1}^{h_1} \otimes \Cg{\gamma_2}^{h_2}) = 0$.
\end{claim}

Since $\gamma_1$ and $\gamma_2$ lie in $[\pi_g,\pi_g]$, their algebraic intersection number on $\Sigma_g$ is $0$.  Their
single intersection at the basepoint is thus not a transverse intersection, so
$\gamma_1$ can be freely homotoped to a curve $\gamma'_1$ that is disjoint from $\gamma_2$ as follows:\\
\Figure{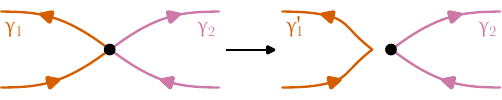}
Let $\tgamma_1$ and $\tgamma_2$ be the lifts of $\gamma_1$ and $\gamma_2$ to $\tSigma_g$.  Lifting the homotopy between
$\gamma_1$ and $\gamma'_1$ to $\tSigma_g$ starting at $\tgamma_1$, we get a lift $\tgamma'_1$ of $\gamma'_1$ that
is homotopic to $\tgamma_1$.
Since $\gamma'_1$ and $\gamma_2$ are disjoint, for $k_1,k_2 \in H_{\Z}$ the curves $k_1 \Cdot \tgamma'_1$ and
$k_2 \Cdot \tgamma_2$ are also disjoint.  This implies that
$\iota(k_1 {\Cdot} [\tgamma_1],k_2 {\Cdot} [\tgamma_2]) = \iota(k_1 {\Cdot} [\tgamma'_1],k_2 {\Cdot} [\tgamma_2]) = 0$.
We conclude that
$\fr(\Cg{\gamma_1}^{h_1} \otimes \Cg{\gamma_2}^{h_2}) = \sum_{k \in H_{\Z}} \iota((k+h_1) {\Cdot} [\tgamma_1],h_2 {\Cdot} [\tgamma_2]) k = 0$.

\begin{claim}{2}
If $\delta \in [\pi_g,\pi_g]$ separates $\eta \in \pi_g$ from $\lambda \in \pi_g$, then for arbitrary $h \in H_{\Z}$ we have
$\fr(\Cg{\delta} \otimes \Cg{\eta,\lambda}^h) = \GrpR{h} - \GrpR{h+\oeta} - \GrpR{h+\olambda} + \GrpR{h+\oeta+\olambda}$.
\end{claim}

Let $\delta'$ be the curve obtained by homotoping $\delta$ like this:\\
\Figure{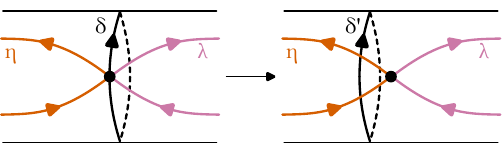}
Let $\tdelta$ be the lift of $\delta$ to $\tSigma_g$.  
Lifting the homotopy between
$\delta$ and $\delta'$ to $\tSigma_g$ starting at $\tdelta$, we get a lift $\tdelta'$ of $\delta'$ that
is homotopic to $\tdelta$.
We then have
\[\fr(\Cg{\delta} \otimes \Cg{\eta,\lambda}^h) = \sum_{k \in H_{\Z}} \iota(k \Cg{\delta},\Cg{\eta,\lambda}^h) = \sum_{k \in H_{\Z}} \iota(k {\Cdot} [\tdelta'],\Cg{\eta,\lambda}^h) k.\]
Let $\tast$ be the basepoint of $\tSigma_g$.
The homology class $\Cg{\eta,\lambda}^h$ is the homology class of the curve obtained by lifting $[\eta,\lambda] = \eta \lambda \eta^{-1} \lambda^{-1}$
to $\tSigma_g$ starting at $h \Cdot \tast$.  The curve $[\eta,\lambda]$ intersects $\delta'$ four times, so 
this lift will intersect four different curves of the form $k \Cdot \tdelta'$.
See here, which shows the initial segments of $[\eta,\lambda] \in \pi_g$ 
whose lifts end in those four translates of $\tdelta'$ (the numbering
is the order in which those intersections appear in $[\eta,\lambda]$) and where
the end of the lift is labeled with the $k \in H_{\Z}$ such that the lift terminates in $k \Cdot \tdelta'$:\\
\Figure{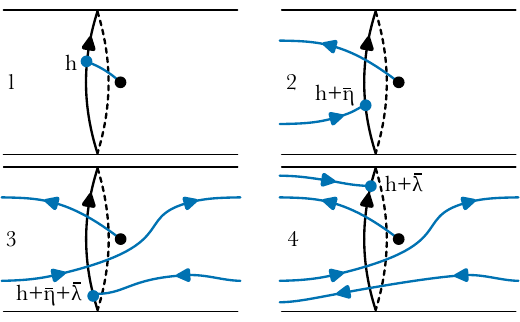}
We have perturbed some parts of these initial segments to make the picture easier to read.
Examining the signs of those intersections, those  labeled by $h$ and $h+\oeta+\olambda$ have a positive sign
and those labeled by $h+\oeta$ and $h+\olambda$ have a negative sign.  We conclude that
\[\fr(\Cg{\delta} \otimes \Cg{\eta,\lambda}^h) = \GrpR{h} - \GrpR{h+\oeta} - \GrpR{h+\olambda} + \GrpR{h+\oeta+\olambda}.\qedhere\]
\end{proof}

\subsection{Notation for coinvariant quotient}
\label{section:coinvariantquotient}

For $\kappa \in \cC_g^{\otimes 2}$, we denote by $\underline{\kappa}$ the associated
element of $(\cC_g^{\otimes 2})_{\Torelli_{g,1}}$.  For example, 
for $c_1,c_2 \in \cC_g$ the image in $(\cC_g^{\otimes 2})_{\Torelli_{g,1}}$ of $c_1 \otimes c_2 \in \cC_g^{\otimes 2}$
is written $\CgCg{c_1}{c_2}$.
Since the point-pushing subgroup $\pi_g$ of $\Torelli_{g,1}$ acts on $\pi_g$ by inner automorphisms,
for $h \in H_{\Z}$ the elements $c_1 \otimes c_2$ and $c_1^h \otimes c_2^h$ differ by an element
of $\Torelli_{g,1}$ and hence $\CgCg{c_1}{c_2} = \CgCg{c_1^h}{c_2^h}$.  Equivalently, we have
$\CgCg{c_1^h}{c_2} = \CgCg{c_1}{c_2^{-h}}$.

\subsection{Main result of Part 2}

Let $\ofr\colon (\cC_g^{\otimes 2})_{\Torelli_{g,1}} \rightarrow \Q[H_{\Z}]$ be the coinvariant
Reidemeister pairing.  Lemma \ref{lemma:coinvariantreidemeister}.(a) implies that $\ker(\ofr)$
contains all elements of the form $\CgCg{\Cg{\gamma_1}^{h_1}}{\Cg{\gamma_2}^{h_2}}$ with
$\gamma_1,\gamma_2 \in [\pi_g,\pi_g]$ almost disjoint and $h_1,h_2 \in H_{\Z}$.  We will prove
that these generate the kernel.  In fact, we only need elements where $\gamma_1$ and $\gamma_2$ are separated:

\begin{theorem}
\label{theorem:part2theorem}
Let $\ofr\colon (\cC_g^{\otimes 2})_{\Torelli_{g,1}} \rightarrow \Q[H_{\Z}]$ be the coinvariant Reidemeister pairing.  Then:
\begin{itemize}
\item the cokernel of $\ofr$ is finite-dimensional; and
\item the kernel of $\ofr$ is generated by the set of elements of the form
$\CgCg{\Cg{\gamma_1}^{h_1}}{\Cg{\gamma_2}^{h_2}}$ with $\gamma_1 \in [\pi_g,\pi_g]$ 
and $\gamma_2 \in [\pi_g,\pi_g]$ separated and $h_1,h_2 \in H_{\Z}$.
\end{itemize}
\end{theorem}

This will be the main theorem of this part of the paper.  Recall that
our yet-unproven Theorem \ref{theorem:bigtheorem} says that $\coker(\ofr)$ is finite-dimensional
and that $\ker(\ofr)$ is a finite-dimensional algebraic representation of $\Sp_{2g}(\Z)$.
The first conclusion of Theorem \ref{theorem:part2theorem} gives the first part of this,
and in Part \ref{part:3} we will use the generators for $\ker(\ofr)$ given by Theorem \ref{theorem:part2theorem}
to prove the second part, i.e., that $\ker(\ofr)$ is a finite-dimensional algebraic representation of $\Sp_{2g}(\Z)$.

\subsection{Outline of Part \ref{part:2}}
Before outlining our proof of Theorem \ref{theorem:part2theorem}, we rephrase it.  Make the following definition:

\begin{definition}
Let $\cQ_g$ be the quotient of $(\cC_g^{\otimes 2})_{\Torelli_{g,1}}$ by the span of the elements
of the form $\CgCg{\Cg{\gamma_1}^{h_1}}{\Cg{\gamma_2}^{h_2}}$ with $\gamma_1 \in [\pi_g,\pi_g]$ and
$\gamma_2 \in [\pi_g,\pi_g]$ separated and $h_1,h_2 \in H_{\Z}$.
\end{definition}

By Lemma \ref{lemma:coinvariantreidemeister}.(a), the coinvariant Reidemeister pairing
factors through a map $\fq\colon \cQ_g \rightarrow \Q[H_{\Z}]$ that we will call
the {\em quotiented Reidemeister pairing}.  Theorem \ref{theorem:part2theorem} is 
equivalent to:

\begin{theorem}
\label{theorem:part2theorem2}
Let $\fq\colon \cQ_g \rightarrow \Q[H_{\Z}]$ be the quotiented Reidemeister
pairing.  Then $\fq$ is an injection whose image has finite codimension.
\end{theorem}

The rest of Part \ref{part:2} is devoted to the proof of Theorem \ref{theorem:part2theorem2}.  The outline
is as follows.  In \S \ref{section:xelements}, we introduce generators $X(h,x,y)$ for $\cQ_g$.  
In \S \ref{section:quotientrel}, we construct some relations between the $X(h,x,y)$.  In
\S \ref{section:quotientgen}, we prove that $\cQ_g$ is generated by a certain subset
of the $X(h,x,y)$.  Finally, in \S \ref{section:calcquotientoutline} -- \S \ref{section:calcquotient3} we 
prove that these generators go to linearly independent elements of $\Q[H_{\Z}]$ that span a subspace
of finite codimension.

\section{Generators for the quotient}
\label{section:xelements}

This section constructs generators $X(h,x,y)$ for $\cQ_g$.

\subsection{Fixing the first curve, I}

Recall that $\cC_g = \HH_1([\pi_g,\pi_g];\Q)$.  
The group $\cQ_g$ is a quotient of $(\cC_g^{\otimes 2})_{\Torelli_{g,1}}$, which is itself a quotient
of $\cC_g^{\otimes 2}$.  For $\delta \in [\pi_g,\pi_g]$, let $\cQ_g[\delta]$ be the image of the map
$\cC_g \rightarrow \cQ_g$ taking $c \in \cC_g$ to $\CgCg{\Cg{\delta}}{c}$.  We then have:

\begin{lemma}
\label{lemma:generatefqweak1}
Let $S \subset [\pi_g,\pi_g]$ be such that $[\pi_g,\pi_g]$ is $\pi_g$-normally generated by $S$.
Then $\cQ_g$ is spanned by the set of all $\cQ_g[\delta]$ with $\delta \in S$.
\end{lemma}
\begin{proof}
Since $[\pi_g,\pi_g]$ is $\pi_g$-normally generated by $S$, it follows that $\cQ_g$ is spanned by
\[\bigcup_{\delta \in S} \bigcup_{x \in \pi_g} \cQ_g[x \delta x^{-1}].\]
The point-pushing subgroup $\pi_g$ of $\Torelli_{g,1}$ acts on $\pi_g$ by conjugation, so since
$\Torelli_{g,1}$ acts trivially on $\cQ_g$ we have
\[\cQ_g[x \delta x^{-1}] = \cQ_g[\delta] \quad \text{for all $x \in \pi_g$ and $\delta \in [\pi_g,\pi_g]$}.\]
The lemma follows.
\end{proof}

\subsection{Fixing the first curve, II}
\label{section:fixfirst}

The set $S$ we will use in Lemma \ref{lemma:generatefqweak1} will consist
of simple closed separating curves $\delta \in [\pi_g,\pi_g]$.  For such $\delta$,
the subspace $\cQ_g[\delta]$ is spanned by the following elements:

\begin{lemma}
\label{lemma:generatefqdeltaweak}
Let $\delta \in [\pi_g,\pi_g]$ be a simple closed separating curve.  Then
$\cQ_g[\delta]$ is spanned by elements of the form $\CgCg{\Cg{\delta}}{\Cg{\eta,\lambda}^h}$ where
$\delta$ separates $\eta \in \pi_g$ from $\lambda \in \pi_g$ and $h \in H_{\Z}$.
\end{lemma}
\begin{proof}
Let $S$ and $T$ be the subsurfaces to the left and right of $\delta$, respectively.  The group
$\pi_g$ is generated by $\pi_1(S)$ and $\pi_1(T)$, and thus $[\pi_g,\pi_g]$ is the subgroup
of $\pi_g$ normally generated by $[\pi_1(S),\pi_1(S)]$ and $[\pi_1(T),\pi_1(T)]$ and $[\pi_1(S),\pi_1(T)]$.
It follows that $\cQ_g[\delta]$ is generated by the following three types of elements:
\begin{itemize}
\item Elements of the form $\CgCg{\Cg{\delta}}{\Cg{\eta_1,\eta_2}^h}$ with $\eta_1,\eta_2 \in \pi_1(S)$
and $h \in H_{\Z}$.  Since $\delta$ and $[\eta_1,\eta_2]$ are separated, it follows that
such $\CgCg{\Cg{\delta}}{\Cg{\eta_1,\eta_2}^h}$ are $0$.
\item Elements of the form $\CgCg{\Cg{\delta}}{\Cg{\lambda_1,\lambda_2}^h}$ with $\lambda_1,\lambda_2 \in \pi_1(T)$
and $h \in H_{\Z}$.  Since $\delta$ and $[\lambda_1,\lambda_2]$ are separated, it follows that
such $\CgCg{\Cg{\delta}}{\Cg{\lambda_1,\lambda_2}^h}$ are $0$.
\item Elements of the form $\CgCg{\Cg{\delta}}{\Cg{\eta,\lambda}^h}$ with $\eta \in \pi_1(S)$ and $\lambda \in \pi_1(T)$
and $h \in H_{\Z}$.\qedhere
\end{itemize}
\end{proof}

\subsection{Vanishing}  

We now prove certain classes in $\cQ_g$ vanish:

\begin{lemma}
\label{lemma:vanishingbasic}
Let $\delta \in [\pi_g, \pi_g]$ be a simple closed separating curve that separates $\mu \in \pi_g$ from $\lambda \in \pi_g$ 
and let $h \in H_{\Z}$ be arbitrary.  Assume that $\overline{\mu} = 0$ or $\overline{\lambda} = 0$.  
Then $\CgCg{\Cg{\delta}}{\Cg{\mu, \lambda}^h} = 0$.
\end{lemma}
\begin{proof}
The proofs for $\overline{\mu} = 0$ and $\overline{\lambda} = 0$ are similar, so we will give
the details for the case where $\overline{\mu} = 0$.  Since $\overline{\mu} = 0$, we have
$\mu \in [\pi_g,\pi_g]$ and thus $\Cg{\mu}$ is well-defined.  We have
\[\Cg{\mu,\lambda}^h = \Cg{\mu\lambda \mu^{-1}\lambda^{-1}}^h = \Cg{\mu} - \Cg{\mu}^{h + \lambda}.\]  
This implies that
\[\CgCg{\Cg{\delta}}{\Cg{\mu,\lambda}^h} = \CgCg{\Cg{\delta}}{\Cg{\mu}^h} - \CgCg{\Cg{\delta}}{\Cg{\mu}^{h + \lambda}}.\]  
The curves $\delta$ and $\mu$ are separated, so by Lemma \ref{lemma:coinvariantreidemeister} both of these
terms vanish.  The lemma follows.
\end{proof}

\subsection{Only homology matters}

Recall that our goal (Theorem \ref{theorem:part2theorem})
is to prove that the quotiented Reidemeister pairing $\fq\colon\cQ_g \rightarrow \Q[H_{\Z}]$ is
an injection whose image has finite codimension.  Let $\delta \in [\pi_g,\pi_g]$ be a simple closed separating curve
that separates $\eta \in \pi_g$ from $\lambda \in \pi_g$ and let $h \in H_{\Z}$, so
$\CgCg{\Cg{\delta}}{\Cg{\eta,\lambda}^h}$ is one of the generators for $\cQ_g[\delta]$ given by
Lemma \ref{lemma:generatefqdeltaweak}.  Lemma \ref{lemma:coinvariantreidemeister}
implies that
\[\fq(\CgCg{\Cg{\delta}}{\Cg{\eta,\lambda}^h}) = \GrpR{h} - \GrpR{h+\oeta} - \GrpR{h+\olambda} + \GrpR{h+\oeta+\olambda}.\]
This only depends on $\oeta$ and $\olambda$ and $h$, so we expect that 
$\CgCg{\Cg{\delta}}{\Cg{\eta,\lambda}^h} \in \cQ_g$ only depends on $\oeta$ and $\olambda$ and $h$.  The following 
shows that this expectation holds:

\begin{lemma}
\label{lemma:xwelldefined}
For $i=1,2$, let $\delta_i \in [\pi_g,\pi_g]$ be a simple closed separating curve that 
separates $\eta_i \in \pi_g$ from $\lambda_i \in \pi_g$.  Assume that
$\oeta_1 = \oeta_2$ and $\olambda_1 = \olambda_2$.  Let $h \in H_{\Z}$.  Then
$\CgCg{\Cg{\delta_1}}{\Cg{\eta_1,\lambda_1}^h} = \CgCg{\Cg{\delta_2}}{\Cg{\eta_2,\lambda_2}^h}$.
\end{lemma}
\begin{proof}
Set $x = \oeta_1 = \oeta_2$ and $y = \olambda_1 = \olambda_2$.  We start by proving a special case of the lemma:

\begin{claim}{1}
\label{claim:xwelldefined1}
If $\delta_1 = \delta_2$, then $\CgCg{\Cg{\delta_1}}{\Cg{\eta_1,\lambda_1}^h} = \CgCg{\Cg{\delta_2}}{\Cg{\eta_2,\lambda_2}^h}$.
\end{claim}
\begin{proof}[Proof of claim]
Let $\delta = \delta_1 = \delta_2$.  We will prove that
$\CgCg{\Cg{\delta}}{\Cg{\eta_1,\lambda_1}^h} = \CgCg{\Cg{\delta}}{\Cg{\eta_2,\lambda_1}^h}$.
The proof that this then equals $\CgCg{\Cg{\delta}}{\Cg{\eta_2,\lambda_2}^h}$ is identical.
Recall that $x = \oeta_1 = \oeta_2$ and $y = \olambda_1 = \olambda_2$.  
Using our commutator identities (Lemma \ref{lemma:commutatoridentities}), we have
\begin{align*}
\Cg{\eta_1,\lambda_1}^h - \Cg{\eta_2,\lambda_1}^h &= \Cg{\eta_1,\lambda_1}^h + \Cg{\eta_2^{-1},\lambda_1}^{h+x}
                                                  = \Cg{\eta_1 \eta_2^{-1},\lambda_1}^h.
\end{align*}                                      
We have $\eta_1 \eta_2^{-1} \in [\pi_g,\pi_g]$, so
\[\Cg{\eta_1 \eta_2^{-1},\lambda_1}^h = \Cg{(\eta_1 \eta_2^{-1}) \lambda_1 (\eta_1 \eta_2^{-1})^{-1} \lambda_1^{-1}} = \Cg{\eta_1 \eta_2^{-1}}^h - \Cg{\eta_1 \eta_2^{-1}}^{h + y}.\]
The curves $\delta$ and $\eta_1 \eta_2^{-1}$ are separated, so we conclude that
\begin{align*}
\CgCg{\Cg{\delta}}{\Cg{\eta_1,\lambda_1}^h} - \CgCg{\Cg{\delta}}{\Cg{\eta_2,\lambda_1}^h} &=
\CgCg{\Cg{\delta}}{\Cg{\eta_1 \eta_2^{-1}}^h - \Cg{\eta_1 \eta_2^{-1}}^{h + y}} \\
&= \CgCg{\Cg{\delta}}{\Cg{\eta_1 \eta_2^{-1}}^h} - \CgCg{\Cg{\delta}}{\Cg{\eta_1 \eta_2^{-1}}^{h+y}} \\
&= 0 - 0 = 0.\qedhere
\end{align*}
\end{proof}

We now turn to the general case.
If either $x=0$ or $y=0$, then we are done by Lemma \ref{lemma:vanishingbasic}, so assume that neither are $0$.  Write
$x = k x'$ and $y = \ell y'$ with $k,\ell \geq 1$ and $x',y' \in H_{\Z}$ primitive.  
Since $x'$ and $y'$ are primitive, for $i = 1,2$ we can choose
the following (see \cite{PutmanRealize}):
\begin{itemize}
\item in the component containing $\eta_i$ of $\Sigma_g$ cut open along $\delta_i$, a nonseparating simple closed curve $\eta'_i \in \pi_g$ with $\oeta'_i = x'$; and
\item in the component containing $\lambda_i$ of $\Sigma_g$ cut open along $\delta_i$, a nonseparating simple closed curve $\lambda'_i \in \pi_g$ with $\olambda'_i = y'$.
\end{itemize}
By freely homotoping $\eta'_i$ and $\lambda'_i$ to other based curves, we can assume that a regular neighborhood of the basepoint looks like this (where the key property is the cyclic order in which the curves enter and leave the basepoint):\\
\Figure{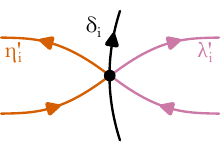}
For instance, such a homotopy might move $\eta'_i$ as follows:\\
\Figure{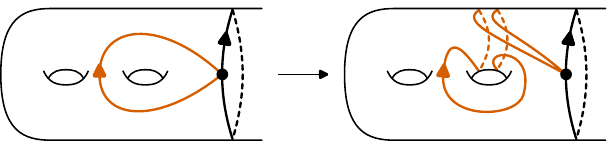}
For $i=1,2$, the homology classes of $\eta_i$ and $(\eta_i')^k$ are the same, and also the homology
classes of $\lambda_i$ and $(\lambda_i')^{\ell}$ are the same.  Using Claim \ref{claim:xwelldefined1},
we can therefore assume without loss of generality that $\eta_i = (\eta'_i)^k$ and $\lambda_i = (\lambda'_i)^{\ell}$.

Farb--Margalit's ``change of coordinates principle'' \cite[\S 1.3]{FarbMargalitPrimer} 
implies that there exists some $f \in \Mod_{g,1}$ with
$f(\eta'_1) = \eta'_2$ and $f(\lambda'_1) = \lambda'_2$.  In fact, using the argument from the proof
of \cite[Lemma 6.2]{PutmanCutPaste} we can actually find such an $f$ in $\Torelli_{g,1}$.
Since $\Torelli_{g,1}$ acts trivially on $\cQ_g$, we thus have
\begin{align*}
\CgCg{\Cg{\delta_1}}{\Cg{\eta_1,\lambda_1}^h} &= \CgCg{\Cg{\delta_1}}{\Cg{(\eta'_1)^k,(\lambda'_1)^{\ell}}^h} \\
                                             &= \CgCg{\Cg{f(\delta_1)}}{\Cg{(\eta'_2)^k,(\lambda'_2)^{\ell}}^h} \\
                                             &= \CgCg{\Cg{f(\delta_1)}}{\Cg{\eta_2,\lambda_2}^h}.
\end{align*}
We can therefore assume without loss of generality that $\eta_1 = \eta_2$ and $\lambda_1 = \lambda_2$.  When
doing this, we replace $\delta_1$ with $f(\delta_1)$.  We have therefore reduced the proof
to:

\begin{claim}{2}
\label{claim:xwelldefined2}
Assume that
\begin{itemize}
\item $\eta = (\eta')^k$ with $\eta' \in \pi_g$ a nonseparating simple closed curve; and
\item $\lambda = (\lambda')^{\ell}$ with $\lambda' \in \pi_g$ a nonseparating simple closed curve; and
\item for $i=1,2$, the simple closed separating curve $\delta_i \in [\pi_g,\pi_g]$ separates $\eta$ from $\lambda$.
\end{itemize}
Then $\CgCg{\Cg{\delta_1}}{\Cg{\eta,\lambda}^h} = \CgCg{\Cg{\delta_2}}{\Cg{\eta,\lambda}^h}$.
\end{claim}
\begin{proof}[Proof of claim]
We have
\[\CgCg{\Cg{\delta_1}}{\Cg{\eta,\lambda}^h} - \CgCg{\Cg{\delta_2}}{\Cg{\eta,\lambda}^h} = \CgCg{\Cg{\delta_1 \delta_2^{-1}}}{\Cg{\eta,\lambda}^h}.\]
We want to prove this vanishes.  To do this, it is enough to prove that $\delta_1 \delta_2^{-1}$ can be written as a product
of elements of $[\pi_g,\pi_g]$ that are separated from $[\eta,\lambda] \in [\pi_g,\pi_g]$.  The loops in question
look like those in the following figure:\\
\Figure{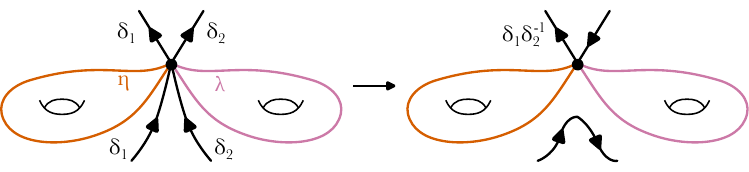}
We have not drawn the rest of the $\delta_i$ since while individually they are simple closed curves, they potentially intersect
each other, and the loop $\delta_1 \delta_2^{-1}$ potentially has self-intersections.  Let $U$ be a $3$-holed
sphere embedded in $\Sigma_g$ such that one of the boundary components of $U$ contains the basepoint, the loops
$\eta$ and $\lambda$ lie in $U$, and $\delta_1 \delta_2^{-1}$ only intersects $U$ in the basepoint:\\
\Figure{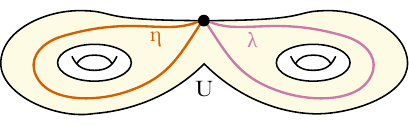}
Set $S = \Sigma_g \setminus \Interior(U)$, so $S \cong \Sigma_{g-2}^3$.  Regard $\pi_1(S)$ as a subgroup
of $\pi_g$, so $\delta_1 \delta_2^{-1} \in \pi_1(S)$.  Since the map $\HH_1(S) \rightarrow \HH_1(\Sigma_g)$
is injective, the intersection of $[\pi_g,\pi_g]$ with $\pi_1(S)$ is $[\pi_1(S),\pi_1(S)]$.  It
follows that $\delta_1 \delta_2^{-1} \in [\pi_1(S),\pi_1(S)]$.  

Recall our standing assumption
that $g \geq 4$ (Assumption \ref{assumption:genus}).  This implies that the genus of $S$ is positive.
It follows (see \cite[Lemma A.1]{PutmanCutPaste}) that $[\pi_1(S),\pi_1(S)]$ is generated by
based isotopy classes of simple closed separating curves that cut off one-holed tori.  We can therefore write $\delta_1 \delta_2^{-1}$
as a product of such curves.  These are all separated from $[\eta,\lambda]$, and we are done.
\end{proof}

This completes the proof of Lemma \ref{lemma:xwelldefined}.
\end{proof}

\subsection{Generators}
\label{section:generatorsx}

A {\em symplectic splitting} of $H_{\Z}$ is a decomposition $H_{\Z} = X \oplus Y$ that is orthogonal
with respect to the intersection form.  For a symplectic splitting $H_{\Z} = X \oplus Y$,
by work of Johnson \cite{JohnsonConj} we can find a simple closed separating curve
$\delta \in \pi_g$ such that if $S$ (resp.\ $T$) is the subsurface to the left (resp.\ right) of $\delta$, then
$\HH_1(S) = X$ and $\HH_1(T) = Y$.  We will say that $\delta$ {\em induces} the symplectic
splitting $H_{\Z} = X \oplus Y$.  For a simple closed separating curve $\delta \in [\pi_g,\pi_g]$,
we will write $H_{\Z} = X(\delta) \oplus Y(\delta)$ for the symplectic splitting induced by $\delta$.

Consider elements $\{x_1,\ldots,x_k\}$ and $\{y_1,\ldots,y_{\ell}\}$ of $H_{\Z}$.  We say that the $x_i$
and $y_j$ are {\em homologically separate} if there exists a symplectic splitting
$H_{\Z} = X \oplus Y$ with $x_i \in X$ and $y_j \in Y$ for all $i$ and $j$.

Let $x \in H_{\Z}$ and $y \in H_{\Z}$ be homologically separate.  By our discussion
above, we can find a simple closed separating curve $\delta$ with $x \in X(\delta)$
and $y \in Y(\delta)$.
Let $S$ and $T$ be the subsurfaces to the left and right of $\delta$, respectively,
so $X(\delta) = \HH_1(S)$ and $Y(\delta) = \HH_1(T)$.
We can find $\eta,\lambda \in \pi_g$ with $\oeta = x$ and $\olambda = y$ such that
$\delta$ separates $\eta$ from $\lambda$.  Indeed, choose $\eta$ lying in $S$
with $\oeta = x$ and $\lambda$ lying in $T$ with $\olambda = y$.
For $h \in H_{\Z}$, we define
\[X(h,x,y) = \CgCg{\Cg{\delta}}{\Cg{\eta,\lambda}^h} \in \cQ_g.\]
By Lemma \ref{lemma:xwelldefined}, this does not depend on any of our choices.  By Lemma \ref{lemma:coinvariantreidemeister}, this satisfies
\[\fq(X(h,x,y)) = \GrpR{h} - \GrpR{h+x} - \GrpR{h+y} + \GrpR{h+x+y}.\]

\subsection{Summary}
\label{section:summaryx}

The following summarizes what we have accomplished in this section:

\begin{lemma}
\label{lemma:summarygenweak}
Let $S \subset [\pi_g,\pi_g]$ be a set of simple closed separating curves such that $[\pi_g,\pi_g]$ is
$\pi_g$-normally generated by $S$.  Then $\cQ_g$ is spanned by
\[\bigcup_{\delta \in S} \Set{$X(h,x,y)$}{$x \in X(\delta)$ and $y \in Y(\delta)$ and $h \in H_{\Z}$}.\]
\end{lemma}
\begin{proof}
Immediate from Lemmas \ref{lemma:generatefqweak1} and \ref{lemma:generatefqdeltaweak} along
with the definition of $X(h,x,y)$.
\end{proof}

\section{Relations in the quotient}
\label{section:quotientrel}

We want to prove that the quotiented Reidemeister pairing $\fq\colon \cQ_g \rightarrow \Q[H_{\Z}]$
is injective.  For $x \in H_{\Z}$ and $y \in H_{\Z}$ homologically separate and $h \in H_{\Z}$, Lemma
\ref{lemma:coinvariantreidemeister} implies that
\[\fq(X(h,x,y)) = \GrpR{h} - \GrpR{h+x} - \GrpR{h+y} + \GrpR{h+x+y}.\]
Our calculations will use the following five relations.  It is enlightening to verify
that $\fq$ takes these to relations in $\Q[H_{\Z}]$.

\begin{lemma}[Vanishing relation]
\label{lemma:vanishingrelation}
For all $h,x,y \in H_{\Z}$ we have $X(h,x,0) = X(h,0,y) = 0$.
\end{lemma}
\begin{proof}
Immediate from Lemma \ref{lemma:vanishingbasic}.
\end{proof} 

\begin{lemma}[Symmetry relation\footnote{Though we will usually name the relations we are using, 
we will use the symmetry relation freely and without mention.}]
\label{lemma:symmetryrelation}
Let $h \in H_{\Z}$, and let $x \in H_{\Z}$ and $y \in H_{\Z}$ be homologically separate.  Then
$X(h,x,y) = X(h,y,x)$.
\end{lemma}
\begin{proof}
Let $H_{\Z} = X \oplus Y$ be a symplectic splitting with $x \in X$ and $y \in Y$.  Let $\delta$
be a simple closed separating curve inducing the splitting $H_{\Z} = X \oplus Y$, and let
$\eta,\lambda \in \pi_g$ be such that $\oeta = x$ and $\olambda = y$ and such that $\delta$
separates $\eta$ from $\lambda$.  We then have
\[X(h,x,y) = \CgCg{\Cg{\delta}}{\Cg{\eta,\lambda}^h}.\]
By definition, $\eta$ lies in the subsurface to the left of $\delta$ and $\lambda$ lies in the subsurface
to the right of $\delta$.  Reversing the orientation of $\delta$, we see that $\delta^{-1}$
separates $\lambda$ from $\eta$.  Lemma \ref{lemma:commutatoridentities} (commutator identities)
says that $\Cg{\lambda,\eta}^h = - \Cg{\eta,\lambda}^h$, so
\[X(h,y,x) = \CgCg{\Cg{\delta^{-1}}}{\Cg{\lambda,\eta}^h} = \CgCg{(-\Cg{\delta})}{(-\Cg{\eta,\lambda}^h)} = \CgCg{\Cg{\delta}}{\Cg{\eta,\lambda}^h} = X(h,x,y).\qedhere\] 
\end{proof}

\begin{lemma}[Additivity relation]
\label{lemma:additivityrelation}
Let $h \in H_{\Z}$, and let $x_1,x_2 \in H_{\Z}$ and $y \in H_{\Z}$ be homologically separate.  Then
\[X(h,x_1+x_2,y) = X(h,x_1,y) + X(h+x_1,x_2,y).\]
\end{lemma}
\begin{proof}
Let $H_{\Z} = X \oplus Y$ be a symplectic splitting with $x_1,x_2 \in X$ and $y \in Y$.  Let
$\delta \in \pi_g$ be a simple closed separating curve inducing the splitting $H_{\Z} = X \oplus Y$.
Let $\eta_1,\eta_2,\lambda \in \pi_g$ be such that $\oeta_i = x_i$ and $\olambda = y$ and such
that $\delta$ separates $\{\eta_1,\eta_2\}$ from $\lambda$.  We then have
\[X(h,x_1+x_2,y) = \CgCg{\Cg{\delta}}{\Cg{\eta_1 \eta_2,\lambda}^h}.\]
Using Lemma \ref{lemma:commutatoridentities} (commutator identities), this equals
\[\CgCg{\Cg{\delta}}{(\Cg{\eta_1,\lambda}^h + \Cg{\eta_2,\lambda}^{h+x_1})} = X(h,x_1,y) + X(h+x_1,x_2,y).\qedhere\]
\end{proof}

\begin{lemma}[Inverse relation]
\label{lemma:inverserelation}
Let $h \in H_{\Z}$, and let $x \in H_{\Z}$ and $y \in H_{\Z}$ be homologically separate.  Then
$X(h,-x,y) = -X(h-x,x,y)$.
\end{lemma}
\begin{proof}
Let $H_{\Z} = X \oplus Y$ be a symplectic splitting with $x \in X$ and $y \in Y$.  Let $\delta$
be a simple closed separating curve inducing the splitting $H_{\Z} = X \oplus Y$, and let
$\eta,\lambda \in \pi_g$ be such that $\oeta = x$ and $\olambda = y$ and such that $\delta$
separates $\eta$ from $\lambda$.  Using Lemma \ref{lemma:commutatoridentities} (commutator identities), we then have
\[X(h,-x,y) = \CgCg{\Cg{\delta}}{\Cg{\eta^{-1},\lambda}^h} = \CgCg{\Cg{\delta}}{(-\Cg{\eta,\lambda}^{h-x})} = -X(h-x,x,y).\qedhere\]
\end{proof}

\begin{lemma}[Cube relation\footnote{We call this the cube relation since 
$\fq$ takes the terms it to the points in $\Q[H_{\Z}]$ forming a ``cube'' with a corner
at $h$ in the directions of $x$, $y$, and $k$, namely $\{h, h+x, h+y, h+k, h+x+y, h+x+k, h+y+k, h+x+y+k\}$.}]
\label{lemma:cuberelation}
Let $h \in H_{\Z}$, and let $x,k \in H_{\Z}$ and $y \in H_{\Z}$ be homologically separate.  
Then\footnote{We think of this relation as relating $X(h,x,y)$ to $X(h+k,x,y)$, with the blue terms
$X(h,k,y)$ and $X(h+x,k,y)$ being the ``error''.}
$X(h+k,x,y) = X(h,x,y) \blue{-X(h,k,y) + X(h+x,k,y)}$
\end{lemma}
\begin{proof}
We apply the additivity relation (Lemma \ref{lemma:additivityrelation}) in two ways:
\begin{align*}
X(h,k+x,y) &= X(h,k,y) + X(h+k,x,y) \\
X(h,x+k,y) &= X(h,x,y) + X(h+x,k,y).
\end{align*}
Comparing these, we see that $X(h,k,y) + X(h+k,x,y) = X(h,x,y) + X(h+x,k,y)$.  Rearranging this
gives the desired relation.
\end{proof}

\section{A specific generating set for the quotient}
\label{section:quotientgen}

Choose a symplectic basis $\fB = \{a_1,b_1,\ldots,a_g,b_g\}$ for $H_{\Z}$.  This basis will be
fixed for the remainder of Part \ref{part:2}.  Let $\omega(-,-)$ be the algebraic intersection pairing
on $H_{\Z}$ and let $\perp$ be the orthogonal complement with respect to $\omega(-,-)$.  
Define $\cV = \cV_1 \cup \cV_2 \cup \cV_3$, where $\cV_i$ consists of
all $X(h,x,y)$ with $h \in H_{\Z}$ arbitrary and $x,y \in H_{\Z}$ nonzero satisfying:
\begin{itemize}
\item[$(\cV_1)$:\ \ ] $x \in \Span{a_d,b_d}$ and $y \in \Span{a_d,b_d}^{\perp}$ for some $1 \leq d \leq g$. 
\item[$(\cV_2)$:\ \ ] $x \in \{a_d,b_d\}$ and $y = x + z$ with $z \in \{\pm a_e,\pm b_e\}$ for some
distinct $1 \leq d,e \leq g$.
\item[$(\cV_3)$:\ \ ] the pair $(x,y)$ is either $(a_d+a_e,b_d-b_e)$ or
$(a_d+b_e,b_d+a_e)$ for some distinct $1 \leq d, e \leq g$.
Note that in both cases the elements
$x$ and $y$ are homologically separate and satisfy $\omega(x,y) = 0$. 
\end{itemize}

\begin{remark}
In $\cV_3$, we do not include $(a_d-a_e,b_d+b_e)$ or $(a_d-b_e,b_d-a_e)$.  It is enlightening
to go through the proof of Lemma \ref{lemma:generateq} below to see why they are not needed to generate $\cQ_g$.
\end{remark}

Our main result about $\cV$ is:

\begin{lemma}
\label{lemma:generateq}
The vector space $\cQ_g$ is spanned by $\cV$.
\end{lemma}
\begin{proof}
By Lemma \ref{lemma:summarygenweak}, it is enough to find a set $S \subset [\pi_g,\pi_g]$ of simple closed separating
curves that $\pi_g$-normally generate $S$ such that for all $\delta \in S$ we have:
\begin{align*} 
\tag{$\dagger$}\label{eqn:deltagood} &\text{\space\space\space letting $H_{\Z} = X(\delta) \oplus Y(\delta)$ be the induced symplectic splitting, each
$X(h,x,y)$ with} \\
&\text{\space\space\space $x \in X(\delta)$ and $y \in Y(\delta)$ and $h \in H_{\Z}$ lies in the span of $\cV$.}
\end{align*} 
Let $B = \{\alpha_1,\beta_1,\ldots,\alpha_g,\beta_g\}$ be the following standard generating set for
$\pi_g$:\\
\Figure{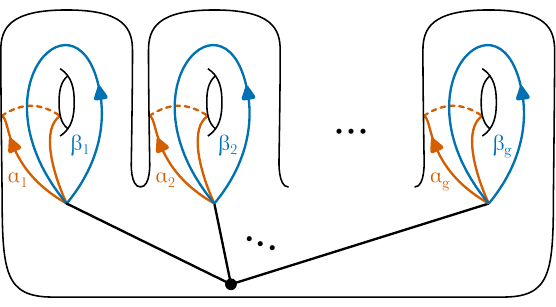}
Recalling that $\fB = \{a_1,b_1,\ldots,a_g,b_g\}$, we can choose $B$ such that
$\oalpha_i = a_i$ and $\obeta_i = b_i$ for all $1 \leq i \leq g$.  Set
\[S' = \Set{$[\alpha_d,\beta_d]$}{$1 \leq d \leq g$} \cup \Set{$[\alpha_d,\alpha_e], [\alpha_d,\beta_e], [\beta_d,\alpha_e], [\beta_d,\beta_e]$}{$1 \leq d < e \leq g$}.\]
The group $[\pi_g,\pi_g]$ is $\pi_g$-normally generated by the elements of $S'$.
Unfortunately, most elements of $S'$ are not simple closed separating curves.  Indeed, the $[\alpha_d,\beta_d] = \alpha_d \beta_d \alpha_d^{-1} \beta_d^{-1}$ are
the only ones that are:\\
\Figure{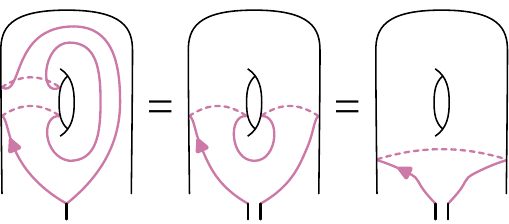}
More generally, an {\em $ab$-pair} of curves consists of $\zeta,\eta \in \pi_g$ that are simple closed curves that only
intersect at the basepoint and have algebraic intersection number $\pm 1$.  For instance, the curves $\alpha_d,\beta_d \in \pi_g$
form an $ab$-pair.  For an $ab$-pair of curves $\zeta$ and $\eta$,
the commutator $[\zeta,\eta] \in \pi_g$ is a simple closed separating curve.

To fix the above issue, we will do the following: for each $\delta' \in S'$, we will write $\delta'$ as a 
product of simple closed separating curves $\delta$ such that a conjugate of $\delta$ satisfies \eqref{eqn:deltagood}.  The
desired $S$ will consist of the \eqref{eqn:deltagood}-satisfying conjugates of all the $\delta$ that appear in these products.  During this,
we will freely use the relations from \S \ref{section:quotientrel}.

We start with $\delta' = [\alpha_d,\beta_d]$ for some $1 \leq d \leq g$.  Since $\alpha_d$ and $\beta_d$ form an $ab$-pair,
$\delta'$ is already a simple closed separating curve, so
for our product we take $\delta = \delta'$.  We must show that $\delta$ satisfies \eqref{eqn:deltagood}.  We have\footnote{It is
annoying that $\Span{a_d,b_d}$ comes second, but this is forced by our convention that $X(\delta)$ is the homology of the subsurface
to the left of $\delta$.}
\[H_{\Z} = X(\delta) \oplus Y(\delta) = \Span{a_d,b_d}^{\perp} \oplus \Span{a_d,b_d}.\]
Each $X(h,x,y)$ with $x \in X(\delta)$ and $y \in Y(\delta)$ and $h \in H_{\Z}$ is either\footnote{This holds when $x=0$ or $y=0$; see
the vanishing relation (Lemma \ref{lemma:vanishingrelation}).} $0$
or an element of $\cV_1$, verifying \eqref{eqn:deltagood}.

There are now four remaining cases: for $1 \leq d < e \leq g$, we either have $\delta' = [\alpha_d,\alpha_e]$ or
$\delta' = [\alpha_d,\beta_e]$ or $\delta' = [\beta_d,\alpha_e]$ or $\delta' = [\beta_d,\beta_e]$.  All four cases
are handled similarly, so we will give full details for $\delta' = [\alpha_d,\alpha_e]$ 
and then sketch the remaining cases.

We want to write $[\alpha_d,\alpha_d]$ as a product of simple closed separating curves.  The intuition guiding
our calculation is that the commutator bracket is very similar to an alternating bilinear pairing.  Because of this, we should
be able to write $[\alpha_d,\alpha_e]$ as a product of terms involving $[\alpha_d \beta'_e,\alpha_e]$ and 
$[\alpha_e,\beta'_e] = [\beta'_e,\alpha_e]^{-1}$
for any $\beta'_e \in \pi_g$.  Choosing $\beta'_e$ as in the following figure, the curves $\alpha_d \beta'_e$ and $\alpha_e$ (resp.\ $\alpha_e$ and $\beta'_e$) form
an $ab$-pair, so $[\alpha_d \beta'_e,\alpha_e]$ (resp.\ $[\alpha_e,\beta'_e]$) is a simple closed
separating curve:\\
\Figure{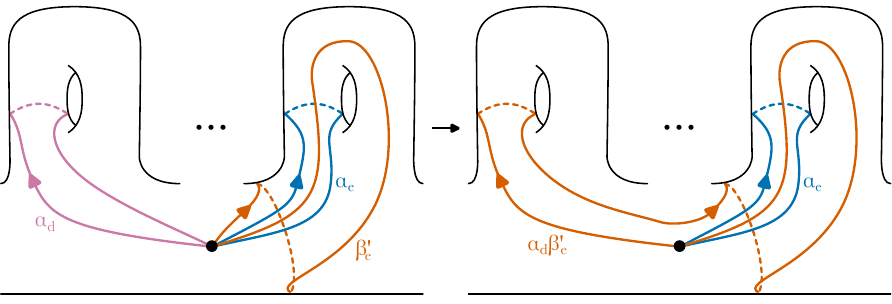}
The desired formula for $\delta' = [\alpha_d,\alpha_e]$ is\footnote{Recall from \S \ref{section:conjugationcommutator} that for $G$ a group and $c, d \in G$, we use the notation
$\precon{d}{c} = d c d^{-1}$ and $[c,d] =c d c^{-1} d^{-1}$.}
\[\delta'=[\alpha_d,\alpha_e] = \precon{\alpha_d}{[\alpha_e,\beta'_e]} [\alpha_d \beta'_e,\alpha_e].\]
We must check that conjugates of $\delta = \precon{\alpha_d}{[\alpha_e,\beta'_e]}$ and 
$\delta = [\alpha_d \beta'_e,\alpha_e]$ satisfy \eqref{eqn:deltagood}.  We will check that in fact
$\delta = [\alpha_e,\beta'_e]$ and $\delta = [\alpha_d \beta'_e,\alpha_e]$ satisfy \eqref{eqn:deltagood}.

The curves $\alpha_e,\beta'_e \in \pi_g$ form an $ab$-pair, so
$[\alpha_e,\beta'_e]$ is a simple closed separating curve.  We also have $\obeta'_e = \obeta_e = b_e$.  The
exact same argument we used above for $[\alpha_e,\beta_e]$ now also verifies \eqref{eqn:deltagood} for $\delta = [\alpha_e,\beta'_e]$.

Now consider $\delta = [\alpha_d \beta'_e,\alpha_e]$.  Again, since $\alpha_d \beta'_e$ and $\alpha_e$ form an $ab$-pair,
this is a simple closed separating curve.
The homology classes of $\{\alpha_d \beta'_e,\alpha_e\}$ are $\{a_d+b_e,a_e\}$, so
\begin{align*}
H_{\Z} &= X(\delta) \oplus Y(\delta) = \Span{a_d+b_e,a_e}^{\perp} \oplus \Span{a_d+b_e,a_e} \\
&= \SpanSet{$a_d,b_d+a_e,a_r,b_r$}{$1 \leq r \leq g$, $r \neq d,e$} \oplus \Span{a_d+b_e,a_e}.
\end{align*}
Using our relations, each $X(h,x,y)$ with $x \in X(\delta)$ and $y \in Y(\delta)$
and $h \in H_{\Z}$ can be written as a linear combination of elements of the form $X(h',x',y')$ with
\[x' \in \Set{$a_d,b_d+a_e,a_r,b_r$}{$1 \leq r \leq g$, $r \neq d,e$} \quad \text{and} \quad y' \in \{a_d+b_e,a_e\} \quad \text{and} \quad h' \in H_{\Z}.\]
These are either elements of $\cV_1$, or fall into one of the following cases:
\begin{itemize}
\item $X(h',a_d,a_d+b_e)$, which lies in $\cV_2$.
\item $X(h',b_d+a_e,a_e) = X(h',a_e,a_e+b_d)$, which lies in $\cV_2$.
\item $X(h',b_d+a_e,a_d+b_e) = X(h',a_d+b_e,b_d+a_e)$, which lies in $\cV_3$.
\end{itemize}
This completes the proof of \eqref{eqn:deltagood} for $\delta = [\alpha_d \beta'_e,\alpha_e]$, and thus verifies
what we must show for $\delta' = [\alpha_d,\alpha_e]$.

We must also handle $\delta' \in \{[\alpha_d,\beta_e],[\beta_d,\alpha_e],[\beta_d,\beta_d]\}$.  These are all
similar to $[\alpha_d,\alpha_e]$:
\begin{itemize}
\item For $\delta' = [\alpha_d,\beta_e]$, use $[\alpha_d,\beta_e] = \precon{\alpha_d}{[\beta_e,\alpha_e]} [\alpha_d \alpha_e, \beta_e]$.
\item For $\delta' = [\beta_d,\alpha_e]$, use $[\beta_d,\alpha_e] = [\beta_d,\alpha_e \alpha_d] \precon{\alpha_e}{[\alpha_d,\beta_d]}$.
\item For $\delta' = [\beta_d,\beta_e]$, use $[\beta_d,\beta_e] = \precon{\beta_d}{[\beta_e,\alpha'_e]} [\beta_d \alpha'_e,\beta_e]$
where $\alpha'_e$ is as shown in the following figure:
\end{itemize}
\Figure{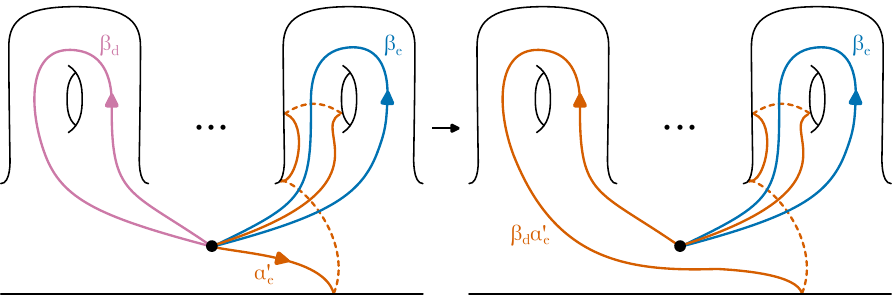}
This completes the proof of the lemma.
\end{proof}

\section{Calculation of the quotient, outline}
\label{section:calcquotientoutline}

Recall from \S \ref{section:part2intro} that our goal in this part of the paper is to prove the following:

\newtheorem*{theorem:part2theorem2}{Theorem \ref{theorem:part2theorem2}}
\begin{theorem:part2theorem2}
Let $\fq\colon \cQ_g \rightarrow \Q[H_{\Z}]$ be the quotiented Reidemeister
pairing.  Then $\fq$ is an injection whose image has finite codimension.
\end{theorem:part2theorem2}

In the previous section, we constructed a generating set $\cV = \cV_1 \cup \cV_2 \cup \cV_3$ for $\cQ_g$.  
In the next three sections, we use our generating set to prove Theorem \ref{theorem:part2theorem2}.  The outline is:
\begin{itemize}
\item In \S \ref{section:calcquotient1}, we prove that the restriction of $\fq$ to $\Span{\cV_1}$ is an injection.
\item Now define $\cQ_{g/1} = \cQ_g/\Span{\cV_1}$.  There is an induced map
\[\fq_1\colon \cQ_{g/1} \rightarrow \Q[H_{\Z}] / \Span{\fq(\cV_1)}.\]
Let $\cV_{2/1}$ be the image of $\cV_{2/1}$ in $\cQ_{g/1}$.   In \S \ref{section:calcquotient2}, we
prove that the restriction of $\fq_1$ to $\Span{\cV_{2/1}}$ is an injection.
\item Finally, define
$\cQ_{g/2} = \cQ_g/\Span{\cV_1,\cV_2}$.  There is an induced map
\[\fq_{2}\colon \cQ_{g/2} \rightarrow \Q[H_{\Z}] / \Span{\fq(\cV_1), \fq(\cV_2)}.\]
The vector space $\cQ_{g/2}$ is spanned by the image of $\cV_3$, and in \S \ref{section:calcquotient3} we prove
that $\fq_{2}$ is an injection.
\end{itemize}
Together the above will imply that $\fq$ is an injection.  To make the calculations possible, we will also have to
control the quotients
\[\Q[H_{\Z}] / \Span{\fq(\cV_1)} \quad \text{and} \quad \Q[H_{\Z}] / \Span{\fq(\cV_1), \fq(\cV_2)}.\]
What we will show is that they can be identified with subspaces $\Q[S]$ of $\Q[H_{\Z}]$ associated to subsets\footnote{These $S$ are just subsets of $H_{\Z}$.  They are
not closed under addition.  The notation $\Q[S]$ simply means the set of formal $\Q$-linear combinations of elements of $S$.} $S \subset H_{\Z}$:
\begin{itemize}
\item $\Q[H_{\Z}] / \Span{\fq(\cV_1)}$ will be identified with $\Q[\cup_{i=1}^g \Span{a_i,b_i}]$; and
\item $\Q[H_{\Z}] / \Span{\fq(\cV_1),\fq(\cV_2)}$ will be identified with $\Q[\cup_{i=1}^g \{0,a_i,b_i,a_i+b_i\}]$.
\end{itemize}
The second identification implies that $\Image(\fq)$ has finite codimension, completing the proof of Theorem \ref{theorem:part2theorem2}.

\section{Calculation of the quotient I: the first set of generators}
\label{section:calcquotient1}

We start with the first step from the outline in \S \ref{section:calcquotientoutline}.  Recall that
the generating set $\cV$ depends on a fixed symplectic basis $\fB = \{a_1,b_1,\ldots,a_g,b_g\}$
for $H_{\Z}$, and $\cV_1$ is the set of all $X(h,x,y)$ with $h \in H_{\Z}$ arbitrary
and $x,y \in H_{\Z}$ nonzero such that $x \in \Span{a_d,b_d}$ and $y \in \Span{a_d,b_d}^{\perp}$
for some $1 \leq d \leq g$.  Our goal in
this section is to prove that the restriction of the quotiented Reidemeister pairing
$\fq\colon \cQ_g \rightarrow \Q[H_{\Z}]$ to $\Span{\cV_1}$ is injective and to identify
the quotient of $\Q[H_{\Z}]$ by $\Span{\fq(\cV_1)}$.  

\subsection{A smaller generating set}

Let $\cW_1$ be the set of all $X(h,x,y) \in \cV_1$ such that $h=0$ and such that
$x \in \Span{a_d,b_d}$ and $y \in \Span{a_{d+1},b_{d+1},\ldots,a_g,b_g}$ for some $1 \leq d \leq g-1$.
Our first order of business is to prove that $\cW_1$ spans the same subspace of $\cQ_g$ as $\cV_1$:

\begin{lemma}
\label{lemma:smallergenset1}
Letting the notation be as above, we have $\Span{\cW_1} = \Span{\cV_1}$.
\end{lemma}
\begin{proof}
Define $\cV_1^1$ and $\cV_1^2$ with $\cW_1 \subset \cV_1^2 \subset \cV_1^1 \subset \cV_1$ as follows.
Let $\cV_1^1$ be the set of all $X(h,x,y) \in \cV_1$ such that $x \in \Span{a_d,b_d}$ and $y \in \Span{a_{d+1},b_{d+1},\ldots,a_g,b_g}$
for some $1 \leq d \leq g-1$.  Let
$\cV_1^2$ be the set of all $X(h,x,y) \in \cV_1$ such that $x \in \Span{a_d,b_d}$ and $y \in \Span{a_{d+1},b_{d+1},\ldots,a_g,b_g}$ and $h \in \Span{a_d,b_d,\ldots,a_g,b_g}$
for some $1 \leq d \leq g-1$.

\begin{step}{1}
\label{step:smallergenset1.1}
We prove that $\Span{\cV_1^1} = \Span{\cV_1}$.
\end{step}

Consider some $X(h,x,y) \in \cV_1$.  We must show that $X(h,x,y)$ is in the span of $\cV_1^1$.  Let $1 \leq d \leq g$
be such that $x \in \Span{a_d,b_d}$ and $y \in \Span{a_d,b_d}^{\perp}$.  Write
\[y = y_1 + \cdots + y_g \quad \text{with $y_i \in \Span{a_i,b_i}$},\]
so $y_d = 0$.  Using the additivity relation (Lemma \ref{lemma:additivityrelation}), we see that
$X(h,x,y) = X(h,x,y_1+\cdots+y_g)$ equals the following, where the colored term vanishes since
$y_d=0$:
\[X(h,x,y_1) + \cdots + \orange{X(h+y_1+\cdots+y_{d-1},x,y_d)} + \cdots + X(h+y_1+\cdots+y_{g-1},x,y_g).\]
For $1 \leq i \leq g$ with $i \neq d$, the symmetry relation (Lemma \ref{lemma:symmetryrelation}) says that
\[X(h+y_1+\cdots+y_{i-1},x,y_i) = X(h+y_1+\cdots+y_{i-1},y_i,x).\]
From this, we see that whether or not $i<d$ or $i>d$ this term lies in $\cV_1^1$.  The step follows.

\begin{step}{2}
\label{step:smallergenset1.2}
We prove that $\Span{\cV_1^2} = \Span{\cV_1^1}$.
\end{step}

Consider some $X(h,x,y) \in \cV_1^1$.  We must show that $X(h,x,y)$ is in the span of $\cV_1^2$.  Let $1 \leq d \leq g-1$
be such that $x \in \Span{a_d,b_d}$ and $y \in \Span{a_{d+1},b_{d+1},\ldots,a_g,b_g}$.  If $d=1$ there is nothing
to prove, so assume that $d>1$.  Write
\[h = h_1 + \cdots + h_{d-1} + h' \quad \text{with $h_i \in \Span{a_i,b_i}$ and $h' \in \Span{a_d,b_d,\ldots,a_g,b_g}$}.\]
Applying the cube relation (Lemma \ref{lemma:cuberelation}) repeatedly, we see that
\begin{align*}
X(h_1+\cdots+h_{d-1}+h',x,y) &= X(h_2+\cdots+h_{d-1}+h',x,y) \blue{-X(h_2+\cdots+h_{d-1}+h',h_1,y)} \\
                                 &\ \ \ \ \blue{+X(h_2+\cdots+h_{d-1}+h'+x,h_1,y)} \\
X(h_2+\cdots+h_{d-1}+h',x,y) &= X(h_3+\cdots+h_{d-1}+h',x,y) \blue{-X(h_3+\cdots+h_{d-1}+h',h_2,y)} \\
                                 &\ \ \ \ \blue{+X(h_3+\cdots+h_{d-1}+h'+x,h_2,y)} \\
                             &\vdotswithin{=} \\
X(h_{d-1}+h',x,y)            &= X(h',x,y) \blue{-X(h',h_{d-1},y) +X(h'+x,h_{d-1},y)}.
\end{align*}
The element $X(h',x,y)$ and all the blue terms are in $\cV_1^2$, so $X(h,x,y) \in \Span{\cV_1^2}$, as desired.

\begin{step}{3}
\label{step:smallergenset1.3}
We prove that $\Span{\cW_1} = \Span{\cV_1}$.
\end{step}

Assume this is false.  By Step \ref{step:smallergenset1.2}, there must be some $X(h,x,y) \in \cV_1^2$ with
$X(h,x,y) \notin \cW_1$.  Choose $X(h,x,y) \in \cV_1^2$ with $X(h,x,y) \notin \Span{\cW_1}$ in the following way:
\begin{itemize}
\item Elements $X(h,x,y) \in \cV_1^2$ have
$x \in \Span{a_d,b_d}$ and $y \in \Span{a_{d+1},b_{d+1},\ldots,a_g,b_g}$
and $h \in \Span{a_d,b_d,\ldots,a_g,b_g}$ for some $1 \leq d \leq g-1$.  They do not lie in
$\cW_1$ precisely when $h \neq 0$.  In that case, for some $d \leq d' \leq g$ we can write
\[h = h_{d'} + \cdots + h_{g} \quad \text{with $h_i \in \Span{a_i,b_i}$ and $h_{d'} \neq 0$}.\]
Among all the $X(h,x,y) \in \cV_1^2$ with $X(h,x,y) \notin \Span{\cW_1}$, choose the one with $d'$ as large as possible.
\end{itemize}
There are now two cases.  The first is $d' = d$.  By the additivity relation (Lemma \ref{lemma:additivityrelation}),
\[\blue{X(h_{d+1}+\cdots+h_g,h_d+x,y)} = \blue{X(h_{d+1}+\cdots+h_g,h_d,y)} + X(h_d+\cdots+h_g,x,y).\]
Since $d'=d$ is as large as possible, both blue terms are in the span of $\cW_1$.  This implies that
$X(h_d+\cdots+h_g,x,y)$ is also in the span of $\cW_1$, a contradiction.

The second case is $d+1 \leq d' \leq g$.
The additivity relation (Lemma \ref{lemma:additivityrelation}) implies that
\[\blue{X(h_{d'+1}+\cdots+h_g,x,h_{d'}+y)} = \blue{X(h_{d'+1}+\cdots+h_g,x,h_{d'})} + X(h_{d'}+\cdots+h_g,x,y).\]
Since $d'$ is as large as possible, both blue terms are in the span of $\cW_1$.  This implies that
$X(h_d+\cdots+h_g,x,y)$ is also in the span of $\cW_1$, a contradiction.
\end{proof}

\subsection{Main result}

We now prove the main result of this section:

\begin{proposition}
\label{proposition:calcquotient1}
The restriction of $\fq$ to $\Span{\cV_1}$ is injective, and\hspace{2pt}\footnote{Here $\Q[\bigcup_{i=1}^g \Span{a_i,b_i}]$ is the
set of formal $\Q$-linear combinations of $\GrpR{h}$ for $h \in H_{\Z}$ an element such that $h \in \Span{a_i,b_i}$ for some
$1 \leq i \leq g$.  This is not a disjoint union since all these terms contain $0 \in H_{\Z}$.}
\[\Q[H_{\Z}] = \Span{\fq(\cV_1)} \oplus \Q[\bigcup_{i=1}^g \Span{a_i,b_i}].\]
\end{proposition}
\begin{proof}
Let $\hfq$ be the composition
\[\begin{tikzcd}
\cQ_g \arrow{r}{\fq} & \Q[H_{\Z}] \arrow[two heads]{r} & \Q[H_{\Z}] / \Q[\bigcup_{i=1}^g \Span{a_i,b_i}]
\end{tikzcd}\]
and let $\cQ_g(1) = \Span{\cV_1}$.  We will prove that the restriction of $\hfq$ to $\cQ_g(1)$ is an isomorphism.  

Let $Q$ be the codomain of $\hfq$.  We construct an inverse $\fp\colon Q \rightarrow \cQ_g(1)$ to
$\hfq|_{\cQ_g(1)}$ as follows.  We can identify $Q$ with the set of formal $\Q$-linear combinations
of terms of the form $\GrpR{z_{d_1} + \cdots + z_{d_r}}$ where:
\begin{itemize}
\item $1 \leq d_1 < d_2 < \cdots < d_r \leq g$ with $r \geq 2$; and
\item for $1 \leq i \leq r$, the term $z_{d_i}$ is a nonzero element of $\Span{a_{d_i},b_{d_i}}$.
\end{itemize}
Define
\[\fp(\GrpR{z_{d_1} + \cdots + z_{d_r}}) = \sum_{j=1}^{r-1} X(0,z_{d_i},z_{d_{i+1}}+\cdots+z_{d_{r}}).\]
To see that this is an inverse to $\hfq|_{\cQ_g(1)}$, we must check two things:

\begin{unnumberedclaim}
For $\GrpR{z_{d_1}+\cdots+z_{d_r}}$ as above, we have 
\[\hfq(\fp(\GrpR{z_{d_1}+\cdots+z_{d_r}})) = \GrpR{z_{d_1}+\cdots+z_{d_r}}.\]
\end{unnumberedclaim}

For this, we calculate as follows.  Terms of $\Q[\bigcup_{i=1}^g \Span{a_i,b_i}]$
are in blue, and vanish in $Q$:
\begin{align*}
\fq(\fp(\GrpR{z_{d_1}+\cdots+z_{d_r}})) &=  \fq\left(\sum_{i=1}^{r-1} X(0,z_{d_i}, z_{d_{i+1}} + \cdots + z_{d_r})\right) \\
                                         &= \sum_{i=1}^{r-1} \blue{\GrpR{0}} - \blue{\GrpR{z_{d_i}}} - \GrpR{z_{d_{i+1}}+\cdots+z_{d_r}} + \GrpR{z_{d_i} + \cdots+z_{d_r}}.
\end{align*}
Deleting the indicated blue terms gives a telescoping sum adding up to $\GrpR{z_{d_1}+\cdots+z_{d_r}}-\blue{\GrpR{z_{d_r}}}$.  Deleting
this final blue term gives $\GrpR{z_{d_1}+\cdots+z_{d_r}}$, as desired.

\begin{unnumberedclaim}
The composition $\fp \circ \hfq$ is the identity on $\cQ_g(1)$. 
\end{unnumberedclaim}

We check this on the generating set
$\cW_1$ given by Lemma \ref{lemma:smallergenset1}.  An element of $\cW_1$ can be written as $X(0,x_{e_1},x_{e_2}+\cdots+x_{e_s})$ where:
\begin{itemize}
\item $1 \leq e_1 < \cdots < e_s \leq g$ with $s \geq 2$; and
\item for $1 \leq i \leq s$, the term $x_{e_{i}}$ is a nonzero element of $\Span{a_{e_i},b_{e_i}}$.
\end{itemize}
If $s = 2$, then when we apply $\hfq$ to this the only term that survives
in the quotient $Q$ is $\GrpR{x_{e_1}+x_{e_2}}$, which is taken by $\fp$ back to $X(0,x_{e_1},x_{e_2})$.  If $s \geq 3$, then
when we apply $\hfq$ to this two terms survive in the quotient $Q$:
\[\hfq(X(0,x_{e_1},x_{e_2}+\cdots+x_{e_s})) = -\GrpR{x_{e_2}+\cdots+x_{e_s}} + \GrpR{x_{e_1}+\cdots+x_{e_s}}.\]
Applying $\fp$ to this gives
\[-\sum_{i=2}^{s-1} X(0,x_{e_i},x_{e_{i+1}}+\cdots+x_{e_{s}}) + \sum_{i=1}^{s-1} X(0,x_{e_i},x_{e_{i+1}}+\cdots+x_{e_{s}}).\]
All terms cancel except $X(0,x_{e_1},x_{e_2}+\cdots+x_{e_s})$, as desired.
\end{proof}

\section{Calculation of the quotient II: the second set of generators}
\label{section:calcquotient2}

We now move on the set $\cV_2$ of generators, which we recall consists of all $X(h,x,x+y)$ with
$h \in H_{\Z}$ arbitrary such that for some distinct $1 \leq d,e \leq g$ we have
$x \in \{a_d,b_d\}$ and $y \in \{\pm a_e,\pm b_e\}$.
As notation, define $\cQ_{g/1} = \cQ_g/\Span{\cV_1}$.  For any generator $X(h,x,y)$ of $\cQ_g$, 
let $X_{1}(h,x,y)$ be its image in $\cQ_{g/1}$.

\subsection{Y-elements}

We start by proving that for $X(h,x,x+y) \in \cV_2$, its image $X_{1}(h,x,x+y) \in \cQ_{g/1}$ does not
depend on $y$:

\begin{lemma}
\label{lemma:ywelldefined}
Let $h \in H_{\Z}$ and $x \in \{a_d,b_d\}$ for some $1 \leq d \leq g$.  For some $1 \leq e,e' \leq g$ with $e,e' \neq d$, let
$y \in \{\pm a_e,\pm b_e\}$ and $y' \in \{\pm a_{e'},\pm b_{e'}\}$.  Then $X_{1}(h,x,x+y) = X_{1}(h,x,x+y')$.
\end{lemma}
\begin{proof}
Recall our standing assumption that $g \geq 4$ (Assumption \ref{assumption:genus}).  Because of this,
it is enough to prove the claim when $e \neq e'$.  Using the additivity relation (Lemma \ref{lemma:additivityrelation}), we
have
\begin{align*}
X(h,x,x+y+y') &= X(h,x,x+y)  + \blue{X(h+x+y,x,y')}, \\
X(h,x,x+y'+y) &= X(h,x,x+y') + \blue{X(h+x+y',x,y)}.
\end{align*}
The blue terms here lie in $\cV_1$ and thus die in $\cQ_{g/1}$.  The lemma follows.
\end{proof}

Using this lemma, if $h \in H_{\Z}$ and $x \in \{a_d,b_d\}$ for some $1 \leq d \leq g$, then we can define
$Y(h,x) \in \cQ_{g/1}$ to be the image of any corresponding element $X(h,x,x+y) \in \cV_2$.

\subsection{h-values of Y-elements}

We now prove that if $h \in H_{\Z}$ and $x \in \{a_d,b_d\}$ for some $1 \leq d \leq g$, then
$Y(h,x)$ only depends on the projection of $h$ to $\Span{a_d,b_d}$:

\begin{lemma}
\label{lemma:hvaluey}
Let $h,h' \in H_{\Z}$ and let $x \in \{a_d,b_d\}$ for some $1 \leq d \leq g$.  Assume that
$h'-h \in \Span{a_d,b_d}^{\perp}$.  Then $Y(h,x) = Y(h',x)$.
\end{lemma}
\begin{proof}
Set $k = h'-h$, so $h' = h + k$.  We can write $k$ as a sum of elements lying
in subspaces of the form $\Span{a_e,b_e}$ with $e \neq d$, and it is enough to prove
the lemma for $k$ an element of such an $\Span{a_e,b_e}$.  Using our standing assumption $g \geq 4$ (see 
Assumption \ref{assumption:genus}), we can find $1 \leq e' \leq g$ with $e' \neq d,e$.  
Using the cube relation (Lemma \ref{lemma:cuberelation}) and the additivity relation (Lemma \ref{lemma:additivityrelation}), we see that
\begin{align*}
X(h+k,x,x+a_{e'}) &= X(h,x,x+a_{e'}) \blue{-X(h,k,x+a_{e'}) + X(h+x,k,x+a_{e'})}\\
                  &= X(h,x,x+a_{e'}) \blue{-X(h,k,x) - X(h+x,k,a_{e'})} \\
                  &\ \ \ \ \blue{+ X(h+x,k,x) + X(h+2x,k,a_{e'})}.
\end{align*}
The blue terms lie in the span of $\cV_1$, and thus vanish in $\cQ_{g/1}$.  We conclude that
\[Y(h,x) = X_{1}(h,x,x+a_{e'}) = X_{1}(h+k,x,x+a_{e'}) = Y(h+k,x).\qedhere\]
\end{proof}

Letting $\cV_{2/1}$ be the image of $\cV_2$ in $\cQ_{g/1}$, we see from the above claim that $\cV_{2/1}$ consists of all $Y(h,x)$
with $h \in \Span{a_d,b_d}$ and $x \in \{a_d,b_d\}$ for some $1 \leq d \leq g$.

\subsection{Mapping Y-elements}

Consider the composition
\[\begin{tikzcd}
\cQ_g \arrow{r}{\fq} & \Q[H_{\Z}] = \Span{\fq(\cV_1)} \oplus \Q[\bigcup_{i=1}^g \Span{a_i,b_i}] \arrow[two heads]{r} & \Q[\bigcup_{i=1}^g \Span{a_i,b_i}],
\end{tikzcd}\]
where the equality comes from Proposition \ref{proposition:calcquotient1}.  This induces a map
\[\fq_{/1}\colon \cQ_{g/1} \rightarrow \Q[\bigcup_{i=1}^g \Span{a_i,b_i}].\]
For $h \in \Span{a_i,b_i}$, we will still denote the corresponding element of $\Q[\bigcup_{i=1}^g \Span{a_i,b_i}]$ by $\GrpR{h}$.
The following lemma calculates the image of $Y(h,x)$ under the map $\fq_{/1}$:

\begin{lemma}
\label{lemma:yimage}
For $Y(h,x) \in \cV_{2/1}$, we have $\fq_{/1}(Y(h,x)) = \GrpR{h} - 2 \GrpR{h+x} + \GrpR{h+2x}$.
\end{lemma}
\begin{proof}
Let $1 \leq d \leq g$ be such that $h \in \Span{a_d,b_d}$ and $x \in \{a_d,b_d\}$.  Pick $1 \leq e \leq g$ with $e \neq d$, so
$Y(h,x) = X_{1}(h,x,x+a_e)$.  We then have
\[\fq(X(h,x,x+a_e)) = \GrpR{h} - \GrpR{h+x} - \GrpR{h+x+a_e} + \GrpR{h+2x+a_e}.\]
To project this into $\Q[\bigcup_{i=1}^g \Span{a_i,b_i}]$, we can add the images under $\fq$ of any
elements of $\cV_1$.  Adding
\begin{align*}
\fq(X(h,x,a_e)) - \fq(X(h,2x,a_e)) &= (\GrpR{h}-\GrpR{h+x}-\GrpR{h+a_e}+\GrpR{h+x+a_e}) \\
                                   &\ \ \ \ -(\GrpR{h}-\GrpR{h+2x} - \GrpR{h+a_e} + \GrpR{h+2x+a_e}),
\end{align*}
we get $\GrpR{h} - 2 \GrpR{h+x} + \GrpR{h+2x} \in \Q[\Span{a_d,b_d}]$.
Since this lies in $\Q[\bigcup_{i=1}^g \Span{a_i,b_i}]$, it is the projection of $\fq(X(h,x,x+a_e))$.  The lemma
follows.
\end{proof}

\subsection{Y-relation}

We now give a basic relation between the $Y(h,x)$:

\begin{lemma}
\label{lemma:yrelation}
Let $1 \leq d \leq g$ and let $h \in \Span{a_d,b_d}$.  Then
\[Y(h,a_d) -2 Y(h+b_d,a_d) + Y(h+2b_d,a_d) = Y(h,b_d) - 2 Y(h+a_d,b_d) + Y(h+2a_d,b_d).\]
\end{lemma}
\begin{proof}
For an arbitrary $k \in H_{\Z}$, we first claim that there is a relation
\begin{align*}
&X(k,a_1,a_2) - X(k+b_1,a_1,a_2) - X(k+b_2,a_1,a_2) + X(k+b_1+b_2,a_1,a_2) \\
&\ \ \ \ = X(k,b_1,b_2) - X(k+a_1,b_1,b_2) - X(k+a_2,b_1,b_2) + X(k+a_1+a_2,b_1,b_2)
\end{align*}
in $\cQ_g$.  To see this, observe that $\fq$ takes this to a true relation in $\Q[H_{\Z}]$, and each term of it
lies in $\cV_1$.  Proposition \ref{proposition:calcquotient1} says that the restriction of $\fq$ to $\Span{\cV_1}$ is an injection,
so we conclude that the above is a relation in $\cQ_g$, as claimed.

Recall our standing assumption that $g \geq 4$ (Assumption \ref{assumption:genus}).
Using this, pick distinct $1 \leq e,e' \leq g$ with $e,e' \neq d$.  Let $\phi \in \Sp_{2g}(\Z)$ be a symplectic automorphism taking $(a_1,b_1,a_2,b_2)$ to
$(a_d+a_e,b_d-b_{e'},a_d+a_{e'},b_d-b_e)$, which exists since both are partial symplectic bases.  Choose $k \in H_{\Z}$ such that $\phi(k) = h$.  The
group $\Sp_{2g}(\Z)$ acts on $\cQ_g$, so we can apply $\phi$ to the above relation and get a new relation
\begin{align*}
  &X(h,a_d+a_e,a_d+a_{e'}) - X(h+b_d-b_{e'},a_d+a_e,a_d+a_{e'}) \\
  &\ \ \  - X(h+b_d-b_e,a_d+a_e,a_d+a_{e'}) + X(h+2 b_d-b_{e'}-b_e,a_d+a_e,a_d+a_{e'}) \\
= &X(h,b_d-b_{e'},b_d-b_e) - X(h+a_d+a_e,b_d-b_{e'},b_d-b_e) \\
  &\ \ \ \ - X(h+a_d+a_{e'},b_d-b_{e'},b_d-b_e) + X(h+2a_d+a_e+a_{e'},b_d-b_{e'},b_d-b_e).
\end{align*}
Each term of this maps to a corresponding term in the desired relation between the $Y(h,x)$.  For instance, using
the additivity relation (Lemma \ref{lemma:additivityrelation}) we have
\begin{small}
\begin{align*}
X_{1}(h,a_d+a_e,a_d+a_{e'}) &= X_{1}(h,a_d,a_d+a_{e'}) + \blue{X_{1}(h+a_d,a_e,a_d+a_{e'})} = Y(h,a_d), \\
X_{1}(h+b_d-b_{e'},a_d+a_e,a_d+a_{e'}) &= X_{1}(h+b_d-b_{e'},a_d,a_d+a_{e'}) \\
                                     &\ \ \ \ + \blue{X_{1}(h+b_d-b_{e'}+a_d,a_e,a_d+a_{e'})} = Y(h+b_d,a_d). 
\end{align*}
\end{small}%
Here the blue terms are images of elements of $\Span{\cV_1}$ and thus die in $\oQ_g$, and for the
final equality we are using Lemma \ref{lemma:hvaluey} to discard the $-b_{e'}$.  The other terms are similar.
\end{proof}

\subsection{A smaller generating set}

We now show that only some of the $Y(h,x)$ are needed to generate $\Span{\cV_{2/1}}$.  Define
$\cW_2$ to be the union of the following two sets:
\begin{itemize}
\item $\Set{$Y(h,a_d)$}{$1 \leq d \leq g$, $h \in \Span{a_d,b_d}$}$; and
\item $\Set{$Y(h,b_d)$}{$1 \leq d \leq g$, $h \in \Span{a_d,b_d}$ with $a_d$-coordinate $0$ or $1$}$.
\end{itemize}
We then have:

\begin{lemma}
\label{lemma:smallergenset2}
Letting the notation be as above, we have $\Span{\cW_2} = \Span{\cV_{2/1}}$.
\end{lemma}
\begin{proof}
Let $1 \leq d \leq g$.  For all $h \in \Span{a_d,b_d}$ we must show that $\Span{\cW_2}$ contains
$Y(h,b_d)$.  Assume for the sake of contradiction that it does not, and let $Y(h,b_d)$ be an element
not lying in $\Span{\cW_2}$ such that the $a_d$-coordinate $\lambda$ of $h$ is as small as possible.
We thus either have $\lambda \geq 2$ or $\lambda \leq -1$.  If $\lambda \geq 2$, then the relation
\begin{align*}
&Y(h-2a_d,a_d) -2 Y(h-2a_d+b_d,a_d) + Y(h-2a_d+2b_d,a_d) \\
&\quad\quad = Y(h-2a_d,b_d) - 2 Y(h-a_d,b_d) + Y(h,b_d)
\end{align*}
from Lemma \ref{lemma:yrelation} allows us to write $Y(h,b_d)$ in terms of elements that lie in $\Span{\cW_2}$, a
contradiction.  The case where $\lambda \leq -1$ is similar.
\end{proof}

\subsection{Main result}

We now prove the main result of this section:

\begin{proposition}
\label{proposition:calcquotient2}
The restriction of $\fq_{/1}$ to $\Span{\cV_{2/1}}$ is injective, and
\[\Q[\bigcup_{i=1}^g \Span{a_i,b_i}] = \Span{\fq_{/1}(\cV_{2/1})} \oplus \Q[\bigcup_{i=1}^g \{0,a_i,b_i,a_i+b_i\}].\]
\end{proposition}
\begin{proof}
By Lemma \ref{lemma:smallergenset2}, we have $\Span{\cV_{2/1}} = \Span{\cW_2}$.  We will prove that $\fq_{/1}$ takes
the elements of $\cW_2$ to linearly independent elements of $\Q[\bigcup_{i=1}^g \Span{a_i,b_i}]$ spanning a subspace that
is a complement to $\Q[\bigcup_{i=1}^g \{0,a_i,b_i,a_i+b_i\}]$.  For this, fix some $1 \leq d \leq g$.
Let $\cW_2(d)$ be the set of all $Y(h,x) \in \cW_2$ such that $h,x \in \Span{a_d,b_d}$.  By Lemma \ref{lemma:yimage}, for
$Y(h,x) \in \cW_2(d)$ we have
\[\fq_{/1}(Y(h,x)) =  \GrpR{h} - 2\GrpR{h+x} + \GrpR{h+2x} \in \Q[\Span{a_d,b_d}].\]
It is enough to prove that $\fq_{/1}$ takes the elements of $\cW_2(d)$ to linearly
independent elements of $\Q[\Span{a_d,b_d}]$ such that
\[\Q[\Span{a_d,b_d}] = \Span{\fq_{/1}(\cW_2(d)} \oplus \Q[\{0,a_d,b_d,a_d+b_d\}].\]
The key to this is the following easy piece of linear algebra:

\begin{unnumberedclaim}
Let $\{\vec{e}_n\}_{n \in \Z}$ be a basis for a $\Q$-vector space $V$ of countable dimension.
For each $n \in \Z$, let $\vec{f}_n = \vec{e}_n - 2 \vec{e}_{n+1} + \vec{e}_{n+2} \in V$.  Then the $\{\vec{f}_n\}_{n \in \Z}$ are linearly
independent elements of $V$ spanning a complement to $\Span{\vec{e}_0,\vec{e}_1}$.
\end{unnumberedclaim}
\begin{proof}[Proof of claim]
Immediate from the fact that
\begin{itemize}
\item $\{\vec{e}_0,\vec{e}_1,\vec{f}_0,\vec{f}_1,\vec{f}_2,\ldots\}$ is a basis for $\Span{\vec{e}_0,\vec{e}_1,\vec{e}_2,\ldots}$; and
\item $\{\vec{e}_1,\vec{e}_0,\vec{f}_{-1},\vec{f}_{-2},\vec{f}_{-3},\ldots\}$ is a basis for $\Span{\vec{e}_1,\vec{e}_0,\vec{e}_{-1},\ldots}$.\qedhere
\end{itemize}
\end{proof}

For $n,m \in \Z$, let $\vec{e}_{n,m} = n a_d+m b_d \in \Span{a_d,b_d}$.  The $\vec{e}_{n,m}$ form a basis for
$\Q[\Span{a_d,b_d}]$.  Define the following subspaces of $\Q[\Span{a_d,b_d}]$:
\begin{itemize}
\item for $n_0 \in \Z$, the subspace $L_{n_0} = \SpanSet{$\vec{e}_{n_0,m}$}{$m \in \Z$}$; and
\item for $m_0 \in \Z$, the subspace $M_{m_0} = \SpanSet{$\vec{e}_{n,m_0}$}{$n \in \Z$}$.
\end{itemize}
For $m_0 \in \Z$, we have
\[\fq_{/1}(Y(\vec{e}_{n,m_0},a_0)) = \vec{e}_{n,m_0} - 2 \vec{e}_{n+1,m_0} + \vec{e}_{n+2,m_0} \in M_{m_0} \quad \text{for all $n \in \Z$}.\]
By the above claim, these are linearly independent elements of $M_{m_0}$ spanning a complement to $\Span{\vec{e}_{0,m_0},\vec{e}_{1,m_0}}$.
We have
\[\Q[\Span{a_d,b_d}] = \bigoplus_{m_0 \in \Z} M_{m_0},\]
so this implies that $\fq_{/1}$ takes the elements of $\Set{$Y(\vec{e}_{n,m},a_d)$}{$n,m \in \Z$}$
to linearly independent elements of $\Q[\Span{a_d,b_d}]$ spanning a complement to $L_0 \oplus L_1$.
For $n_0 \in \{0,1\}$, we have
\[\fq_{/1}(Y(\vec{e}_{n_0,m},b_0)) = \vec{e}_{n_0,m} - 2 \vec{e}_{n_0,m+1} + \vec{e}_{n_0,m+2} \in L_{n_0} \quad \text{for all $m \in \Z$}.\]
By the above claim, these are linearly independent elements of $L_{n_0}$ spanning a complement to $\Span{\vec{e}_{n_0,0},\vec{e}_{n_0,1}}$.
Putting this all together, we conclude that $\fq_{/1}$ takes the elements of
\[\cW_2(d) = \Set{$Y(\vec{e}_{n,m},a_d)$}{$n,m \in \Z$} \cup \Set{$Y(\vec{e}_{n,m},b_d)$}{$n \in \{0,1\}$, $m \in \Z$}\]
to linearly independent elements of $\Q[\Span{a_d,b_d}]$ spanning a complement to $\Span{\vec{e}_{0,0}, \vec{e}_{0,1},\vec{e}_{1,0},\vec{e}_{1,1}}$,
as desired.
\end{proof}

\section{Calculation of the quotient III: the third set of generators}
\label{section:calcquotient3}

We conclude with the final set $\cV_3$ of generators, which we recall consists of all $X(h,x,y)$ with
$h \in H_{\Z}$ arbitrary and the pair $(x,y)$ equal to either
$(a_d+a_e,b_d-b_e)$ or $(a_d+b_e,b_d+a_e)$ for some distinct $1 \leq d,e \leq g$.
As notation, define $\cQ_{g/2} = \cQ_g/\Span{\cV_1, \cV_2}$.  For any generator $X(h,x,y)$ of $\cQ_g$,
let $X_{2}(h,x,y)$ be its image in $\cQ_{g/2}$.

\subsection{ZW-elements}

We start by proving that for $X(h,x,y) \in \cV_3$, its image $X_{2}(h,x,y) \in \cQ_{g/2}$ does not
depend on $h$.  

\begin{lemma}
\label{lemma:zwelldefined} 
For distinct $1 \leq d, e \leq g$, let $(x,y)$ be either $(a_d+a_e,b_d-b_e)$ or $(a_d+b_e,b_d+a_e)$.
Then for $h,h' \in H_{\Z}$ we have $X_2(h,x,y) = X_2(h',x,y)$.
\end{lemma}
\begin{proof}
We will give the details for $(x,y) = (a_d+a_e,b_d-b_e)$; the other case is similar.  It is enough
to prove that for\footnote{You might think we also need to handle things like $k=-a_1$, but
since $h$ is arbitrary this is not necessary.} $k \in \{a_1,b_1,\ldots,a_g,b_g\}$, we have $X_2(h,x,y) = X_2(h+k,x,y)$.
All these values of $k$ are handled using the cube relation (Lemma \ref{lemma:cuberelation}).  We
will give the details for $k = a_d$ and leave the other cases to the reader.  

The elements $a_d+a_e \in H_{\Z}$ and $a_d,b_d-b_e \in H_{\Z}$ are homologically separate.  By the cube relation (Lemma \ref{lemma:cuberelation}),
we thus have
\begin{align*}
X(h+a_d,a_d+a_e,b_d-b_e) &= X(x,a_d+a_e,b_d-b_e) \\
                         &\ \ \ \ \ \ \ \ \blue{-X(x,a_d+a_e,a_d) + X(x+b_d-b_e,a_d+a_e,a_d)}.
\end{align*}
Both blue terms lie in $\cV_2$ and thus vanish in $\cQ_{g/2}$.  The lemma follows.
\end{proof}

By this lemma, for distinct $1 \leq d,e \leq g$ we can define
$Z(a_d+a_e,b_d-b_e) = X_2(h,a_d+a_e,b_d-b_e)$ and $W(a_d+b_e,b_d+a_e) = X_2(h,a_d+b_e,b_d+a_e)$
for any $h \in H_{\Z}$.  Let $\cV_{3/2}$ be the set of these $Z(x,y)$ and $W(x,y)$.  
Lemma \ref{lemma:generateq} says that $\cV = \cV_1 \cup \cV_2 \cup \cV_3$ spans $\cQ_g$,
so $\cV_{3/2}$ spans $\cQ_{g/2}$.

\subsection{ZW-relations}

These elements satisfy several relations:

\begin{lemma}
\label{lemma:zwrelationdumb}
For distinct $1 \leq d,e \leq g$, we have $Z(a_d+a_e,b_d-b_e) = - Z(a_e+a_d,b_e-b_d)$ and
$W(a_d+b_e,b_d+a_e) = W(a_e+b_d,b_e+a_d)$.
\end{lemma}
\begin{proof}
The first follows from the inverse relation (Lemma \ref{lemma:inverserelation}) and the second follows
from the symmetry relation (Lemma \ref{lemma:symmetryrelation}).
\end{proof}

\begin{lemma}
\label{lemma:zrelation}
For distinct $1 \leq d,e,f \leq g$ we have 
\[Z(a_d+a_f,b_d-b_f) = Z(a_d+a_e,b_d-b_e) + Z(a_e+a_f,b_e-b_f).\]
\end{lemma}
\begin{proof}
Since $(b_d-b_e) + (b_e-b_f) = b_d-b_f$, the additivity relation (Lemma \ref{lemma:additivityrelation}) says that
\begin{equation}
\label{eqn:zrelation1lift}
X(0,a_d+a_e+a_f,b_d-b_f) = X(0,a_d+a_e+a_f,b_d-b_e) + X(b_d-b_e,a_d+a_e+a_f,b_e-b_f).
\end{equation}
The additivity relation also implies that
\[X(0,a_d+a_e+a_f,b_d-b_f) = X(0,a_d+a_f,b_d-b_f) \blue{+X(a_d+a_f,a_e,b_d-b_f)}.\]
The blue term lies in the span of $\cV_1$, and thus vanishes in $\cQ_{g/2}$.  We therefore have
\[X_2(0,a_d+a_e+a_f,b_d-b_f) = X_2(0,a_d+a_f,b_d-b_f) = Z(a_d+a_f,b_d-b_f).\]
Similarly, we have
\begin{align*}
X_2(0,a_d+a_e+a_f,b_d-b_e) &= Z(a_d+a_e,b_d-b_e), \\
X_2(b_d-b_e,a_d+a_e+a_f,b_e-b_f) &= Z(a_e+a_f,b_e-b_f).
\end{align*}
Plugging all of this into \eqref{eqn:zrelation1lift} gives the desired relation.
\end{proof}

\begin{lemma} 
\label{lemma:wrelation}
For distinct $1 \leq d,e,f \leq g$ we have
\[W(a_d+b_f,b_d+a_f) = Z(a_d+a_e,b_d-b_e) + W(a_e+b_f,b_e+a_f).\]
\end{lemma}
\begin{proof}
Since $(b_d-b_e) + (b_e+a_f) = b_d+a_f$, the additivity relation (Lemma \ref{lemma:additivityrelation}) says that
\[X(0,a_d+a_e+b_f,b_d+a_f) = X(0,a_d+a_e+b_f,b_d-b_e) + X(b_d-b_e,a_d+a_e+b_f,b_e+a_f).\]
Just like in the proof of Lemma \ref{lemma:zrelation}, this projects to the desired relation.
\end{proof}

\subsection{A smaller generating set}

Let 
\[\cW_3 = \Set{$Z(a_d+a_{d+1},b_d-b_{d+1})$}{$1 \leq d \leq g-1$} \cup \{W(a_1+b_2,b_1+a_2)\}.\]
We prove that this spans $\cQ_{g/2}$:

\begin{lemma}
\label{lemma:smallergenset3}
Letting the notation be as above, we have $\Span{\cW_3} = \cQ_{g/2}$.
\end{lemma}
\begin{proof}
It is enough to prove that $\Span{\cW_3}$ contains $\cV_{3/2}$.  We do this in
two steps:

\begin{unnumberedclaim}
For all distinct $1 \leq d,e \leq g$, we have $Z(a_d+a_e,b_d-b_f) \in \Span{\cW_3}$.
\end{unnumberedclaim}

Using Lemma \ref{lemma:zwrelationdumb}, we can assume that $d < e$.  Lemma \ref{lemma:zrelation} then
implies that
\[Z(a_d+a_e,b_d-b_f) = \sum_{i=d}^{e-1} Z(a_i+a_{i+1},b_i-b_{i+1}) \in \Span{\cW_3}.\]

\begin{unnumberedclaim}
For all distinct $1 \leq d,e \leq g$, we have $W(a_d+b_e,b_d+a_e) \in \Span{\cW_3}$.
\end{unnumberedclaim}

We will prove this in the case where $d,e \notin \{1,2\}$.  The cases where one or both
are in $\{1,2\}$ are similar (but easier).
Start by using Lemmas \ref{lemma:zwrelationdumb} and \ref{lemma:wrelation} to see that
\begin{align*}
W(a_2+b_e,b_2+a_e) &= W(a_e+b_2,b_e+a_2) \\
                     &= Z(a_e+a_1,b_e-b_1)+W(a_1+b_2,b_1+a_2) \in \Span{\cW_3}.
\end{align*}
Using this, another application of Lemma \ref{lemma:wrelation} says that
\[W(a_d+b_e,b_d+a_e) = Z(a_d+a_2,b_d-b_2)+W(a_2+b_e,b_2+a_e) \in \Span{\cW_3}.\qedhere\]
\end{proof}

\subsection{Mapping ZW-elements}

Consider the composition
\[\begin{tikzcd}[column sep=small]
\cQ_g \arrow{r}{\fq} & \Q[H_{\Z}] = \Span{\fq(\cV_1), \fq(\cV_2)} \oplus \Q[\bigcup_{i=1}^g \Span{0,a_i,b_i,a_i+b_i}] \arrow[two heads]{r} & \Q[\bigcup_{i=1}^g \Span{0,a_i,b_i,a_i+b_i}],
\end{tikzcd}\]
where the equality comes from Proposition \ref{proposition:calcquotient2}.  This induces a map
\[\fq_{/2}\colon \cQ_{g/2} \rightarrow \Q[\bigcup_{i=1}^g \Span{0,a_i,b_i,a_i+b_i}].\]
For $h \in \Span{0,a_i,b_i,a_i+b_i}$, we will still denote the corresponding element of $\Q[\bigcup_{i=1}^g \Span{0,a_i,b_i,a_i+b_i}]$ by $\GrpR{h}$.
For $1 \leq d \leq g$, let $\theta_d = \GrpR{0} - \GrpR{a_i} - \GrpR{b_i} + \GrpR{a_i+b_i}$.
The following lemma calculates the images of $Z(x,y)$ and $W(x,y)$ under the map $\fq_{/2}$:

\begin{lemma}
\label{lemma:zimage}
For distinct $1 \leq d, e \leq g$, we have
\[\fq_{/2}(Z(a_d+a_e,b_d-b_e)) = \theta_d - \theta_e \quad \text{and} \quad \fq_{/2}(W(a_d+b_e,b_d+a_e)) = \theta_d + \theta_e.\]
Consequently, the image of $\fq_{/2}$ is $\SpanSet{$\theta_d$}{$1 \leq d \leq g$}$.
\end{lemma}
\begin{proof}
The two calculations are similar, so we will give the details for the first.  Note that
\[\fq(X(0,a_d+a_e,b_d-b_e)) = \GrpR{0} - \GrpR{a_d+a_e} - \GrpR{b_d-b_e} + \GrpR{a_d+a_e+b_d-b_e}.\]
To project this into $\Q[\bigcup_{i=1}^g \Span{0,a_i,b_i,a_i+b_i}]$, we can add the images under $\fq$ of any elements of $\cV_1$ or $\cV_2$.  Adding
\[\fq(X(0,a_d,a_e)) + \fq(X(0,b_d,-b_e)) - \fq(X(0,a_d+b_d,a_e-b_e)),\]
we get
\[\left(\GrpR{0} - \GrpR{a_d} - \GrpR{b_d} + \GrpR{a_d+b_d}\right) + \left(\GrpR{0} - \GrpR{a_e} - \GrpR{-b_e} + \GrpR{a_e-b_e}\right).\]
Project this to $\Q[\bigcup_{i=0}^g \Span{a_i,b_i}]$ and add
\begin{align*}
\fq_{/1}(Y(-b_e,b_e)) - \fq_{/1}(Y(a_e-b_e,b_e)) &= \left(\GrpR{-b_e} - 2 \GrpR{0} +\GrpR{b_e}\right) \\
                                                 &\ \ \ \ \ \ \ \ -\left(\GrpR{a_e-b_e} - 2\GrpR{a_e} + \GrpR{a_e+b_e}\right).
\end{align*}
We get
\[\left(\GrpR{0} - \GrpR{a_d} - \GrpR{b_d} + \GrpR{a_d+b_d}\right) + \left(-\GrpR{0} + \GrpR{a_e} + \GrpR{-b_e} - \GrpR{a_e-b_e}\right) = \theta_d - \theta_e.\qedhere\]
\end{proof}

\subsection{Main result}

We now prove the main result of this section.  This will complete the proof of Theorem \ref{theorem:part2theorem2}
outlined in \S \ref{section:calcquotientoutline}.

\begin{proposition}
\label{proposition:calcquotient3}
The map $\fq_{/2}\colon \cQ_{g/2} \rightarrow \Q[\bigcup_{i=1}^g \Span{0,a_i,b_i,a_i+b_i}]$ is injective.
\end{proposition}
\begin{proof}
Lemma \ref{lemma:zimage} says that $\fq_{/2}$ takes $\cQ_{g/2}$ onto $\SpanSet{$\theta_d$}{$1 \leq d \leq g$}$.  Since
the $\theta_d$ are linearly independent, the image of $\fq_{/2}$ is $g$-dimensional.
Lemma \ref{lemma:smallergenset3} says that $\cW_3$ spans $\cQ_{g/2}$.  Since $\cW_3$ contains
$g$ elements, we deduce that $\fq_{/2}$ takes the elements of $\cW_3$ to linearly independent elements.
We conclude that $\fq_{/2}$ is injective.
\end{proof} 

\part{Relations in the kernel of the coinvariant Reidemeister pairing}
\label{part:3}

Let $\ofr$ be the coinvariant Reidemeister pairing.  Having proved Theorem \ref{theorem:part2theorem}
in Part \ref{part:2}, the remaining conclusion of Theorem \ref{theorem:bigtheorem} 
that must be proved is that $\ker(\ofr)$ is a finite-dimensional algebraic representation
of $\Sp_{2g}(\Z)$.  In this part of the paper, we will construct enough relations in
$\ker(\ofr)$ to force it to be a subrepresentation of $((\wedge^2 H)/\Q)^{\otimes 2}$.
See the introductory \S \ref{section:part3intro}
for an outline.

\section{Introduction to Part \ref{part:3}}
\label{section:part3intro}

This section fixes some notation and does some preliminary calculations, and then outlines
what we will do in the rest of Part \ref{part:3}.

\subsection{Intersection pairing}
\label{section:introductionpairing}

Let $\omega$ be the algebraic intersection pairing on $H$.  Like in \S \ref{section:cgpicoinv}, we will
identify $\omega$ with an $\Sp_{2g}(\Q)$-invariant element $\omega \in \wedge^2 H$.  If $\{a_1,b_1,\ldots,a_g,b_g\}$
is a symplectic basis for $H$, then
\[\omega = a_1 \wedge b_1 + \cdots + a_g \wedge b_g.\]
The line spanned by $\omega$ is an $\Sp_{2g}(\Q)$-invariant copy of $\Q$ in $\wedge^2 H$, and whenever we talk
about $(\wedge^2 H)/\Q$ we mean the quotient by this line.

\subsection{Coinvariants}

As we discussed in \S \ref{section:birmanexactseq}, the group $\pi_g$ is the point-pushing subgroup
of $\Torelli_{g,1}$.  The coinvariants $(\cC_g)_{\Torelli_{g,1}}$ are thus a quotient of
$(\cC_g)_{\pi_g}$.  However:

\begin{lemma}
\label{lemma:cgtorellicoinv}
We have $(\cC_g)_{\Torelli_{g,1}} \cong (\cC_g)_{\pi_g} \cong (\wedge^2 H)/\Q$.
\end{lemma}
\begin{proof}
Lemma \ref{lemma:cgpicoinv} says that that $(\cC_g)_{\pi_g} \cong (\wedge^2 H)/\Q$.
The induced action of $\Torelli_{g,1}/\pi_g \cong \Torelli_g$ on $(\wedge^2 H)/\Q$ is trivial since
$\Torelli_g$ acts trivially on $H$.  It follows that nothing has to be killed when passing
from $(\cC_g)_{\pi_g}$ to $(\cC_g)_{\Torelli_{g,1}}$, as desired.
\end{proof}

The product $\Torelli_{g,1} \times \Torelli_{g,1}$ acts on $\cC_g^{\otimes 2}$, and Lemma \ref{lemma:cgtorellicoinv}
implies that:

\begin{corollary}
\label{corollary:doublecoinv}
We have 
$(\cC_g^{\otimes 2})_{\Torelli_{g,1} \times \Torelli_{g,1}} \cong (\cC_g^{\otimes 2})_{\pi_g \times \pi_g} \cong \left((\wedge^2 H)/\Q\right)^{\otimes 2}$.
\end{corollary}

\subsection{Reidemeister kernel}

Recall from \S \ref{section:mainproofs} that the coinvariant Reidemeister pairing is a linear map
\[\ofr\colon (\cC_g^{\otimes 2})_{\Torelli_{g,1}} \rightarrow \Q[H_{\Z}].\]
In this, $\Torelli_{g,1}$ acts on $\cC_g^{\otimes 2}$ via the diagonal map $\Delta\colon \Torelli_{g,1} \rightarrow \Torelli_{g,1} \times \Torelli_{g,1}$.
To help distinguish this from the $\Torelli_{g,1} \times \Torelli_{g,1}$-action,
we will write this using $\Delta(\Torelli_{g,1})$.  With this notation, the coinvariant Reidemeister pairing is a linear map
\[\ofr\colon (\cC_g^{\otimes 2})_{\Delta(\Torelli_{g,1})} \rightarrow \Q[H_{\Z}].\]
Define $\cK_g = \ker(\ofr)$, so we have an exact sequence
\[\begin{tikzcd}
0 \arrow{r} & \cK_g \arrow{r} & (\cC_g^{\otimes 2})_{\Delta(\Torelli_{g,1})} \arrow{r}{\ofr} & \Q[H_{\Z}].
\end{tikzcd}\]
Theorem \ref{theorem:part2theorem} says that $\cK_g$ is generated by the set of elements of the form
$\CgCg{\Cg{\gamma_1}^{h_1}}{\Cg{\gamma_2}^{h_2}}$ with $\gamma_1 \in [\pi_g,\pi_g]$
and $\gamma_2 \in [\pi_g,\pi_g]$ separated and $h_1,h_2 \in H_{\Z}$ arbitrary.

\subsection{Main theorem}
\label{section:mainpart3}

Let us recall the theorem we are trying to prove:

\newtheorem*{theorem:bigtheorem}{Theorem \ref{theorem:bigtheorem}}
\begin{theorem:bigtheorem}
Let $\ofr\colon (\cC_g \otimes \cC_g)_{\Delta(\Torelli_{g,1})} \rightarrow \Q[H_{\Z}]$ be the coinvariant Reidemeister pairing.
Then both $\ker(\ofr)$ and $\coker(\ofr) = \coker(\fr)$ are finite-dimensional.  Moreover, $\ker(\ofr)$ is an algebraic
representation of $\Sp_{2g}(\Z)$.
\end{theorem:bigtheorem}

In Part \ref{part:2}, we proved Theorem \ref{theorem:part2theorem}, which among other things
said that $\coker(\ofr)$ is finite-dimensional.  To complete the proof of Theorem \ref{theorem:bigtheorem},
we must therefore prove that $\cK_g = \ker(\ofr)$ is a finite-dimensional algebraic representation
of $\Sp_{2g}(\Z)$.
Using Corollary \ref{corollary:doublecoinv}, we obtain a surjective map
\[\fa\colon (\cC_g^{\otimes 2})_{\Delta(\Torelli_{g,1})} \twoheadrightarrow (\cC_g^{\otimes 2})_{\Torelli_{g,1} \times \Torelli_{g,1}} \cong \left((\wedge^2 H)/\Q\right)^{\otimes 2}\]
whose codomain is a finite-dimensional algebraic representation of\footnote{Actually, it is an algebraic representation of $\Sp_{2g}(\Z) \times \Sp_{2g}(\Z)$, but we only
care about the diagonal subgroup $\Sp_{2g}(\Z)$ since the domain is only a representation of this diagonal subgroup.} $\Sp_{2g}(\Z)$.  We will call this the {\em algebraization map}.  It
is far from an isomorphism; indeed, the coinvariant Reidemeister pairing takes $(\cC_g^{\otimes 2})_{\Delta(\Torelli_{g,1})}$ to an infinite-dimensional
representation of $\Sp_{2g}(\Z)$.  However, we will prove that this is the only obstruction
to $\fa$ being an isomorphism:

\begin{theorem}
\label{theorem:part3theorem}
The restriction of the algebraization map $\fa$ to $\cK_g$ is an injection.
\end{theorem}

This will imply that $\cK_g$ is a subrepresentation of the finite-dimensional algebraic representation $\left((\wedge^2 H)/\Q\right)^{\otimes 2}$ of $\Sp_{2g}(\Z)$.  By
\eqref{eqn:algclosed}, this will imply that $\cK_g$ is a finite-dimensional algebraic representation of $\Sp_{2g}(\Z)$, as was claimed by Theorem \ref{theorem:bigtheorem}.
The rest of this paper is devoted to the proof of Theorem \ref{theorem:part3theorem}.  We
divide the proof into four steps:
\begin{itemize}
\item \S \ref{section:prelimkg} does some preliminary calculations in $\cK_g$.
\item \S \ref{section:kggen} constructs a refined generating set for $\cK_g$.
\item \S \ref{section:kgrel} identifies some redundancies among these generators.
\item \S \ref{section:kgalg} uses these generators and relations to prove
Theorem \ref{theorem:part3theorem}.  
\end{itemize}

\section{Preliminary calculations in \texorpdfstring{$\cK_g$}{Kg}}
\label{section:prelimkg}

In this section, we make a preliminary study of $\cK_g$.

\subsection{Pairs of genus-1 curves}
Recall from Theorem \ref{theorem:part2theorem} that $\cK_g$ is spanned by elements 
$\CgCg{\Cg{\gamma_1}^{h_1}}{\Cg{\gamma_2}^{h_2}}$ with $\gamma_1 \in [\pi_g,\pi_g]$
and $\gamma_2 \in [\pi_g,\pi_g]$ separated and $h_1,h_2 \in H_{\Z}$ arbitrary.  Our first
result is that we can take the $\gamma_i$ to be simple closed 
curves that bound on their right sides genus-$1$ subsurfaces $T_i$ such that 
$T_1 \cap T_2$ is the basepoint:\\
\Figure{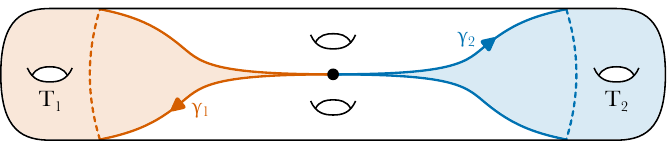}

\begin{lemma}
\label{lemma:kggenus1}
The vector space $\cK_g$ is spanned by 
$\CgCg{\Cg{\gamma_1}^{h_1}}{\Cg{\gamma_2}^{h_2}}$ where:
\begin{itemize}
\item for $i=1,2$, the curve $\gamma_i \in [\pi_g,\pi_g]$ is a simple closed separating
curve that bounds on its right side a genus-$1$ subsurface $T_i$ and $h_i \in H_{\Z}$ is arbitrary; and
\item the intersection of $T_1$ and $T_2$ is the basepoint.
\end{itemize}
\end{lemma}
\begin{proof}
Theorem \ref{theorem:part2theorem} says that $\cK_g$ is spanned by 
$\CgCg{\Cg{\gamma_1}^{h_1}}{\Cg{\gamma_2}^{h_2}}$ with $\gamma_1 \in [\pi_g,\pi_g]$
and $\gamma_2 \in [\pi_g,\pi_g]$ separated and $h_1,h_2 \in H_{\Z}$ arbitrary.  Fixing
such $\gamma_1,\gamma_2 \in [\pi_g,\pi_g]$ and $h_1,h_2 \in H_{\Z}$, we must show that
$\CgCg{\Cg{\gamma_1}^{h_1}}{\Cg{\gamma_2}^{h_2}}$ can be written as a linear
combination of the indicated generators.

Let $\delta \in \pi_g$ separate $\gamma_1$ from $\gamma_2$ and let $S_1$ and $S_2$ 
be the subsurfaces to the left and right of $\delta$, respectively:\\
\Figure{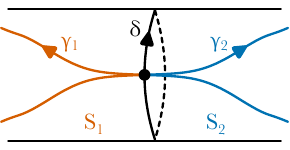}
For $i=1,2$, the curve $\gamma_i$ is in the image of the map
$[\pi_1(S_i),\pi_1(S_i)] \rightarrow [\pi_g,\pi_g]$.
Since $[\pi_1(S_i),\pi_1(S_i)]$ is $\pi_1(S_i)$-normally generated by simple closed curves bounding genus-$1$ subsurfaces on their right sides (see, e.g., \cite[Lemma A.1]{PutmanCutPaste}\footnote{This reference
gives generation rather than $\pi_1(S_i)$-normal generation.  However, it
requires the basepoint to lie in the interior, while ours lies on the boundary.  The proof of \cite[Lemma A.1]{PutmanCutPaste} shows that in this case
we only get $\pi_1(S_i)$-normal generation.}), we can write\footnote{Here we
are using our convention that $\precon{a}{b} = a b a^{-1}$.}
\[\gamma_i = \left(\precon{c_{i,1}}{\delta_{i,1}}\right)^{\epsilon_{i,1}} \cdots \left(\precon{c_{i,n_i}}{\delta_{i,n_i}}\right)^{\epsilon_{i,n}}\]
where each $\delta_{i,j} \in [\pi_1(S_i),\pi_1(S_i)]$ is a simple closed curve bounding a genus-$1$ subsurface of $S_i$ on its right side, each $c_{ij}$ is an element of $\pi_1(S_i)$, and each $\epsilon_{ij}$ is $\pm 1$.

Regard these expressions as occurring in $\pi_g$.  
By construction, for $1 \leq j \leq n_1$ and $1 \leq k \leq n_2$ 
the curves $\delta_{1,j}$ and $\delta_{2,k}$ only intersect
at the basepoint and bound genus-$1$ subsurfaces on their right sides that only intersect
at the basepoint.  The desired
expression is then
\begin{align*}
\CgCg{\Cg{\gamma_1}^{h_1}}{\Cg{\gamma_2}^{h_2}} 
&= \CgCg{\left(\sum_{j=1}^{n_1} \epsilon_{1,j} \Cg{\delta_{1,j}}^{h_1+\oc_{1,j}}\right)}{\left(\sum_{k=1}^{n_2} \epsilon_{2,k} \Cg{\delta_{2,k}}^{h_2+\oc_{2,k}}\right)} \\
&= \sum_{j=1}^{n_1} \sum_{k=1}^{n_2} \epsilon_{1,j} \epsilon_{2,k} \CgCg{\Cg{\delta_{1,j}}^{h_1+\oc_{1,j}}}{\Cg{\delta_{2,k}}^{h_2+\oc_{2,k}}}. \qedhere
\end{align*}
\end{proof}

\subsection{Exponents do not matter}
We next prove that the exponents $h_i$ are unnecessary:

\begin{lemma}
\label{lemma:kgexponents}
Let $\gamma_1 \in [\pi_g,\pi_g]$ and $\gamma_2 \in [\pi_g,\pi_g]$ be separated and let $h_1,h_2 \in H_{\Z}$ be
arbitrary.  Then
$\CgCg{\Cg{\gamma_1}^{h_1}}{\Cg{\gamma_2}^{h_2}} = \CgCg{\Cg{\gamma_1}}{\Cg{\gamma_2}}$.
\end{lemma}
\begin{proof}
The proof has two steps.  The first step handles the generators given by Lemma \ref{lemma:kggenus1} above, and the
second step reduces the lemma to those generators.

\begin{step}{1}
\label{step:kgexponents1}
The lemma is true if $\gamma_1$ and $\gamma_2$ are simple closed separating curves that bound on
their right sides genus-$1$ subsurfaces that only intersect at the basepoint.
\end{step}

Identify the group $\pi_g$ of inner automorphisms of $\pi_g$ with 
the point-pushing subgroup of $\Torelli_{g,1}$.  Since the group $\Delta(\Torelli_{g,1})$ acts trivially on 
$\cK_g \subset (\cC_g^{\otimes 2})_{\Delta(\Torelli_{g,1})}$, the subgroup $\Delta(\pi_g)$ of $\Delta(\Torelli_{g,1})$ acts trivially.  Letting $h = h_2 - h_1$,
we therefore have the following (c.f.\ \S \ref{section:coinvariantquotient}):
\[\CgCg{\Cg{\gamma_1}^{h_1}}{\Cg{\gamma_2}^{h_2}} = \CgCg{\Cg{\gamma_1}}{\Cg{\gamma_2}^{h_2-h_1}} =  \CgCg{\Cg{\gamma_1}}{\Cg{\gamma_2}^{h}}.\]
Let $T \cong \Sigma_1^1$ be the genus-$1$ subsurface bounded by $\gamma_1$ on its right side and let $S \cong \Sigma_{g-1}^1$ be the
subsurface bounded by $\gamma_1$ on its left side:\\
\Figure{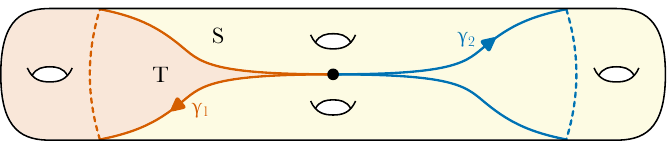}  
We have $H_{\Z} = \HH_1(T) \oplus \HH_1(S)$.  Write $h = t+s$ with
$t \in \HH_1(T)$ and $s \in \HH_1(S)$, so our goal is to prove that
$\CgCg{\Cg{\gamma_1}}{\Cg{\gamma_2}^{t+s}} = \CgCg{\Cg{\gamma_1}}{\Cg{\gamma_2}}$.

We first prove that
\begin{equation}
\label{eqn:cancelt}
\CgCg{\Cg{\gamma_1}}{\Cg{\gamma_2}^{t+s}} = \CgCg{\Cg{\gamma_1}}{\Cg{\gamma_2}^{s}}.
\end{equation}
For this, pick $\lambda \in \pi_1(T) \subset \pi_g$ with $\olambda = -t$.  Choose a simple closed nonseparating curve
$\eta \in \pi_1(S) \subset \pi_g$ that intersects $\gamma_2$ as depicted in the following figure:\\
\Figure{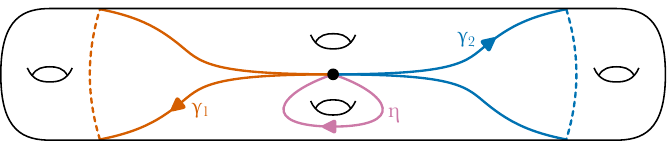}
The curve $\gamma_2 \eta$ is also a simple closed nonseparating curve, and since $\gamma_2 \in [\pi_g,\pi_g]$
the curves $\gamma_2 \eta$ and $\eta$ are homologous.  Using work of Johnson \cite{JohnsonConj}, we can
find $f \in \Torelli_{g,1}$ that is supported on $S$ such that $f(\eta) = \gamma_2 \eta$.  Since $\Delta(\Torelli_{g,1})$
acts trivially on $\cK_g \subset (\cC_g^{\otimes 2})_{\Delta(\Torelli_{g,1})}$ and also fixes $t,s \in H_{\Z}$, we therefore have
\begin{equation}
\label{eqn:cancelt1}
\CgCg{\Cg{\gamma_1}}{\Cg{\eta,\lambda}^{t+s}} = \CgCg{\Cg{f(\gamma_1)}}{\Cg{f(\eta),f(\lambda)}^{f(t)+f(s)}} = \CgCg{\Cg{\gamma_1}}{\Cg{\gamma_2 \eta,\lambda}^{t+s}}.
\end{equation}
Using our commutator identities (Lemma \ref{lemma:commutatoridentities}) along with the fact that $\ogamma_2 = 0$, we have
\[\Cg{\gamma_2 \eta,\lambda}^{t+s} = \Cg{\gamma_2,\lambda}^{t+s}+\Cg{\eta,\lambda}^{t+s+\ogamma_2} = \Cg{\gamma_2,\lambda}^{t+s} + \Cg{\eta,\lambda}^{t+s}.\]
Plugging this into \eqref{eqn:cancelt1} and canceling the term $\CgCg{\Cg{\gamma_1}}{\Cg{\eta,\lambda}^{t+s}}$, we get
$\CgCg{\Cg{\gamma_1}}{\Cg{\gamma_2,\lambda}^{t+s}} = 0$.
Using the fact that $\olambda = -t$, we conclude that
\begin{align*}
0 &= \CgCg{\Cg{\gamma_1}}{\Cg{\gamma_2,\lambda}^{t+s}} = \CgCg{\Cg{\gamma_1}}{\Cg{\gamma_2 \lambda \gamma_2^{-1} \lambda^{-1}}^{t+s}} 
  = \CgCg{\Cg{\gamma_1}}{\Cg{\gamma_2}^{t+s}} - \CgCg{\Cg{\gamma_1}}{\Cg{\gamma_2}^{t+s+\olambda}} \\
  &= \CgCg{\Cg{\gamma_1}}{\Cg{\gamma_2}^{t+s}} - \CgCg{\Cg{\gamma_1}}{\Cg{\gamma_2}^{t+s+(-t)}} = \CgCg{\Cg{\gamma_1}}{\Cg{\gamma_2}^{t+s}} - \CgCg{\Cg{\gamma_1}}{\Cg{\gamma_2}^{s}},
\end{align*}
as was claimed in \eqref{eqn:cancelt}.

To complete the proof, it is now enough to prove that
\begin{equation}
\label{eqn:cancelx}
\CgCg{\Cg{\gamma_1}}{\Cg{\gamma_2}^{s}} = \CgCg{\Cg{\gamma_1}}{\Cg{\gamma_2}}.
\end{equation}
For this, note that since $s \in \HH_1(S)$ we have that $\CgCg{\Cg{\gamma_1}}{\Cg{\gamma_2}^{s}}$ is in the image of the map
\begin{equation}
\label{eqn:mapsin}
\begin{tikzcd}
\HH_1({[\pi_1(S),\pi_1(S)]};\Q) \arrow{r} & \cK_g \subset (\cC_g^{\otimes 2})_{\Delta(\Torelli_{g,1})}
\end{tikzcd}
\end{equation}
taking the homology class of $\zeta \in [\pi_1(S),\pi_1(S)]$ to $\CgCg{\Cg{\gamma_1}}{\Cg{\zeta}}$.  Let
$\Torelli(S)$ denote the Torelli group of $S$.  By extending mapping classes on $S$ to $\Sigma_{g,1}$ by the identity,
we get an inclusion $\Torelli(S) \hookrightarrow \Torelli_{g,1}$.  Since $\Torelli(S)$ fixes $\gamma_1$, the map
\eqref{eqn:mapsin} factors through the $\Torelli(S)$-coinvariants as 
\begin{equation}
\label{eqn:mapsin2}
\begin{tikzcd}
\HH_1({[\pi_1(S),\pi_1(S)]};\Q)_{\Torelli(S)} \arrow{r} & \cK_g \subset (\cC_g^{\otimes 2})_{\Delta(\Torelli_{g,1})}.
\end{tikzcd}
\end{equation}
Lemma \ref{lemma:cg1torellicoinv} implies that\footnote{Though when we proved Lemma \ref{lemma:cg1torellicoinv} we were
working under our standing assumption that $g \geq 4$ (Assumption \ref{assumption:genus}), the proof only requires
$g \geq 3$ and thus applies to $S$.}
\begin{equation}
\label{eqn:commutator1}
\HH_1([\pi_1(S),\pi_1(S)];\Q)_{\Torelli(S)} \cong \wedge^2 \HH_1(S;\Q).
\end{equation}
The free group $\pi_1(S)$ also acts on $[\pi_1(S),\pi_1(S)]$ by conjugation, and it is classical
that
\begin{equation}
\label{eqn:commutator2}
\HH_1([\pi_1(S),\pi_1(S)];\Q)_{\pi_1(S)} \cong \wedge^2 \HH_1(S;\Q).
\end{equation}
See, e.g., \cite[Theorem C]{PutmanCommutator}.  The isomorphisms \eqref{eqn:commutator1} and \eqref{eqn:commutator2} 
give two quotients of $\HH_1([\pi_1(S),\pi_1(S)];\Q)$ that happen to be isomorphic, and thus two different maps
\[\begin{tikzcd}
{[\pi_1(S),\pi_1(S)]} \arrow{r} & \HH_1({[\pi_1(S),\pi_1(S)]};\Q) \arrow[two heads]{r} & \wedge^2 \HH_1(S;\Q).
\end{tikzcd}\]
Examining the proofs of Lemma \ref{lemma:cg1torellicoinv} and \cite[Theorem C]{PutmanCommutator}, we see
that these are actually the same map.
This implies that $\pi_1(S)$-conjugate elements of $[\pi_1(S),\pi_1(S)]$ map to the same
element of $\HH_1([\pi_1(S),\pi_1(S)];\Q)_{\Torelli(S)}$.  Choosing $\sigma \in \pi_1(S)$ with $\osigma = s$,
the images of $\precon{\sigma}{\gamma_2}$ and $\gamma_2$ in $\HH_1([\pi_1(S),\pi_1(S)];\Q)_{\Torelli(S)}$ are
therefore the same.
Mapping this to $(\cC_g^{\otimes 2})_{\Delta(\Torelli_{g,1})}$ via
\eqref{eqn:mapsin2}, we conclude that $\CgCg{\Cg{\gamma_1}}{\Cg{\gamma_2}^{s}}$ and
$\CgCg{\Cg{\gamma_1}}{\Cg{\gamma_2}}$ are the same, as was claimed in \eqref{eqn:cancelx}.

\begin{step}{2}
\label{step:kgexponents2}
The lemma is true for general $\gamma_i$.
\end{step}

Use Lemma \ref{lemma:kggenus1} to write
\[\CgCg{\Cg{\gamma_1}}{\Cg{\gamma_2}} = \sum_{j=1}^n c_j \CgCg{\Cg{\delta_{1,j}}^{k_{1,j}}}{\Cg{\delta_{2,j}}^{k_{2,j}}}\]
where for $1 \leq j \leq n$ we have $c_j \in \Z$ and the following holds:
\begin{itemize}
\item for $i=1,2$, the curve $\delta_{i,j} \in [\pi_g,\pi_g]$ is a simple closed separating
curve that bounds on its right side a genus-$1$ subsurface $T_{i,j}$ and $k_{i,j} \in H_{\Z}$ is arbitrary; and
\item the intersection of $T_{1,j}$ and $T_{2,j}$ is the basepoint.
\end{itemize}
We then have
\begin{equation}
\label{eqn:kgexponents2.1}
\CgCg{\Cg{\gamma_1}^{h_1}}{\Cg{\gamma_2}^{h_2}} = \sum_{j=1}^n c_j \CgCg{\Cg{\delta_{1,j}}^{h_1+k_{1,j}}}{\Cg{\delta_{2,j}}^{h_2+k_{2,j}}}.
\end{equation}
Applying Step \ref{step:kgexponents1} to each term in this sum, we get that
\begin{equation}
\label{eqn:kgexponents2.2}
\sum_{j=1}^n c_j \CgCg{\Cg{\delta_{1,j}}^{h_1+k_{1,j}}}{\Cg{\delta_{2,j}}^{h_2 + k_{2,j}}} = \sum_{j=1}^n c_j \CgCg{\Cg{\delta_{1,j}}^{k_{1,j}}}{\Cg{\delta_{2,j}}^{k_{2,j}}} = \CgCg{\Cg{\gamma_1}}{\Cg{\gamma_2}}.
\end{equation}
Combining \eqref{eqn:kgexponents2.1} and \eqref{eqn:kgexponents2.2}, we conclude that
$\CgCg{\Cg{\gamma_1}^{h_1}}{\Cg{\gamma_2}^{h_2}} = \CgCg{\Cg{\gamma_1}}{\Cg{\gamma_2}}$,
as desired.
\end{proof}

\section{A refined generating set for \texorpdfstring{$\cK_g$}{Kg}}
\label{section:kggen}

Our goal in this section is to construct a refined generating set for $\cK_g$.

\subsection{Symplectic terminology}

A {\em symplectic summand} of $H_{\Z}$ is a subgroup $V$ of $H_{\Z}$ such that $H_{\Z} = V \oplus V^{\perp}$,
where $\perp$ is taken with respect to the algebraic intersection pairing.  A symplectic summand
$V$ of $H_{\Z}$ is isomorphic to $\Z^{2h}$ for an integer $h$ called its {\em genus}.
For a subgroup $W$ of $H_{\Z}$, let $W_{\Q}$ denote the subspace $W \otimes \Q$ of $H = H_{\Z} \otimes \Q$.

\subsection{Generators}

Fix a genus-$1$ symplectic summand $V$ of $H_{\Z}$ and some $\kappa \in \wedge^2 V_{\Q}^{\perp}$.  We construct
elements $\Pres{V,\kappa}$ and $\Pres{\kappa,V}$ of $\cK_g$ in the following way.  By work of Johnson \cite{JohnsonConj},
we can find a simple closed separating curve $\delta \in [\pi_g,\pi_g]$ bounding on its right side a subsurface $T \cong \Sigma_1^1$
with $\HH_1(T) = V$.  Let $S \cong \Sigma_{g-1}^1$ be the subsurface to the left of $\delta$:\\
\Figure{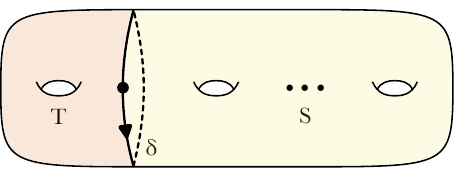}
Regard $\pi_1(S)$ and $\HH_1(S)$ as subgroups of $\pi_g$ and $\HH_1(\Sigma_g)$, so $\HH_1(S) = V^{\perp}$.  Let
\begin{align*}
\phi_L\colon [\pi_1(S),\pi_1(S)] &\longrightarrow \cK_g \subset (\cC_g^{\otimes 2})_{\Delta(\Torelli_{g,1})}, \\
\phi_R\colon [\pi_1(S),\pi_1(S)] &\longrightarrow \cK_g \subset (\cC_g^{\otimes 2})_{\Delta(\Torelli_{g,1})}
\end{align*}
be the maps defined by
\[\phi_L(\eta) = \CgCg{\Cg{\delta}}{\Cg{\eta}} \quad \text{and} \quad \phi_R(\eta) = \CgCg{\Cg{\eta}}{\Cg{\delta}} \quad \text{for $\eta \in [\pi_1(S),\pi_1(S)]$}.\]
Since their targets are $\Q$-vector spaces, these maps factor through $\HH_1([\pi_1(S),\pi_1(S)];\Q)$.  Letting $\pi_1(S)$ act
on $\HH_1([\pi_1(S),\pi_1(S)];\Q)$ via the conjugation action of $\pi_1(S)$ on $[\pi_1(S),\pi_1(S)]$, it follows from
Lemma \ref{lemma:kgexponents} that both induced maps $\HH_1([\pi_1(S),\pi_1(S)];\Q) \rightarrow \cK_g$ are
$\pi_1(S)$-invariant.  They therefore both factor through the coinvariants
\[\HH_1([\pi_1(S),\pi_1(S)];\Q)_{\pi_1(S)} \cong \wedge^2 \HH_1(S;\Q) = \wedge^2 V_{\Q}^{\perp},\]
where the first isomorphism is classical (see, e.g., \cite[Theorem C]{PutmanCommutator}).  Let
\begin{align*}
\ophi_L\colon \wedge^2 V_{\Q}^{\perp} &\longrightarrow \cK_g \subset (\cC_g^{\otimes 2})_{\Delta(\Torelli_{g,1})}, \\
\ophi_R\colon \wedge^2 V_{\Q}^{\perp} &\longrightarrow \cK_g \subset (\cC_g^{\otimes 2})_{\Delta(\Torelli_{g,1})}
\end{align*}
be these two induced maps.  Recalling that $\kappa \in \wedge^2 V_{\Q}^{\perp}$, we define
\[\Pres{\delta,\kappa} = \ophi_L(\kappa) \quad \text{and} \quad \Pres{\kappa,\delta} = \ophi_R(\kappa).\]
We claim this only depends on $V$:

\begin{lemma}
\label{lemma:genwelldefined}
Let $V$ be a genus-$1$ symplectic summand of $H_{\Z}$ and let $\kappa \in \wedge^2 V_{\Q}^{\perp}$.  Let
$\delta_1,\delta_2 \in [\pi_g,\pi_g]$ be simple closed separating curves such that $\delta_i$ bounds
on its right side a subsurface $T_i \cong \Sigma_1^1$ with $\HH_1(T_i) = V$.  Then
$\Pres{\delta_1,\kappa} = \Pres{\delta_2,\kappa}$ and $\Pres{\kappa,\delta_1} = \Pres{\kappa,\delta_2}$.
\end{lemma}
\begin{proof}
By work of Johnson \cite{JohnsonConj}, we can find $f \in \Torelli_{g,1}$ such that $f(\delta_1) = \delta_2$.
Recall that $\Delta\colon \Torelli_{g,1} \rightarrow \Torelli_{g,1} \times \Torelli_{g,1}$ is the diagonal map
and $\cK_g \subset (\cC_g^{\otimes 2})_{\Delta(\Torelli_{g,1})}$.  By construction,
\[\Pres{\delta_2,\kappa} = \Delta(f)(\Pres{\delta_1,\kappa}) \quad \text{and} \quad \Delta(f)(\Pres{\kappa,\delta_1}) = \Pres{\kappa,\delta_2}.\]
Since $\Delta(\Torelli_{g,1})$ acts trivially on $\cK_g \subset (\cC_g^{\otimes 2})_{\Delta(\Torelli_{g,1})}$,
the lemma follows.
\end{proof}

Because of this lemma, we can define
$\Pres{V,\kappa} = \Pres{\delta,\kappa}$
and
$\Pres{\kappa,V} = \Pres{\kappa,\delta}$.

\subsection{Refined generating set}

The following says that the $\Pres{V,\kappa}$ and $\Pres{\kappa,V}$ generate $\cK_g$ and identifies
some relations between them:

\begin{lemma}
\label{lemma:refinegensetkg}
The vector space $\cK_g$ is generated by the $\Pres{V,\kappa}$ and $\Pres{\kappa,V}$ as
$V$ ranges over genus-$1$ symplectic summand of $H_{\Z}$ and $\kappa$ ranges over
elements of $\wedge^2 V_{\Q}^{\perp}$.  Moreover, for a genus-$1$ symplectic summand $V$ and
$\kappa_1,\kappa_2 \in \wedge^2 V_{\Q}^{\perp}$ and $\lambda_1,\lambda_2 \in \Q$ we have
relations
\begin{align*}
\Pres{V,\lambda_1 \kappa_1+\lambda_2 \kappa_2} &= \lambda_1 \Pres{V,\kappa_1} + \lambda_2 \Pres{V,\kappa_2},\\
\Pres{\lambda_1 \kappa_1+\lambda_2 \kappa_2,V} &= \lambda_1 \Pres{\kappa_1,V} + \lambda_2 \Pres{\kappa_2,V}.
\end{align*}
\end{lemma}
\begin{proof}
The elements $\Pres{V,\kappa}$ and $\Pres{\kappa,V}$ generate $\cK_g$ since by
Lemma \ref{lemma:kgexponents} they contain all the generators for $\cK_g$ identified
by Lemma \ref{lemma:kggenus1}.  The indicated relations are all immediate from the
construction of $\Pres{V,\kappa}$ and $\Pres{\kappa,V}$.
\end{proof}

\section{Redundancies among generators for \texorpdfstring{$\cK_g$}{Kg}}
\label{section:kgrel}

The generating set for $\cK_g$ given by Lemma \ref{lemma:refinegensetkg} has
some redundancies.

\subsection{Commutator projection}

Describing these redundancies requires some preliminaries.  
Let $F$ be a free group.  The group $F$ acts on conjugation on $[F,F]$, and it is classical
that the coinvariants of the induced action on $\HH_1([F,F])$ satisfy
\[\HH_1([F,F])_{F} \cong \wedge^2 \HH_1(F).\]
See, e.g., \cite[Theorem C]{PutmanCommutator}.  We have used this isomorphism several times
already.  Let $\rho\colon [F,F] \rightarrow \wedge^2 \HH_1(F)$ be the composition
\[\begin{tikzcd}
{[F,F]} \arrow[hook]{r} & \HH_1({[F,F]}) \arrow[two heads]{r} & \HH_1({[F,F]})_{F} \cong \wedge^2 \HH_1(F).
\end{tikzcd}\]
We will call this the {\em commutator projection} map.  For $z \in F$, let $\oz$ be the image of $z$
in $\HH_1(F)$.  The commutator projection map satisfies
\[\rho([x,y]) = \ox \wedge \oy \quad \text{for all $x,y \in F$}.\]

\subsection{Subsurface intersection form}

Let $W$ be a genus-$h$ symplectic summand of $H_{\Z}$.
Alternating bilinear forms on $W$ can be identified with elements of $\wedge^2 W$.  In particular, the restriction
to $W$ of the algebraic intersection form can be identified with an element $\omega_W$ of $\wedge^2 W \subset \wedge^2 H_{\Z}$. 
If $\{a_1,b_1,\ldots,a_h,b_h\}$ is a symplectic basis for $W$, then
\[\omega_W = a_1 \wedge b_1 + \cdots + a_h \wedge b_h.\]
The importance for us of these elements comes from:

\begin{lemma}
\label{lemma:identifyform}
Let $S \cong \Sigma_h^1$ be a subsurface of $\Sigma_g$ such that the basepoint $\ast$ of $\Sigma_g$ lies on $\partial S$
and let $\rho\colon [\pi_1(S),\pi_1(S)] \longrightarrow \wedge^2 \HH_1(S)$
be the commutator projection map.  Let $\gamma \in [\pi_1(S),\pi_1(S)]$ be a simple 
closed separating curve bounding on its right side a subsurface
$X \cong \Sigma_k^1$ of $S$ with $\partial X \cap \partial S = \{\ast\}$.  Then
$\rho(\gamma) = \omega_{\HH_1(X)}$.
\end{lemma}
\begin{proof}
We can draw $S$ and $\gamma$ and $X$ as follows:\\
\Figure{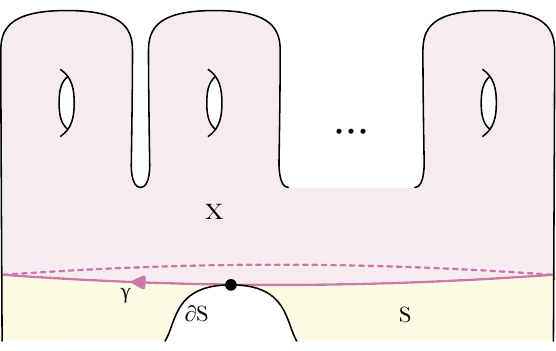}
As in the following figure, we can then find a generating set $\{\alpha_1,\beta_1,\ldots,\alpha_k,\beta_k\}$ for
$\pi_1(X)$ such that $\gamma = [\alpha_1,\beta_1] \cdots [\alpha_k,\beta_k]$:\\
\Figure{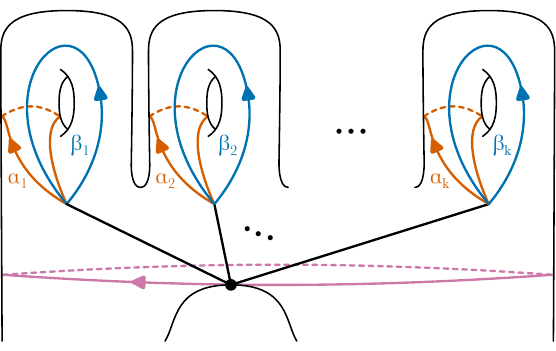}
We then have
$\rho(\gamma) = \rho([\alpha_1,\beta_1])+\cdots+\rho([\alpha_k,\beta_k]) = \oalpha_1 \wedge \obeta_1 + \cdots + \oalpha_k \wedge \obeta_k = \omega_{\HH_1(X)}$.
\end{proof}

\subsection{Redundancy}

With the above preliminaries, the following identifies the redundancies between our generators:

\begin{lemma}
\label{lemma:redundancykg}
In $\cK_g$, we have the following relations:
\begin{itemize}
\item[(a)] For all orthogonal genus-$1$ symplectic summands $V_1$ and $V_2$ of $H_{\Z}$, the relation
$\Pres{V_1,\omega_{V_2}} = \Pres{\omega_{V_1},V_2}$.
\item[(b)] For all genus-$1$ symplectic summands $V$ of $H_{\Z}$, the relation $\Pres{V,\omega_{V^{\perp}}} = \Pres{\omega_{V^{\perp}},V}$.
\end{itemize}
\end{lemma}
\begin{proof}
We start by verifying (a).  Let $V_1$ and $V_2$ be orthogonal genus-$1$ symplectic summands of $H_{\Z}$.  Using work of Johnson \cite{JohnsonConj},
we can find $\gamma_1,\gamma_2 \in \pi_g$ such that:
\begin{itemize}
\item for $i=1,2$, the curve $\gamma_i$ is a simple closed separating curve bounding a genus-$1$ surface $T_i \cong \Sigma_1^1$ on its right side with $\HH_1(T_i) = V_i$; and
\item the intersection of $T_1$ and $T_2$ is the basepoint.
\end{itemize}
See here:\\
\Figure{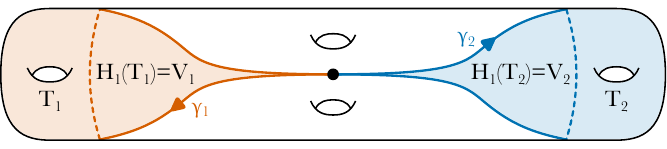}
Using Lemma \ref{lemma:identifyform}, we have
\[\Pres{V_1,\omega_{V_2}} = \CgCg{\Cg{\gamma_1}}{\Cg{\gamma_2}} \quad \text{and} \quad \Pres{\omega_{V_1},V_2} = \CgCg{\Cg{\gamma_1}}{\Cg{\gamma_2}},\]
so $\Pres{V_1,\omega_{V_2}} = \Pres{\omega_{V_1},V_2}$, as claimed in (a).

We next verify (b).  Let $V$ be a genus-$1$ symplectic summand of $H_{\Z}$.  Again using work of Johnson \cite{JohnsonConj},
we can find a simple close separating curve $\gamma \in [\pi_g,\pi_g]$ that bounds a genus-$1$ subsurface $T \cong \Sigma_1^1$ on its
right side with $\HH_1(T) = V$.  Note that $\gamma^{-1}$ can be homotoped to be disjoint from $\gamma$ and bound a subsurface $S \cong \Sigma_{g-1}^1$
on its right side such that $\HH_1(S) = V^{\perp}$ and such that $S$ and $T$ only intersect at the basepoint:\\
\Figure{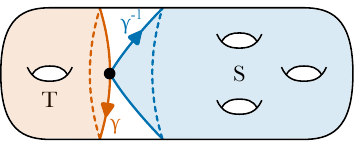}
Using Lemma \ref{lemma:identifyform}, we have
\[\Pres{V,\omega_{V^{\perp}}} = \CgCg{\Cg{\gamma}}{\Cg{\gamma^{-1}}} \quad \text{and} \quad \Pres{\omega_{V^{\perp}},V} = \CgCg{\Cg{\gamma^{-1}}}{\Cg{\gamma}}.\]
It follows that
\[\Pres{V,\omega_{V^{\perp}}} = \CgCg{\Cg{\gamma}}{\Cg{\gamma^{-1}}} = -\CgCg{\Cg{\gamma}}{\Cg{\gamma}} = \CgCg{\Cg{\gamma^{-1}}}{\Cg{\gamma}} = \Pres{\omega_{V^{\perp}},V},\]
as claimed by (b).
\end{proof}

\section{Identifying \texorpdfstring{$\cK_g$}{Kg}}
\label{section:kgalg}

Recall that the algebraization map is the map
\[\fa\colon (\cC_g^{\otimes 2})_{\Delta(\Torelli_{g,1})} \twoheadrightarrow (\cC_g^{\otimes 2})_{\Torelli_{g,1} \times \Torelli_{g,1}} \cong \left((\wedge^2 H)/\Q\right)^{\otimes 2}.\]
See Corollary \ref{corollary:doublecoinv} for this isomorphism.
We close the paper by proving Theorem \ref{theorem:part3theorem}, whose statement
we recall:

\newtheorem*{theorem:part3theorem}{Theorem \ref{theorem:part3theorem}}
\begin{theorem:part3theorem}
The restriction of the algebraization map $\fa$ to $\cK_g$ is an injection.
\end{theorem:part3theorem}
\begin{proof}
For $\kappa \in \wedge^2 H$, let $\okappa$ be the image of $\kappa$ in $(\wedge^2 H)/\Q$.
For a genus-$1$ symplectic summand $V$ of $H_{\Z}$ and $\kappa \in \wedge^2 V_{\Q}^{\perp}$,
it is immediate from Lemma \ref{lemma:identifyform} that
\[\fa(\Pres{V,\kappa}) = \oomega_V \otimes \okappa \quad \text{and} \quad \fa(\Pres{\kappa,V}) = \okappa \otimes \oomega_V.\]
Here are are identifying $\wedge^2 V_{\Q}^{\perp}$ with the corresponding subspace of
$\wedge^2 H$ to allow us to talk about $\okappa \in (\wedge^2 H)/\Q$.

Now define $\fK_g$ to be the vector space with the following presentation:
\begin{itemize}
\item {\bf Generators}.  For all genus-$1$ symplectic summands $V$ of $H_{\Z}$ and all
$\kappa \in \wedge^2 V_{\Q}^{\perp}$, generators $\PresPrime{V,\kappa}$ and $\PresPrime{\kappa,V}$.
\item {\bf Relations}.  The following families of relations:
\begin{itemize}
\item For all genus-$1$ symplectic summands $V$ of $H_{\Z}$ and all $\kappa_1,\kappa_2 \in \wedge^2 V_{\Q}^{\perp}$
and all $\lambda_1,\lambda_2 \in \Q$, the linearity relations
\begin{align*}
\PresPrime{V,\lambda_1 \kappa_1 + \lambda_2 \kappa_2} &= \lambda_1 \PresPrime{V,\kappa_1} + \lambda_2 \PresPrime{V,\kappa_2} \quad \text{and} \\
\PresPrime{\lambda_1 \kappa_1 + \lambda_2 \kappa_2,V} &= \lambda_1 \PresPrime{\kappa_1,V} + \lambda_2 \PresPrime{\kappa_2,V}.
\end{align*}
\item For all orthogonal genus-$1$ symplectic summands $V$ and $W$ of $H_{\Z}$, the relation
\[\PresPrime{V,\omega_W} = \PresPrime{\omega_V,W}.\]
\item For all genus-$1$ symplectic summands $V$ of $H_{\Z}$, the relation
\[\PresPrime{V,\omega_{V^{\perp}}} = \PresPrime{\omega_{V^{\perp}},V}.\]
\end{itemize}
\end{itemize}
Define a map $\pi\colon \fK_g \rightarrow \cK_g$ on generators $\PresPrime{V,\kappa}$ and $\PresPrime{\kappa,V}$
by letting
\[\pi(\PresPrime{V,\kappa}) = \Pres{V,\kappa} \quad \text{and} \quad \pi(\PresPrime{V,\kappa}) = \Pres{V,\kappa}.\]
This makes sense since by Lemmas \ref{lemma:refinegensetkg} and \ref{lemma:redundancykg} it
takes relations to relations.  Moreover, since the image of $\pi$ contains all the generators
of $\fK_g$ identified by Lemma \ref{lemma:refinegensetkg} it follows that $\pi$ is surjective.

The composition $\fa \circ \pi\colon \fK_g \rightarrow \left((\wedge^2 H)/\Q\right)^{\otimes 2}$ satisfies
\[\fa \circ \pi(\PresPrime{V,\kappa}) = \oomega_V \otimes \okappa \quad \text{and} \quad
\fa \circ \pi(\PresPrime{\kappa,V}) = \okappa \otimes \oomega_V.\]
In \cite[Theorem A.6]{MinahanPutmanRepPresentations}, the authors proved that
this map $\fa \circ \pi$ is injective.  The paper \cite{MinahanPutmanRepPresentations} calls
the image of $\fa \circ \pi$ the {\em symmetric kernel}.  It 
is the kernel of a contraction
\[\left((\wedge^2 H)/\Q\right)^{\otimes 2} \rightarrow \Sym^2(H).\]
Since $\pi$ is surjective and $\fa \circ \pi$ is injective, it follows that $\fa$ is injective,\footnote{And
also that $\pi$ is an isomorphism, so $\cK_g$ is also isomorphic to the symmetric kernel.}
as desired.
\end{proof}

\end{document}